\documentclass[reqno, 11pt]{amsart}

\usepackage{fullpage}
\usepackage{srcltx} \usepackage{pdfsync}
\usepackage{calc,amsfonts,amsthm,amscd,epsfig,psfrag,amsmath,amssymb,enumerate,graphicx,mathrsfs}
\usepackage{wrapfig} 
\usepackage{caption}
\newcommand{\intN}{\stackrel{\raisebox{-2pt}[0pt][0pt]{ \scriptsize{\ensuremath{\circ}} }}{N}}

\usepackage[usenames]{color}\usepackage{graphicx}

\usepackage[showlabels,sections,floats,textmath,displaymath]{}

\reversemarginpar
\newlength\fullwidth
\numberwithin{equation}{section}

\DeclareMathSymbol{\leqslant}{\mathalpha}{AMSa}{"36} 
\DeclareMathSymbol{\geqslant}{\mathalpha}{AMSa}{"3E} 
\DeclareMathSymbol{\eset}{\mathalpha}{AMSb}{"3F}     
\renewcommand{\leq}{\;\leqslant\;}                   
\renewcommand{\geq}{\;\geqslant\;}                   

\newcommand{\abs}[1]{ \lvert #1 \rvert}
\newcommand{\norm}[1]{ \lVert #1 \rVert}
\newcommand{\td}[0]{\mathrm{d}}

\newcommand{\Reel}{{\mathrm{I\hspace{-.15em}R}}}

\newcommand{\Complex}{{\mathrm{\hspace{.20em}\mathsf{l}\hspace{-.47em}C}}}
\newcommand{\cst}{\text{cst}}

\newcommand{\rw}{{\rm w}}
\newcommand{\rb}{{\rm b}}
\newcommand{\eps}{\epsilon}

\def\1{\ifmmode {1\hskip -3pt \rm{I}} \else {\hbox {$1\hskip -3pt \rm{I}$}}\fi}

\newcommand{\var}{\operatorname{Var}}

\newcommand{\tmix}{T_{\rm mix}}

\newcommand{\grad}{\nabla}

\newcommand{\g}{\gamma}

\newtheorem{theorem}{Theorem}[section]

\newtheorem{lemma}[theorem]{Lemma}
\newtheorem{proposition}[theorem]{Proposition}
\newtheorem{corollary}[theorem]{Corollary}
\newtheorem{remark}[theorem]{Remark}

\newtheorem{assumption}{Assumption}

\newtheorem{claim}[theorem]{Claim}

\newtheorem{definition}[theorem]{Definition}
\theoremstyle{definition}{
\newtheorem{RG_}[theorem]{Reader's Guide}
}
\def\beginRG{\begin{RG_}\begin{quotation}}
\def\endRG{\end{quotation}\end{RG_}}

\newcommand{\Z}{\mathbb Z}

\newcommand{\cF}{\ensuremath{\mathcal F}}

\newcommand{\cN}{\ensuremath{\mathcal N}}

\newcommand{\cV}{\ensuremath{\mathcal V}}

\newcommand{\bbE}{{\ensuremath{\mathbb E}} }

\newcommand{\bbP}{{\ensuremath{\mathbb P}} }

\newcommand{\bbR}{{\ensuremath{\mathbb R}} }

\newcommand{\bbT}{{\ensuremath{\mathbb T}} }

\newcommand{\bbZ}{{\ensuremath{\mathbb Z}} }

\newcommand{\gep}{\varepsilon}

\newcommand{\gO}{\Omega}

\newcommand{\dd}{\mathrm{d}}

\newcommand{\E}{\mathbb{E}}

\title[Glauber dynamics of perfect matchings]{How quickly can we
  sample a uniform domino tiling of the $2L\times 2L$ square via
  Glauber dynamics?}

\author{Beno\^it Laslier}

\author{Fabio Lucio Toninelli}
\address{\!\!\!\!\!\!\!Universit\'e de Lyon, CNRS and Institut Camille
  Jordan, Universit\'e Lyon 1\newline
43 bd du 11 novembre 1918,
69622 Villeurbanne - France.\newline
\rm {\texttt{laslier@math.univ-lyon1.fr}\newline\rm {\texttt{toninelli@math.univ-lyon1.fr}
}
}
}

\date{}


\begin{document}

\maketitle

\begin{abstract}
  The prototypical problem we study here is the following. Given a
  $2L\times 2L$ square, there are approximately $\exp(4KL^2/\pi )$
  ways to tile it with dominos, i.e. with horizontal or vertical
  $2\times 1$ rectangles, where $K\approx 0.916$ is Catalan's constant
  \cite{refcarre,refcarre2}. A conceptually simple (even if
  computationally not the most efficient) way of sampling uniformly
  one among so many tilings is to introduce a Markov Chain algorithm
  (Glauber dynamics) where, with rate $1$, two adjacent horizontal
  dominos are flipped to vertical dominos, or vice-versa. The unique
  invariant measure is the uniform one and a classical question
  \cite{Wilson,LRS,LPW} is to estimate the time $\tmix$ it takes to
  approach equilibrium (i.e. the running time of the algorithm). In
  \cite{LRS,RT}, \emph{fast mixing} was proven: $\tmix=O(L^C)$ for
  some finite $C$. Here, we go much beyond and show that $c L^2\le
  \tmix\le L^{2+o(1)}$. Our result applies to rather general domain
  shapes (not just the $2L\times 2L$ square), provided that the
  typical height function associated to the tiling is macroscopically
  planar in the large $L$ limit, under the uniform measure (this is
  the case for instance for the Temperley-type boundary conditions
  considered in \cite{Kdom}).  Also, our method extends to some other
  types of tilings of the plane, for instance the tilings associated
  to dimer coverings of the hexagon or square-hexagon lattices.
  \end{abstract}


\section{Introduction}

Uniform random perfect matchings (or dimer coverings) of a
bipartite, infinite, planar periodic graph $G$ (e.g. $\bbZ^2$ or the hexagonal lattice $\mathcal H$) play a crucial role in
statistical mechanics and combinatorics, and a vast literature exists on the subject (cf. for instance the
classical papers \cite{refcarre,refcarre2} and the much more recent
\cite{KOS}).  On one hand, thanks to the bijection between perfect
matchings and discrete height functions (see Section
\ref{sec:height}), they provide natural and exactly solvable models of
random $(2+1)$-dimensional interfaces (which can be thought of as simplified  models for the interface separating two coexisting thermodynamic phases \cite{Spohn}).  On the other hand, thanks to
their conformal invariance and Gaussian Free Field-like fluctuation
properties in the scaling limit \cite{Kdom,Kdom2,KOS}, they belong,  like the Ising
model at $T=T_c$, to the family  of critical two-dimensional systems.

In contrast, the study of stochastic dynamics of perfect matchings is
a much less developed topic. Typically,
one takes a large but finite portion $G'$ of the graph $G$ and defines
a simple Glauber-type Markov chain such that each update locally modifies
the matching within $G'$. The unique equilibrium measure is
the uniform measure over perfect matchings of $G'$. From the point of
view of theoretical computer science \cite{LRS,LPW,Wilson,RT}, the
interesting question is to understand how quickly, as a function of
the size of $G'$, the Markov chain approaches equilibrium (i.e. how
quickly it samples reliably a uniformly chosen random perfect matching of
$G'$). From the point of view of statistical physics, thanks to the
above mentioned bijection between perfect matchings and height functions, this
Markov chain can be seen as a dynamics for a $(2+1)$-dimensional
interface \ and it is of interest to understand how the geometry of
the interface evolves in time. For instance, when $G=\mathcal H$ 
the evolution of the height function exactly coincides with the
zero-temperature heat-bath dynamics for an interface, separating
``$-$'' from ``$+$'' spins, of the three-dimensional nearest-neighbor
Ising model \cite{CMST}.

Until recently, the best  mathematical result on this dynamics issues was that, when $G=\bbZ^2$ or
$G=\mathcal H$, the total-variation mixing  time
$\tmix$ of the Markov chain is at most polynomial in the size of
$G'$. See Section \ref{sec:previous} for a short review. Results of this type are based on simple but effective coupling
arguments that unfortunately have little chance of providing the sharp behavior of
$\tmix$. (A notable exception is the work \cite{Wilson}, where sharp
bounds on $\tmix$ for $G=\mathcal H$ were given, but for a very particular, spatially non-local,
dynamics).

In \cite{CMT}, instead, in the case where $G=\mathcal H$ it was proven
that, under a certain condition on the shape of the finite region $G'$
(``almost planar boundary condition'' assumption, see below), $\tmix$
behaves like $L^2$, up to logarithmic corrections, if $L$ is the
diameter of $G'$.  As we briefly explain in Section \ref{sec:novelty},
going beyond the case of the hexagonal lattice $\mathcal H$ is a
mathematical challenge, since certain exact identities that hold there
do not survive on more general graphs.  In this work we prove that,
when $G=\bbZ^2$ (and in a few other cases, see below), again under the
``planar boundary condition'' assumption, $\tmix= L^{2+o(1)}$. We
present this result, still informally but with a bit more of detail,
in the next section.

\smallskip

The major 
improvement with respect to \cite{LRS,Wilson} is that both \cite{CMT}
and the present work
use the intuition that the height function should evolve,
on a diffusive time-space scale, according to  a deterministic,
anisotropic mean-curvature
type evolution \cite{Spohn}. More precisely, call $h_t(X,Y)$ the
height function at time $t$, with $(X,Y)$
a bi-dimensional space coordinate on the lattice $G$. Then, one expects that under diffusive
scaling 
(i.e. setting $\tau=t/L^2 $,  $(x,y)=(X,Y)/L$,
$\phi_\tau(x,y)=L^{-1} h_{\tau L^2}(x L,y L)$ and letting $L\to\infty$) the limiting deterministic evolution  of
the height function $\phi$ should
be of the type
\begin{eqnarray}
  \label{eq:5ter}
  \frac{d}{d\tau}\phi= \mu(\nabla\phi)\mathcal L\phi. 
\end{eqnarray}
Here,  $\mu(\nabla\phi)$ is a positive, slope-dependent ``mobility
coefficient'' while 
 $\mathcal L$, directly related to the first variation of the surface
 energy functional, is  a non-linear elliptic operator of the type
 \[
\mathcal
L\phi=a_{11}(\nabla\phi)\partial^2_x\phi+2a_{12}(\nabla\phi)\partial^2_{xy}\phi+a_{22}(\nabla\phi)\partial^2_y\phi
\]
where the matrix ${\bf a}=\{a_{ij}(\nabla\phi)\}_{i,j=1,2}$ (with
$a_{21}=a_{12}$) is positive-definite.
See
\cite{Spohn} for a  discussion of these issues. Remark
that $\mathcal L\phi$ is a linear combination (with slope-dependent
coefficients) of the principal curvatures of the interface, hence the
name ``anisotropic mean-curvature evolution''.
By the way, such intuition suggests the precise
scaling $\tmix\sim const\times L^2\log L$.  Indeed, for
$\tau\to\infty$ the solution of \eqref{eq:5ter} approaches a ``limit
shape'' $\bar\phi$ satisfying $\mathcal L\bar\phi=0$, and one can
consider that equilibrium is reached when
\begin{eqnarray}
  \label{eq:1}
\|\phi_\tau-\bar\phi\|_\infty\approx (\log L)/L  
\end{eqnarray}
(the typical
equilibrium height fluctuations before space rescaling being expected
to be $O(\log L)$, see Remark \ref{rem:refined}). 
Assume for
simplicity that the matrix ${\bf a}$ is the identity and $\mu(\cdot)$ is
constant: then, \eqref{eq:5ter} is just the heat
equation and \eqref{eq:1} is satisfied as soon as $\tau$ is a suitable
constant times $\log L$, i.e. $t$ is some constant times $L^2\log L$.

\subsection{Informal presentation of the main result}\label{sec:plan_article}

A domino tiling of the plane is a covering of $\bbR^2$ with $2\times
1$ non-overlapping vertical or horizontal rectangles (dominos), with vertices sitting at
points of $(\bbZ^2)^*=\bbZ^2+(1/2,1/2)$. Domino tilings are in one-to-one correspondence with
perfect matchings (or simply ``matchings'' in the following) of $G=
\bbZ^2$, i.e. subsets of edges (called dimers) of $\bbZ^2$ such that each vertex is
contained in exactly one dimer (to see the correspondence, just draw a
segment of unit length inside each domino, parallel to its longer
side, with endpoints on $\bbZ^2$). Similarly, tilings of a finite portion $P$ of the plane correspond
to matchings of a finite subset $G'$ of $\bbZ^2$. From now on, we will abandon the tiling language and adopt the matching one. Typically, if the
set $G'$ does admit matchings (an obvious necessary condition is that its
cardinality is even) and its area  is large, the number $Z(G')$ of
matchings grows like the exponential of a constant times its
area. This is for instance the case when $G'=\{1,\dots,2L\}\times \{1,\dots,2L\}$, in
which case $(1/L^2)\log Z(G')\sim 4K/\pi $, where $K\approx 0.916$ is Catalan's constant
  \cite{refcarre,refcarre2}. 

A natural way to uniformly sample one among so many matchings (even if
  computationally not the most efficient, see Section \ref{moreffi}) is
to run a Markov chain where, with unit rate, two vertical
dimers belonging to the same square face of
$\bbZ^2\cap G'$ are flipped to vertical, or vice-versa. The unique stationary
(and reversible) measure is the uniform measure over all matchings
of $G'$ and a classical question
in theoretical computer science \cite{LRS,LPW,Wilson} is to evaluate
how quickly the Markov chain reaches equilibrium, as a function of the
diameter (call it $L$) of $G'$. This is measured for instance via the
so-called total-variation mixing
time $\tmix$, defined as the first time $t$ such that, uniformly in the
initial condition, the law of the chain at time $t$ is within
variation distance $1/(2e)$ from equilibrium
(see Section \ref{sec:dynamics} for a
definition in formulas).

In the present work we prove that $\tmix=L^{2+o(1)}$, under a
non-trivial restriction (``almost-planar boundary height'' condition)
on the shape of the region $G'$, that we briefly introduce now.  The
lattice $\bbZ^2$ being a bipartite graph, it is possible to associate
in a canonical way (see Section \ref{sec:height}) a discrete height
function (defined on faces of $G'$) to each matching of $G'$. The
height along the boundary $\partial G'$ of $G'$ is instead independent
of the matching and depends only on the shape of $G'$. We say
that \emph{the boundary height of $G'$ is ``almost planar''} if the
graph of the height function, restricted to $\partial G'$, is within
distance of order $1$ from some plane of $\bbR^3$.  In this case, for
$L$ large the height function of a typical matching of $G'$ (sampled
from the uniform measure) is macroscopically planar not only along
$\partial G'$ but also in the interior of $G'$ (see Theorem
\ref{th:stimeflutt}).

The almost-planar boundary height
hypothesis  is verified for instance when $G'$
is the $2L\times 2L$ square as above. More general domain shapes that
verify this hypothesis are introduced in \cite{Kdom} (``Temperley
boundary conditions'') and in that case the height function fluctuations are  proven
to converge to the Gaussian Free Field \cite{Kdom,Kdom2}.
 
Our main result can be informally stated as follows (see Sections \ref{sec:height} and \ref{sec:dynamics} for a precise statement of the hypothesis and of the result):
\begin{theorem}
\label{th:informal}
If the diameter of  $G'$ is $L$ and the boundary height
is almost planar then, as $L$ goes to infinity, 
   \begin{eqnarray}
     \label{eq:informale}
c L^2 \leq \tmix \leq L^{2+o(1)} . 
   \end{eqnarray}
The result holds also when $\bbZ^2$ is replaced by the hexagon or square-hexagon lattices of Fig. \ref{graphes_et_polygones_newton}. 
\end{theorem}
Based on the ``mean curvature motion'' heuristics mentioned above, we
conjecture the true behavior to be $\tmix\sim const\times L^2\log L$
for reasonably regular domains $G'$.

As we explain in Section \ref{sec:entro},  there are good
reasons why we cannot consider general bipartite periodic planar
graphs (for instance, why our method necessarily fails for the
square-octagon graph of Fig. \ref{graphes_et_polygones_newton}). This
is related to the existence for such graphs of so-called ``gaseous
phases'' in their phase diagram \cite{KOS}. In a gaseous phase, the
height function looks qualitatively like a $(2+1)$-dimensional low
temperature Solid-on-Solid interface (the interface is rigid, height
fluctuations have bounded variance and their spatial correlations
decay exponentially. In the scaling limit, the interface does not
behave like the Gaussian Free Field in this case).

\subsubsection{Review of previous results}

\label{sec:previous}

The first mathematical result we are aware of on this problem is in
\cite{LRS}, where dynamics of perfect matchings of either $\mathbb
Z^2$ or $\mathcal H$ are studied.  There, the authors introduced and
analyzed a non-local Markov dynamics whose updates can involve an
unbounded number of dimer rotations (cf. Section \ref{sec:fd}). Via a
coupling argument, they managed to prove that the mixing time
$\tmix^\star$ of such dynamics is $\tmix^\star\le const\times L^6$ (no
lower bound was given). Subsequently, in the case of the hexagonal
lattice $\mathcal H$ this result was sharpened to $c_1 L^2\log L\le
\tmix^\star\le c_2 L^2\log L$ by Wilson \cite{Wilson}.  Via the
application of comparison arguments for Markov chains, these upper bounds
for $\tmix^\star$ imply polynomial upper bounds on the mixing time
$\tmix$ of the local Glauber dynamics: indeed, it was deduced in
\cite{RT} that $\tmix\le L^C$ for some finite $C$.  In this case, in
the theoretical computer science language, the Markov chain is said to
be ``rapidly mixing'' (slow mixing would correspond to $\tmix$ being
super-polynomial in $L$).  In the particular case of the hexagonal
lattice, using results of \cite{Wilson} on the spectral gap of the
non-local dynamics and the comparison arguments of \cite{RT}, one
obtains $\tmix\le const\times L^6$.

The results we mentioned so far do not require any restriction on the
boundary height. If instead one assumes the boundary height to be
almost-planar,  for the hexagonal lattice the upper bound in
\eqref{eq:informale} was proven in \cite{CMT} (in the stronger form
$\tmix=O(L^2(\log L)^{12})$), while the best known lower bound was
$\tmix\ge L^2/(c\log L)$ (based on \cite{CMST}).  We are not aware of
previous results for the square-hexagon lattice.

\begin{remark}
  The main reason why in Theorem \ref{th:informal} we require the
  boundary conditions to be almost-planar is that in this case the
  height fluctuations at equilibrium (i.e. under the uniform measure)
  are well-controlled, see Theorems \ref{th:clt} and
  \ref{th:stimeflutt}. In the case of general boundary conditions,
  only partial results are known (e.g. \cite{Petrov,K_non_planaire}) and these are not sufficient
  to implement our scheme. We emphasize that instead the $\tmix=O(L^C)$ result
  of \cite{LRS} does not require boundary conditions to be
  almost-planar.
\end{remark}

\subsubsection{Alternative ways of quickly sampling random perfect
  matchings}
\label{moreffi}
There are several known algorithms that sample uniform perfect
matchings.  The main reason why we focus on the Glauber algorithm
is its above-mentioned connection with the three-dimensional zero temperature Ising
dynamics and with interface motion in non-equilibrium statistical
mechanics. However, there are more efficient algorithms in terms of
 running time.

Let us first of all observe that in algorithmic terms, our Theorem
\ref{th:informal} says that the running time of the Glauber dynamics with
almost-planar boundary conditions is $L^{4+o(1)}$, i.e. it requires at most that many updates to approach the uniform measure (our Markov chain
was defined in continuous time, so that there are of order $L^2$
elementary updates per unit time). There are at least two families of  more
efficient methods to sample random perfect matchings.

In \cite{MuchaSankowski} (see also \cite{WilsonDet}) it is proven that
one can sample uniform perfect matchings of planar graphs $G'$ in a
time $O(L^\omega)$, where $\omega \leq 2.376$ (matrix multiplication exponent) is the exponent of the
running time of the best known algorithm to multiply two $L\times L$
matrices.  In   \cite{MuchaSankowski,WilsonDet}  there is  essentially no restriction on the domain
$G'$ (i.e. no assumption on the boundary height), apart from obviously requiring
that the number of vertices is of order $L^2$. This algorithm  can even be used to
find a maximum matching for domains that do not admit perfect
matchings. The starting point is a classical formula, the analog of the one in 
Theorem \ref{thm:proba_evn_fini} but for finite domains, that expresses the probability of local dimer events in terms of minors of the adjacency matrix of the graph.
In the proof of the  $O(L^\omega)$ bound then \cite{MuchaSankowski} cleverly uses the planarity of the graph and the fact that the adjacency matrix is sparse, to efficiently compute the minors.
 
The second class of algorithms is based on the mapping between perfect
matchings of $G'$ and spanning trees of a related graph (T-graph) that
has approximately the same size \cite{KPWdimerTree,KenTree}. Then one
can sample a spanning tree using algorithms based on random walks
\cite{Wilsonalgo}, whose running time is expressed in terms of the
mean hitting time of the random walk.
For reasonable domains (boundary heights)
one can deduce a
$O(L^2\log L)$ bound on the algorithm running time. In the general
case the same bound should still hold but it seems delicate to
precisely estimate the mean hitting time in complete generality.


\subsection{Sketch of the proof and novelty}
\label{sec:novelty}
Here we briefly sketch how the proof of Theorem \ref{th:informal} works,
and we point out the main novelties, especially with respect to 
\cite{LRS,Wilson,CMT}.

The idea of \cite{CMT} (see also \cite{CMST}) is to break the proof of the upper bound $\tmix\le L^{2+o(1)
}$  into two steps:
\begin{enumerate}[(i)]
\item first prove that $\tmix\le c(\epsilon) L^{2+\epsilon
}$ when the height function is constrained for all times between a
``floor'' and a ``ceiling'' that are at small mutual distance, say
$L^{\epsilon/10}$. Here $\epsilon$ is an arbitrarily small, positive, $L$-independent constant;
\item then, via an iterative procedure that mimics the mean curvature motion that should emerge in the diffusive limit, deduce $\tmix\le c(\epsilon) L^{2+\epsilon
}$ for the unconstrained dynamics.
\end{enumerate}

While this general scheme is robust and will be employed also here,
step (i) is very much model-dependent. In particular, in
\cite{CMST,CMT} for the hexagonal graph $\mathcal H$ its
implementation was based on the crucial observation (by D. Wilson
\cite{Wilson}) that, for the  non-local  dynamics
introduced in \cite{LRS} (cf. Section \ref{sec:fd}), one can write
explicitly an eigenfunction of the generator, and the evolution of the
height function is controlled by the discrete heat equation. As we
explain in Remark \ref{rem:calore}, this fact fails for graphs other
than $\mathcal H$ and it has to be replaced by a more robust argument.

The ``more robust argument'' starts from the observation that, under
the non-local dynamics, the mutual volume $V_t$ between two evolving
height functions, that can be seen as an integer-valued random walk,
is  on average non-increasing with time $t$. This was already realized in
\cite{LRS}, but without any further input this only implies a polynomial
upper bound $O(L^{4+\epsilon})$ on the mixing time $\tmix^\star$ of the non-local
dynamics. The argument is as follows. The maximal volume
between two configurations in a region of diameter $L$, constrained between floor and ceiling at mutual distance $L^{\epsilon/10}$, is of
order $L^{2+\epsilon/10}$. The random walk $V_t$ has non-positive
drift and it is not hard to see that 
the variance increase
$\lim_{\delta\searrow0}\frac1\delta\bbE[(V_{t+\delta}-V_t)^2|\mathcal F_t]$  is bounded away from
zero as long as $V_t\ne 0$, where $\mathcal F_t$ is the sigma-algebra
generated by the non-local dynamics up to time $t$.  A simple
martingale argument  (see Lemma \ref{lemme:temps_martingale}) implies
then that $V_t$ 
will hit the value $0$ in a time  of order
$(L^{2+\epsilon/10})^2\le L^{4+\epsilon}$. When the volume is zero,
the two configurations have coalesced and  a simple coupling argument
allows us to conclude that $\tmix^\star\le L^{4+\epsilon}$.  The new input
we provide for the proof of point (i) (see Section \ref{sec:mart}) is
that $\lim_{\delta\searrow0}\frac1\delta\bbE[(V_{t+\delta}-V_t)^2|\mathcal F_t]$ can be \emph{lower bounded}
essentially by $V_t$ itself: then, an iterative application of 
Lemma \ref{lemme:temps_martingale} allows to
conclude that the coalescence time for the non-local dynamics is of
the (essentially optimal) order
$L^{2+\epsilon/2}$ and not $L^{4+\epsilon}$. Via a  comparison
argument that relates the mixing times for the local and non-local dynamics (Proposition \ref{prop:ts}) one finally deduces
$\tmix=O(L^{2+\epsilon})$ for the \emph{local} dynamics (always
constrained between ``floor'' and ``ceiling'' at distance $L^{\eps/10}$).

To prove the bounds on drift and variance of $V_t$, we introduce a
mapping between perfect matchings and configurations of what we call a
``bead model''. This mapping turns out to be convenient in that it makes the proofs visually clear. 
In particular, the definition of the non-local dynamics looks somewhat more natural in this language than in the ``non-intersecting-path'' language 
\cite{LRS}.

A last comment concerns the mixing time lower bound in Theorem
\ref{th:informal}, which is better (by a factor $\log L$) than the
lower bound $\tmix\ge C\, L^2/\log L$ found in \cite{CMT} and based on an idea developed in \cite{CMST}. First of
all, the proof of \cite{CMST} would not extend for instance to
$G=\bbZ^2$, again because it is based on Wilson's eigenfunction
argument that fails there. Moreover, even for $G=\mathcal H$ where
Wilson's argument does work, removing the $\log L$ in the denominator
involves a genuinely new idea, see Section \ref{sec:lower_bound}: one
needs to prove that the drift of the volume $V_t$ under the non-local dynamics, which as we
mentioned is non-positive, is not smaller than the size of the boundary of $G'$,
times some negative constant.

\subsection{Organization of the paper} All the definitions and results
the reader needs about perfect matchings, height functions and
translation-invariant infinite measures of a given slope are in Section
\ref{sec:back}
 (results are given for general
periodic bipartite planar graphs and not just for the square, hexagon and
square-hexagon graphs). The dynamics is precisely defined in Section
\ref{sec:dynamics} and its monotonicity properties are discussed in
Section \ref{sec:monotonie}.  In Section \ref{sec:billes} we map
height functions into the configurations of  a ``bead
model''. In Section \ref{sec:dyn_billes} we rewrite the dynamics in
terms of beads and we introduce two auxiliary, spatially non-local, dynamics, that are
essential in proving the mixing time estimates of Theorem
\ref{th:informal}: the mixing time upper bound is proven in Section
\ref{sec:upper_bound} and the lower bound in Section
\ref{sec:lower_bound}.

\section{Some background on perfect matchings}

\label{sec:back}
\subsection{Dimer coverings, height functions and uniform measures}
\label{sec:dimers}

We follow the notations of \cite{KOS}.  Let $G=(V,E)$ be an infinite,
$\Z^2$-periodic, bipartite planar graph. ``Bipartite'' means that its
vertices can be colored black or white in such a way that white vertices
have only black neighbors and vice-versa. $\Z^2$-Periodicity means
that $G$ can be embedded in the plane in such a way that $\Z^2$ acts
as a color-preserving isomorphism.  The dual graph of $G$, whose
vertices are the faces $f$ of $G$, is denoted $G^*$.


We let $G_1$ (the fundamental domain) denote $G/\Z^2$, which is a
finite and periodic bipartite graph, embedded on the two-dimensional
torus.  See Fig. \ref{graphes_et_polygones_newton} for some
classical examples (the square, hexagon, square-octagon and
square-hexagon lattice) together with their fundamental domains.

\begin{figure}[tb]
   \centering
   \includegraphics[width = 5cm]{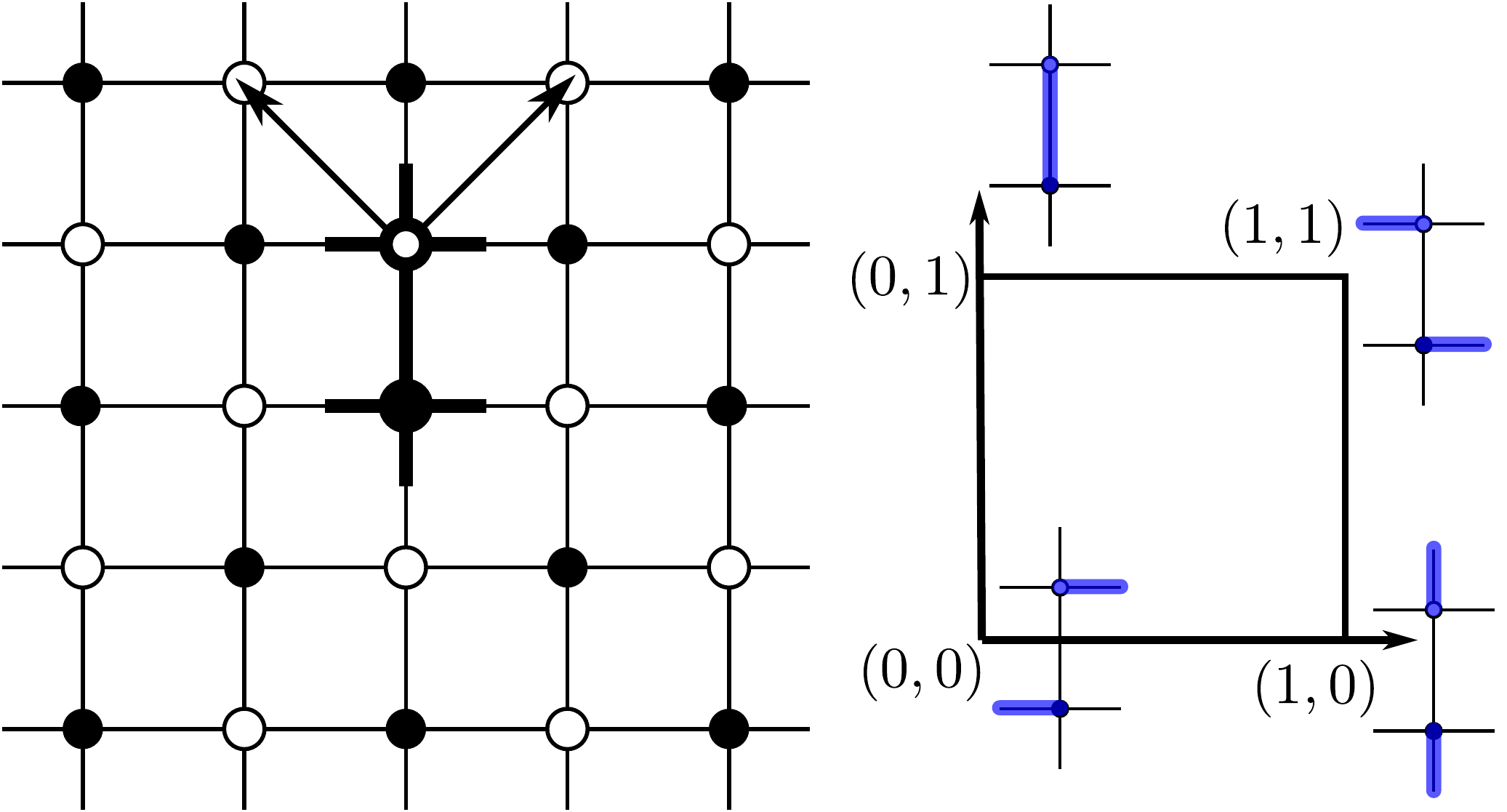}
\quad
\includegraphics[width=5cm]{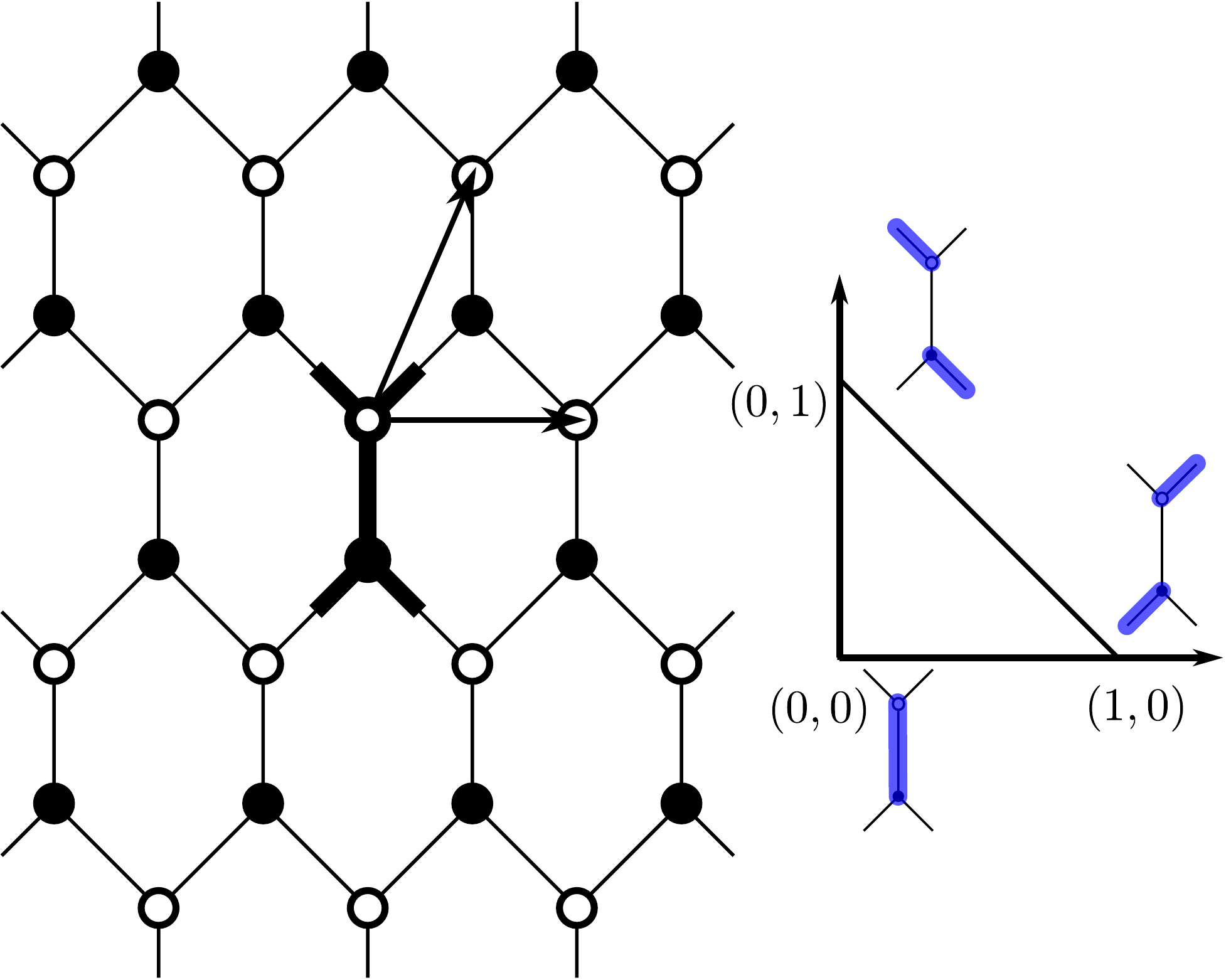}
 \\
    \medskip
    \includegraphics[width=5cm]{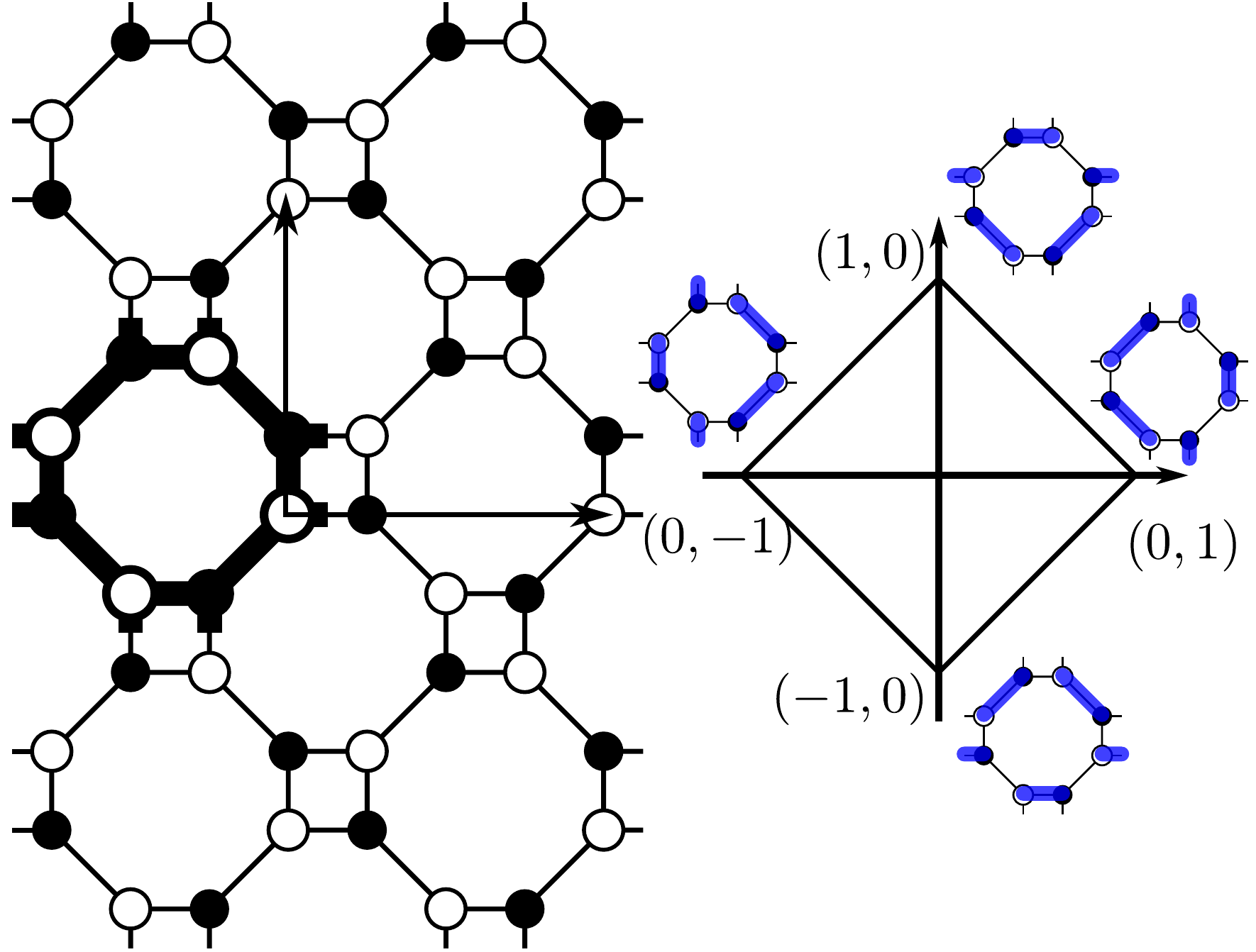}
\quad
    \includegraphics[width=5cm]{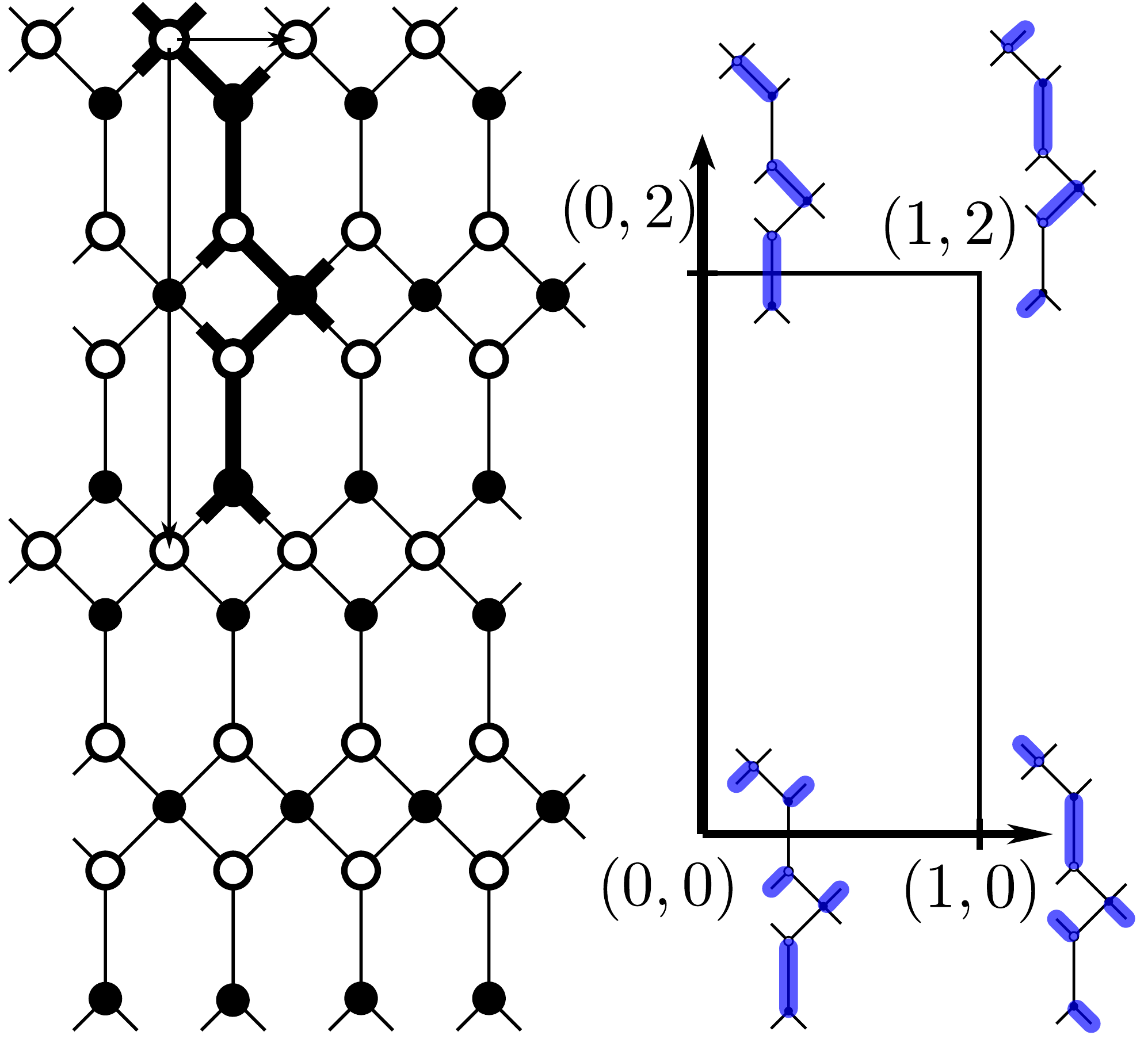}
    \caption{Some examples of $\mathbb{Z}^2$-periodic bipartite graphs
(square, hexagon, square-octagon, square-hexagon)
      and their Newton polygons (cf. Theorem \ref{th:Newton}).  The fundamental domain is indicated with
      thicker lines in the graph while the action of $\Z^2$ is
      represented by two arrows. Near each vertex of the Newton
      polygon is indicated the associated matching of the fundamental
      domain. \label{graphes_et_polygones_newton}}
\end{figure}

Note that there is a certain degree of arbitrariness in the embedding of
$G$ in the plane and as a consequence a certain arbitrariness in the
choice of the fundamental domain. For instance, in Fig.
\ref{graphes_et_polygones_newton} the fundamental domain of $\Z^2$
contains two sites, but with a different embedding it could contain
the four sites around a face (in this case the two axes of
$\bbZ^2$ would be horizontal and vertical). In general, it is
convenient to work with the smallest possible fundamental domain, as in Fig.
\ref{graphes_et_polygones_newton}.

A perfect matching of $G$ is
a subset of edges, $M\subset E$, such that each vertex of $G$ is
contained in one and exactly one edge in $M$.  It is known that $G$
admits a matching (which is implicitly assumed from now on) if and
only if $G_1$ does, and Fig. \ref{graphes_et_polygones_newton}
shows that the fundamental domains of
the four graphs we mentioned
do admit several matchings. We denote $\Omega$ the set of matchings of $G$.

\begin{assumption}\label{hyp_pas_deterministe}
To avoid trivialities, we will assume that for every edge $e$ of $G$
there exists $M\in\Omega$ such that $e\in M$ and $M'\in\Omega$ such that $e\notin M'$ (one can easily construct pathological examples where this fails, but the edges in question can be simply removed
and the matching problem is unchanged).
\end{assumption}

We will often refer to paths on the dual graph $G^*$:
\begin{definition}
\label{def:periodico}
A path $\gamma$ on $G^*$ is a possibly infinite sequence
$(\dots,f_{-1},f_0,f_1,\dots)$ of faces of $G$, such that $f_i$ is a
neighbor of $f_{i+1}$. An infinite path $\gamma$ is called periodic if
there exists a finite path $\gamma_0$ and $v\in \bbZ^2$ such that $\gamma$ is 
the concatenation of $ \{T_{n v}\gamma_0\}_{n\in\bbZ}$, with $T_v$ the translation
by $v$.
\end{definition}

\subsubsection{Height function and uniform measure}
\label{sec:height}
A flux is a function on the oriented edges of $G$, which is
antisymmetric under the change of orientation of the edges.  To each
$M\in\Omega$ is associated a flux $\omega_M$: edges contained in $M$ carry
unit flux, oriented from the white to the black vertex. Edges not
contained in $M$ carry zero flux.  Note that the divergence of
$\omega_M$ is $1$ at white vertices and $-1$ at black vertices.

Fix now a reference matching $M_0$ (typically, a $\bbZ^2$-periodic matching, but the following definition would work for any flux of divergence $+1/-1$ at white/black
vertices).
 $M_0$ allows to associate to $M$ a height function
$h_M$ on $G^*$, as follows. Fix some face $f_0\in G^*$ (``the origin'')
and set $h_M(f_0)$ to some value, say $0$. For every $f\ne f_0$, let $\g$ be a path on $G^*$
starting at $f_0$ and ending at $f$. Then, $h_M(f)-h_M(f_0)$ is the total
flux of $\omega_M-\omega_{M_0}$ (say from right to left) across $\g$.
Note that $h_M$ does not depend on the choice of the path (because
$\omega_M-\omega_{M_0}$ has zero divergence) and that the height
difference between two matchings $M,M'$ is independent of the choice
of the reference matching $M_0$.

In the following, the set of perfect matchings $\Omega$ will denote
equivalently the set of all admissible height functions (it is
understood that $M_0$ and $f_0$ are fixed). For lightness of notation,
we will often write $h$ instead of $h_M$.

\begin{definition}
\label{def:G'}
Let $U$ be a simply connected open subset set of $[-1/2,1/2]^2$, $L > 0$ and $U_L=L\,
U$. We let $G'$ be the finite subset of $G$ obtained by keeping all
the vertices and edges belonging to faces which are entirely contained in $U_L$.
\end{definition}

Given $m\in \Omega$ (called the ``boundary condition'', with
corresponding height function $h_m$) and $G'=(E',V')$ as in Definition
\ref{def:G'}, let
\begin{eqnarray}
  \label{eq:Omg}
  \Omega_{m,G'}=\left\{M\in \Omega:M|_{G\setminus G'}=m|_{G\setminus G'}\right\}
\end{eqnarray}
be the finite collection of matchings that coincide with $m$ outside
of $G'$.  Equivalently, we can identify $ \Omega_{m,G'}$ as the set of
height functions that coincide with $h_m$ except on the faces
of $G'$.  We will implicitly assume
(without loss of generality) that the reference face $f_0$ is
\emph{not} one of the faces of $G'$.  Clearly, $\Omega_{m,G'}$ is
non-empty (it includes at least $m$) and we will let $\pi_{m,G'}$
denote the uniform measure over $\Omega_{m,G'}$.

\subsection{Pure phases}\label{sec:phase_pure}

In this section we review known results about measures on the infinite
graph $G$ whose typical height functions are close to a plane.  First
we will identify the set of ``natural'' measures of fixed average
slope, then give a classification into three ``phases'' with very
different correlation properties, and finally give their
``microscopic'' behavior, i.e. the probabilities of events depending
on a finite subset of edges. Most results come from \cite{KOS}.


\subsubsection{Ergodic Gibbs measure of fixed slope}

Fix a reference matching $M_0$ assumed to be $\bbZ^2$-periodic.  A
measure $\mu$ on $\gO$ is said to have slope $(s,t)$ if its expected
height function is a linear function, with slope $(s,t)$: for all faces $f$, if
$f'$ denotes the translate of $f$ by $(x,y)\in \bbZ^2$, then $\mu[
h(f')-h(f)] = sx+ty$. $\mu$ is said to be a Gibbs measure if its
conditional distributions on finite sub-graphs are uniform, $\mu( \,
\cdot \, |M\in \gO_{m,G'}) = \pi_{m,G'}(\cdot)$ (\emph{DLR property}). It is ergodic if it
is not a linear combination of other Gibbs measures. Ergodic Gibbs
measures of fixed slope can be thought of as the natural uniform
measures on matchings of $G$ conditioned on their average
slopes. The following theorem due to Sheffield \cite{Sheffield}
classifies all of them :

\begin{theorem}
\label{th:Newton}
  There exists a closed, convex polygon $N$ in $\bbR^2$ such that, for all
  $(s,t)$ in its interior $\intN$, there exists a unique ergodic Gibbs
  measure $\mu_{s,t}$ of slope $(s,t)$. The vertices of $N$ are
  determined by the slopes of some $\bbZ^2$-periodic matchings of $G$ (i.e.  matchings of the fundamental domain) and thus are integer
  points. For $(s,t)\in \partial N:={N} \setminus \intN$, there exists an ergodic
  Gibbs measure but it may not be unique.
\end{theorem}
$N$ is called the Newton polygon, see Fig. \ref{graphes_et_polygones_newton}. 

\subsubsection{Phase classification}
\label{sec:phaseclass}

As proved in \cite{KOS}, ergodic Gibbs measures come in three possible phases:
solid, liquid and gas, depending on the position of $(s,t)$ in
$N$.

\begin{itemize}
\item Solid phases correspond to slopes in $\partial N$. For any side $\ell$ of $\partial N$, there exists at
  least one infinite periodic path $\gamma$ on $G^*$ (cf. Definition
\ref{def:periodico}) such that the
  configuration of the edges crossed by $\gamma$ (or by any of its
  translates) is deterministic and is the same for all measures with
  slope $(s,t)\in\ell$. The path $\gamma$ is said to be frozen.

The asymptotic direction of $\gamma$ is determined as follows. All the
planes with slope in $\ell$ and containing the origin of $\bbR^3$
intersect in a straight line. The direction of this line, when
projected on the $(x,y)$ plane, is the direction of $\gamma$.

  At a vertex of the Newton polygon, which is the intersection of two
  sides of $\partial N$,  there are two
  families of frozen paths with different directions, which form so to
  speak a grid on $G^*$. The components of the complement
  of the frozen paths are finite sets of faces. Heights are clearly
  independent in two distinct components, and the fluctuations of the
  height difference between two faces $f_1,f_2$ are bounded
  deterministically and uniformly in the distance between them.

\item Liquid phases correspond to generic points of $\intN$. In these
  phases, heights fluctuations behave like a Gaussian free field in
  the plane. In particular the variance of the height difference
  between $f_1 $ and $f_2$ grows like $1/\pi$ times the logarithm of
  the distance, while edge correlations decay slowly (as the inverse
  of the square of the distance). Liquid phases are discussed in finer
  detail in the next section.

\item Gaseous phases have exponentially decreasing edge correlations;
  the height difference fluctuations are not deterministically
  bounded, but their variance is bounded, uniformly with the distance
  of the faces.  Gaseous phases may (but do not necessarily) occur
  when the slope $(s,t)$ is an integer point in $\intN$. The condition
  for the occurrence of a gaseous phase at an integer slope $(s,t)\in\intN$ is
  discussed in Section \ref{sec:Asy}.

\end{itemize}

In the example of Fig. \ref{graphes_et_polygones_newton}, only the
square-octagon graph has a gaseous phase which has slope $(0,0)$. 

\subsubsection{Edge probabilities}

When $(s,t)\in\intN$ (i.e. for liquid and gaseous phases) there is an
explicit expression of edge probabilities under $\mu_{s,t}$.
\begin{theorem}\label{thm:proba_evn_fini}\cite{KOS}
  Fix $(s,t) \in \intN$. There exists an infinite periodic matrix
  $K_{s,t}=\{K_{s,t}(\rw,\rb)\}_{\rw,\rb} $ with $\rb$ (resp. $\rw$) ranging on
  black (resp. white) vertices of $G$ and an infinite periodic matrix
  $K_{s,t}^{-1}=\{K^{-1}_{s,t}(\rb,\rw)\}_{\rw,\rb} $ satisfying $K_{s,t}K^{-1}_{s,t}=Id$ such
  that, for any finite subset $ \{ e_1=(\rw_1,\rb_1), \ldots,
  e_l=(\rw_l,\rb_l) \}$ of edges of $G$, the $\mu_{s,t}$-probability
  of seeing all of them occupied is:
\[
  \mu_{s,t}\bigl(e_1\in M,\dots,e_l\in M\bigr)= \Biggl( \prod_{j=1}^l K_{s,t}(\rw_j,\rb_j) \Biggr) \det\bigl(K_{s,t}^{-1}(\rb_k,\rw_i) \bigr)_{1\leq i,k \leq l}.
  \]
\end{theorem}

$K_{s,t}$ is called a Kasteleyn matrix. It is a weighted and signed
version of the adjacency matrix, so in particular $K_{s,t}(\rw,\rb)$
can be non-zero only if $(\rb,\rw)$ is an edge of $G$. The signs
(which are independent of the slope $(s,t)$) are chosen so that their
product around any face $f$ is $-1$ if $f$ has $0 \mod 4$ sides and
$+1$ if it has $2\mod 4$
sides
. We will not need to specify the explicit choice of signs, see
\cite{KOS}. 
Periodicity means that $K_{s,t}(\rw+(x,y),\rb+(x,y))=K_{s,t}(\rw,\rb)$ for
every $(x,y)\in \bbZ^2$, and similarly for $K^{-1}_{s,t}$.

Given $K_{s,t}$ and two complex numbers $w,z$, we define a
\emph{finite} matrix ${\bf K}_{s,t}(z,w)$ from white to black vertices
of the fundamental domain $G_1$, as follows.  Consider $G_1$ as a
weighted periodic bipartite graph on the torus, where the weight of an edge is
the one induced by $K_{s,t}$,   and note that it can contain
multiple edges between two vertices, even if the infinite graph $G$
does not (see e.g. Fig. \ref{graphes_et_polygones_newton}). Consider a
path $\gamma_x$ (resp. $\gamma_y$) winding once horizontally
(resp. vertically) along the torus and multiply by $z$
(resp. $1/z$) the weight  of each edge crossed by
$\gamma_x$ with the black vertex on the left (resp. on the right) and
similarly by $w,1/w$ the edges crossed by $\gamma_y$.
Then, ${\bf K}_{s,t}(z,w)$ is the adjacency matrix of $G_1$, with 
these modified  weights. With the usual graph theory convention, this means that
the $(\rw,\rb)$ element of ${\bf K}_{s,t}(z,w)$ (with $\rw$ (resp. $\rb$) a 
white (resp. black) vertex of $G_1$) is the sum of the weights of the edges 
joining $\rw$ to $\rb$.
Let $Q_{s,t}(z,w)$ (a matrix from black to white vertices of $G_1$) be the adjugate matrix of ${\bf K}_{s,t}(z,w)$  so that $[Q_{s,t}
{\bf K}_{s,t}] (z,w) = P(z,w)
 \text{Id}$ where $P(z,w) = \det( {\bf K}_{s,t}(z,w))$. 
The $(\rb,\rw)$ element of $Q_{s,t}(z,w)$ is
denoted $Q^{\rb,\rw}_{s,t}(z,w)$.

We can now give a formula for the inverse infinite Kasteleyn matrix
$K_{s,t}^{-1}$:
\begin{theorem}\cite{KOS}
  Let $\rb$ and $\rw$ be a black and a white vertex in $G_1$. The following holds for $(x,y)\in \bbZ^2$:
\[
   K_{s,t}^{-1}(\rb,\rw+(x,y)) = \frac{1}{(2i\pi)^2} \int_{\bbT^2} \frac{Q^{\rb,\rw}_{s,t}(z,w)}{P_{s,t}(z,w)} w^x z^y \frac{\td w}{w}\frac{\td z}{z}
\]
where $\bbT^2= \{ z,w \in \Complex^2, \abs{z}=\abs{w}=1 \}$ is the unit complex torus.
\end{theorem}
Here, to avoid confusion, it can be useful to emphasize that
$[Q_{s,t}/P_{s,t}](z,w)$ is the inverse of the finite matrix ${\bf
  K}_{s,t}(z,w)$, while $K_{s,t}^{-1}$ is an inverse of the infinite matrix $K_{s,t}$, with no dependence on $z,w$
(i.e. with the original edge weights).

\subsection{Asymptotics of $K_{s,t}^{-1}$ and Gaussian fluctuations in the ``liquid phase''}
\label{sec:Asy}

It is shown in \cite{KOS} that, for an integer slope $ (s,t) \in \intN
\cap \bbZ^2 $, the Laurent polynomial $ P_{s,t}$ has either no zeros on the
unit torus (in which case $\mu_{s,t}$ corresponds to a gaseous phase)
or has a unique zero of order two (which corresponds to a liquid
phase).  For any non-integer slopes in $\intN \setminus \bbZ^2$, instead,
$P_{s,t}$ has exactly two conjugate simple zeros. In this case Theorem
\ref{th:asy} gives the asymptotics of $K_{s,t}^{-1}(\rb,\rw)$ when the two
vertices $\rb,\rw$ are far apart. We emphasize that in our applications (i.e. in
the proof of Theorem \ref{th:clt}), we will have to consider only
cases where the Newton polygon has no integer points in its interior.

\begin{theorem}\cite{KOS}
\label{th:asy}
   Fix $(s,t)$ a non-integer slope in $\intN$, so that $P_{s,t}$ has
   two simple zeros $(z_0,w_0)$ and $(\bar{z}_0,\bar{w}_0)$ on
   $\bbT^2$. Let $\alpha = \tfrac{\partial }{\partial z}
   P_{s,t}(z_0,w_0)$ and $\beta = \tfrac{\partial}{\partial w}
   P_{s,t}(z_0,w_0)$ and define $\phi(x,y) = x\alpha z_0 - y\beta
   w_0$. Then the map $\phi : \Reel^2 \mapsto \Complex$ is
   invertible, the matrix $Q_{s,t}(z_0,w_0)$ is of rank $1$ and can be
   written as $U_{s,t} V_{s,t}^T$ where the column vector $U_{s,t}$
   (resp. $V_{s,t}$) is indexed
   by the black (resp. white) vertices of $G_1$. Moreover, we have
   \begin{eqnarray}
     \label{eq:K-1asy}
   K^{-1}_{s,t} \bigl(\rb,\rw+(x,y) \bigr) = -\Im \left( \frac{w_0^x z_0^y U_{s,t}(\rb) V_{s,t}(\rw)}{\pi \phi(x,y)} \right) + O\left(\frac{1}{x^2+y^2}\right)
   \end{eqnarray}
where $O((x^2+y^2)^{-1})$ has to be understood as
$\frac{h(x,y)}{x^2+y^2+1}$ with $h$ bounded on $\bbZ^2$ and $\Im(z)$
denotes the imaginary part of $z$.
\end{theorem}

\begin{remark}
  The invertibility of $\phi$ is a consequence of the fact that $\alpha
z_0$ is not collinear with $\beta w_0$. This is not proved explicitly
in \cite{KOS} but the argument is simple: Both the torus $\bbT^2$ and
$P_{s,t}^{-1}(\{0\})$, the set of zeros of $P_{s,t}$, are two-dimensional manifolds in $\Complex^2$ which contain
$(z_0,w_0)$. The tangent space of $\bbT^2$ at  $(z_0,w_0)$
is given by
\[
p_1=\{  a(i z_0,0) + b (0,iw_0) + (z_0,w_0), (a,b)\in \bbR^2\}
\]
and the tangent space of $P_{s,t}^{-1}(\{0\})$ is
\[
p_2=\{\zeta (\beta, -\alpha) + (z_0,w_0),\zeta\in \Complex \}.
\]
Since $(z_0,w_0)$
is a simple
zero of $P_{s,t}$ seen as a function on $\bbT^2$, one necessarily has
$p_1\cap p_2=\{(z_0,w_0)\}$ and it is easy to check that this fails if
$\alpha z_0=\lambda \beta w_0$ with $\lambda\in \bbR$.
\end{remark}

The asymptotic  expression  \eqref{eq:K-1asy} is the main tool for the
following result, that is proven in Appendix \ref{sec:moments}:
\begin{theorem}
   \label{th:clt}
Fix a non-integer slope $(s,t)\in\intN$ and
$g\in G^*$. Under the Gibbs
measure $\mu_{s,t}$, the moments of the variable
  \begin{eqnarray}
\frac{h(f)-h(g)-(\mu_{s,t}(h(f))-\mu_{s,t}(h(g)))}{\sqrt{{\var}_{\mu_{s,t}}(h(f)-h(g))}}\sim
\pi\;\frac  {h(f)-h(g)-(\mu_{s,t}(h(f))-\mu_{s,t}(h(g)))}{\sqrt{\log  |\phi(f)-\phi(g)|}}
  \end{eqnarray}
tend as
  $|\phi(f)-\phi(g)|\to\infty$ to those of a standard Gaussian $\cN(0,1)$.
\end{theorem}
Here and later, when we write $\phi(f)$ we mean $\phi(x,y)$ if $f$ is the $(x,y)$ translate of a face in the fundamental domain $G_1$.
The fact that the variance of $h(f)-h(g)$ behaves like
$(1/\pi^2)\log |\phi(f)-\phi(g)|$ is proved in \cite{KOS}.

For the hexagonal lattice and under the assumption that $f$ and $f_0$ are along
the same column of hexagons, convergence of the moments is proven in
 \cite{Kenyon_lectures}. 
The general case is qualitatively more difficult and requires
 non-trivial work (see the discussion at the beginning of Appendix \ref{sec:moments}; our proof uses ideas from
 \cite[Sec. 7]{K_non_planaire}
 but the setting
here is more general and we give a more explicit control of the
``error terms'').



\subsection{Almost-planar boundary conditions}

A central role will be played by ``almost planar'' boundary
conditions.

We say that $h\in \Omega$ is an almost-planar height function with slope
${(s,t)}\in\intN$  if
there exists $C$ such that, for every $f\in G^*$,
\begin{eqnarray}
  \label{eq:almost}
  |h(f)-\mu_{s,t}(h(f))|\le C.
\end{eqnarray}

We will sketch briefly in Section \ref{sec:gen} a proof that almost-planar
boundary conditions actually exist for every ${( s,t)}\in \intN$ (even with $C=1$).

Theorem \ref{th:clt} implies the following:
\begin{theorem}
  \label{th:stimeflutt}
  Fix a non-integer slope ${(s,t)}\in \intN$ and let $m$ be an almost-planar boundary
  condition with slope ${ (s,t)}$. Let the finite graph $G'$ be as in
  Definition \ref{def:G'}.  One has for every $\gep>0$ and $n>0$:
\begin{eqnarray}
  \label{eq:stimeflutt}
  \pi_{m,G'}(\exists f\in G^*:|h(f)-\mu_{ s,t}(h(f))|\ge L^\gep)=O(L^{-n}).
\end{eqnarray}
\end{theorem}
\begin{remark}
\label{rem:refined}
  The maximal equilibrium height fluctuation with respect to the
  average height should be of order $\log L$ with high probability,
  but we will not need such a refined result.
\end{remark}

\begin{proof}[Proof of Theorem \ref{th:stimeflutt} given Theorem \ref{th:clt}]
  For the hexagonal lattice this is given in detail in \cite{CMST}
  (see Proposition 4 there). For general graphs the proof is almost
  identical and we recall just the basic principle.  

By monotonicity (see Section \ref{sec:monotonie}),
  the event $E_f=\{h(f)-\mu_{ s,t}(h(f))\ge L^\gep\}$ is more likely
  if we change the boundary condition for a higher one, i.e. if we
  replace $m$ with $m'$ such that $h_{m'}(f')\ge h_{m}(f')$ for every
  face $f'\notin (G')^*$ adjacent to some face in $(G')^*$ (this
  set of faces is denoted here $\partial G'$, and $(G')^*$ is the
  collection of faces of $G'$). Assume without loss of generality that
  the reference face $f_0$ where heights functions are fixed to zero
  belongs to $\partial G'$. Then choose a \emph{random} boundary
  condition $m'$ from the measure $\mu_{s,t}$, and this time fix its
  height at the reference face as $h_{m'}(f_0)=
  h_m(f_0)+L^\gep/2=L^\gep/2$. Thanks to Theorem \ref{th:clt}, one has
  $h_{m'}\ge h_m $ on $\partial G'$, except with probability
  $O(L^{-n})$ for any given $n$. Finally, with such random boundary
  condition, by the DLR property the probability of $E_f$ is nothing but
  $\mu_{s,t}(h(f)-\mu_{ s,t}(h(f))\ge L^\gep/2)$, which is also
  $O(L^{-n})$, again thanks to Theorem \ref{th:clt}.
\end{proof}

\subsection{Perfect matchings, capacities and maximal configurations}
\label{sec:gen}

\subsubsection{Linear characterization of height functions}\label{sec:characterisation_lineaire}

The set of height functions corresponding to a perfect matching of a
finite subset of $G$ can be characterized by linear inequalities as follows.

Consider as in Definition \ref{def:G'} a finite sub-graph $G'$ of $G$ and a
boundary condition $m \in \gO$. In this subsection we will use $m$
(even if it is not necessarily periodic) as reference matching for the
definition of height functions. For any two neighboring faces $f,f'$
with a common edge $e$ oriented positively (i.e. such that going from
$f$ to $f'$ one crosses $e$ leaving the white vertex on the right),
let the oriented capacities $d(f,f')$ and $d(f',f)$ be defined as
follows:
\begin{eqnarray}
  \label{eq:def_capacites}
  d(f,f')=
\begin{cases}
    0 \text{ if }  e\notin G'\\
0 \text{ if }  e\in G' \text{ and } e\in m\\
1 \text{ if }  e\in G' \text{ and } e\notin m;
  \end{cases}
\quad\quad
d(f',f)=
\begin{cases}
    0 \text{ if }  e\notin G'\\
1 \text{ if }  e\in G' \text{ and } e\in m\\
0 \text{ if }  e\in G' \text{ and } e\notin m  
\end{cases}
.
\end{eqnarray}
Now for any pair of faces
$f,f'$ (not necessarily neighbors) let $D(f,f')$ be the \emph{minimum} over
all paths $f=f_1, \ldots, f_n=f'$ in $G^{*}$ of the sum of the
$d(f_i,f_{i+1})$ (the minimum is well defined, the capacities being
non-negative).
\begin{proposition}
\label{prop:fournier}
An integer-valued function $h$ on $G^*$ is the height function (with reference matching $m$) of a
matching in $\Omega_{m,G'}$ if and only if
   \begin{eqnarray}
     \label{eq:2}
  \ D(f,f') \geq h(f')-h(f) \text{ for every }  f,f'\in G^*.
   \end{eqnarray}
\end{proposition}
\begin{proof} The proof is in the spirit of \cite[Theorem 1]{Fournier}.
The ``only if'' part is trivial since, going back to Section \ref{sec:height}, 
it is immediate to see that the
\emph{maximal} possible height difference $h(f')-h(f)$ between neighboring faces $f,f'$, for any matching in
$\Omega_{m,G'}$, does not exceed $d(f,f')$.  
As for the ``if'' part, remark first of all that, thanks to
\eqref{eq:2}, for every neighboring faces $f,f'$ one has
$h(f)-h(f')\in\{-1,0,1\}$. Let us ``mark'' all edges $e \in G'$ between faces $f,f'$ (with $e$ oriented positively from $f$ to $f'$) such that
$h(f') - h(f)=d(f,f')$, together with edges $e \in G \setminus G'$ such that $e \in m$. Let $M$ be the union of all marked edges and let us prove it is a matching (note that, automatically, $M\equiv m$ outside of $G'$). For any white (resp. black) vertex $v$, let
$e_v$ be the unique edge incident to $v$ which belongs to $m$. From
\eqref{eq:2} and considering paths that turn counterclockwise (resp. clockwise) around $v$,
it is easy to see that:
\begin{itemize}
\item either all the faces $f$ sharing vertex $v$ have the same value of
$h(f)$ and $e_v$ is the single marked edge around $v$; 

\item or there exists a single marked edge $e'_v\ne e_v$, incident to $v$,
such that $h(f')=h(f)-1$, with $f,f'$ neighboring faces sharing $e'_v$, 
such that
$v$ is on the left (resp. right) when going from $f$ to $f'$.
\end{itemize}
$M$ is thus a matching and by construction 
$M\equiv m$ outside $G'$. 
In conclusion, $M\in
\Omega_{m,G'}$ and of course its height function is just $h$.
\end{proof}


\subsubsection{Maximal and minimal configurations}
\label{sec:minmax}
The characterization of height functions provided by Proposition
\ref{prop:fournier} shows the existence of a unique maximal (resp.
minimal) height function $h_{max}$ (resp. $h_{min}$) in
$\Omega_{m,G'}$. ``Maximal'' means that for any other height function
$h$ in $\Omega_{m,G'}$ satisfying $h(f_0)=0$ (recall from Section
\ref{sec:phase_pure} that the height is fixed to zero at some face
$f_0$ outside of $G'$) one has $h(f)\le h_{max}(f)$ for every $f\in
G^*$.  Indeed, define $h_{max}(f) := D(f_0,f)$ on $G^*$. This
satisfies \eqref{eq:2} (since $D(\cdot,\cdot)$ satisfies the
triangular inequality) and maximality is a consequence of the fact
that $d(f,f')$ is the maximal possible height difference between
neighboring faces. Similarly, one has $h_{min}(f)=-D(f,f_0)$.  Observe
that the height functions $h_{max},h_{min}$ (with respect to the
reference configuration $m$) vanish outside $G'$ as they should (this
is because the set of faces of $G$ not belonging to $G'$ is connected,
recall Definition \ref{def:G'}).

\subsubsection{Free paths and possible rotations}


\begin{definition}
Fix a matching $M\in \Omega$. We say that an oriented path $\gamma$ in $G^*$ is a
\emph{free path} (relative to $M$) if all edges crossed by $\gamma$ are free
(i.e. not occupied) and have the same orientation (i.e. either all of them have their white vertex
on the right of $\gamma$ or all of them on the left).
If white vertices are on the right (resp. left) then $\gamma$ is called a
positive (resp. negative) free path.
\end{definition}
See Fig. \ref{exemple_chemin_pente}.
\begin{figure}[h]
   \centering
   \includegraphics[width=12cm]{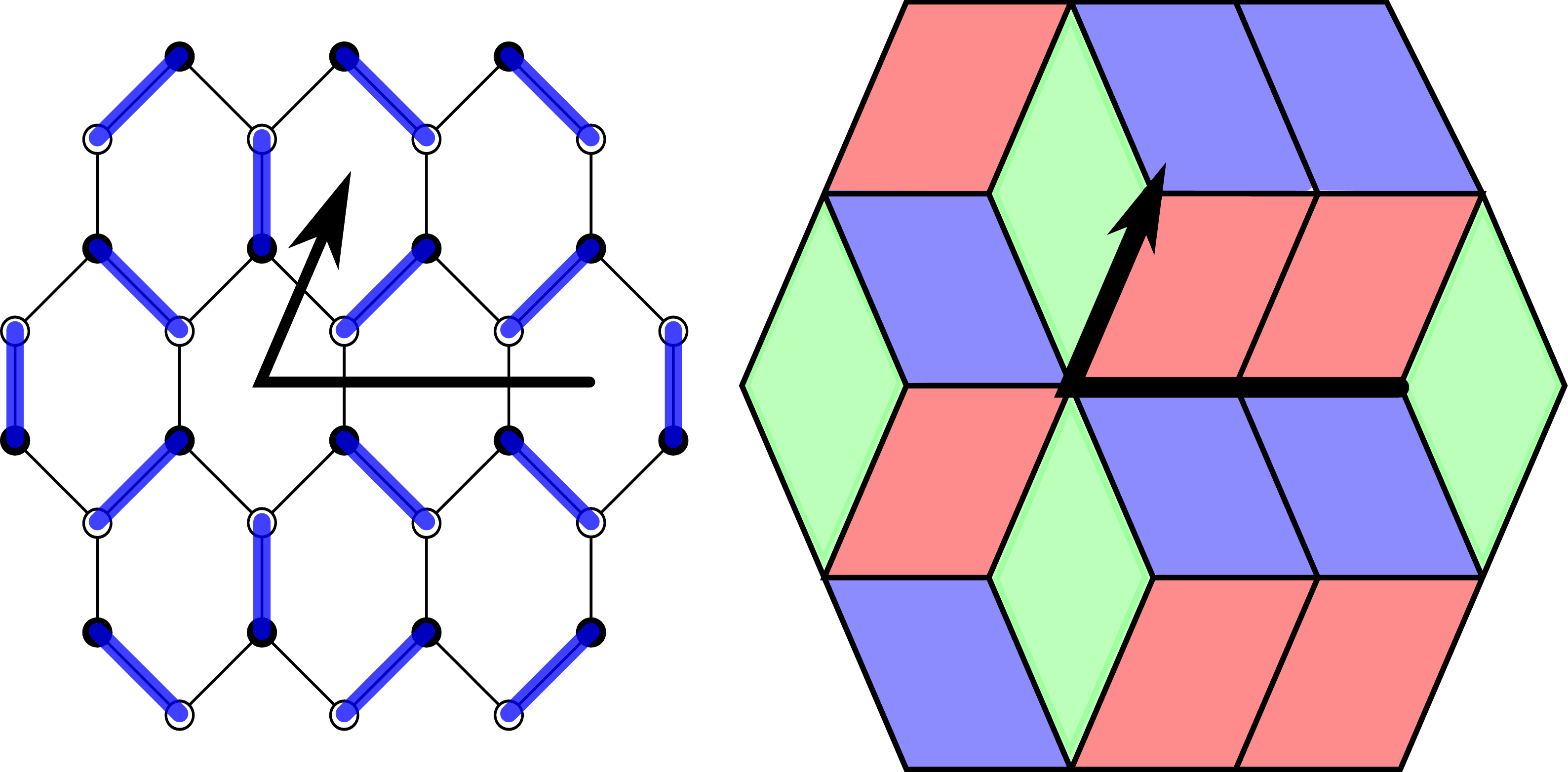}
   \caption{The relation between height function and dimer
     configuration in the case of the honeycomb
     graph. It is easy to see the right-hand drawing as a stepped
     surface in $3$ dimensions. It should be clear from the right-hand
     drawing that the positive free path (thick line) moves away at
     constant speed from the $(1,1,1)$ plane; at its endpoint (marked
     by an arrow) the free path cannot be possibly continued, and a
     cube can be added there (i.e. a rotation can be performed in the left-hand
     drawing).
     }
   \label{exemple_chemin_pente}
\end{figure}

A first observation is that free paths  cannot form loops:
\begin{proposition}
\label{prop:noloop}
  Let $\gamma$ be a free path relative to some $M\in \Omega$,
  and assume that $\gamma$ forms a simple loop. Then, for every
  $M'\in\Omega$ the edges crossed by $\gamma$ are free.
\end{proposition}
Together with Assumption \ref{hyp_pas_deterministe}, this excludes loops.

\begin{proof}[Proof of Proposition \ref{prop:noloop}]
Let $h$ be the height function of $M'$, with reference matching
$M$. Let $f$ be a face along $\gamma$. By symmetry, suppose $\gamma$ is a positive path. Since all edges are traversed
with the positive orientation, we have
\begin{eqnarray*}
h(f)=h(f)+|\{\text{edges crossed by $\gamma$ and occupied in $M'$}\}|\\
-|\{\text{edges crossed by $\gamma$ and occupied in $M$}\}|.
\end{eqnarray*}
Since $\gamma$ crosses no occupied edge of $M$ by assumption, it crosses no
occupied edge of $M'$ either.
\end{proof}

A second observation is that, since only the reference matching
(however it is chosen) makes a contribution to the height
difference along a free path  $\gamma$, the height function
is non-increasing (resp. non-decreasing) if $\gamma$ is a positive
(resp. negative) free path.
An important consequence, that we will need in Section \ref{sec:dyn_billes}
to upper bound the equilibration time of the dynamics, is
the following:
\begin{proposition}
\label{prop:strictdecr}
Fix $(s,t)\in\intN$.
Let $M\in\Omega$ be such that the corresponding height function stays
between two planes of slope $(s,t)$ and mutual distance $H$. All free paths
 relative to $M$ have length at most $C H$, where the
constant $C$ depends only on $(s,t)$.
%
\end{proposition}

\begin{proof}[Proof of Proposition \ref{prop:strictdecr}] Since the
  graph $G$ is periodic, there exists only a finite number $k$ of
  types of faces that are not obtained by integer translation of each
  other.  Let $\gamma$ be a positive free path  (if it is a
  negative free path, the argument is similar) relative to some matching
  $M$. We claim that,
  \begin{gather}
    \label{eq:claim}
  \text{  if one walks $n$ steps along $\gamma$,
the
  function $f\mapsto h(f)-\mu_{s,t}(h(f))$ decreases by at least}\\
\nonumber\text{$-\lfloor n/k\rfloor
  \gep$ for some $\gep=\gep_{s,t}>0$.}
  \end{gather}
Then, the proposition follows
  (with the constant $C$ being inversely proportional to $\gep/k$)
  because the function $\mu_{s,t}(h(\cdot))$ on $G^*$ is essentially planar
  with slope $(s,t)$.

  To prove \eqref{eq:claim}, observe first that the matching $M$ gives
  no contribution to the variation of $h$ along $\gamma$ (all crossed
  edges are free) so that the variation of $h-\mu_{s,t}(h)$ is simply
  minus the $\mu_{s,t}$-average number of crossed edges which are
  covered by dimers.  Fix some face $f\in \gamma$ and walk along
  $\gamma$ until a face $f'$ which is a translate of $f$ is reached
  (the number of steps is at most $k$). The $\mu_{s,t}$-average of
  crossed dimers between $f$ and $f'$ is non-negative and we will
  actually prove that it is strictly positive and independent of the
  type of face $f$, which implies the claim.  Indeed, let
  $\tilde\gamma$ be the infinite periodic path on $G^*$ obtained by
  repeating periodically the finite portion of $\gamma$ which joins
  $f$ to $f'$. If the average of crossed edges is zero, then clearly
  the slope of the height under the measure $\mu_{s,t}$ along the
  asymptotic direction of $\tilde\gamma$ is extremal, which
  contradicts the assumption that $(s,t)$ is in the interior of the
  Newton polygon $N$. Uniformity w.r.t. the type of the face $f$ is
  just a consequence of the fact that the number of different face
  types is finite.
\end{proof}


For any face $f$ of $G$, there exist exactly two ways to perfectly
match its vertices among themselves. Label ``$+$'' one of the two
matchings, and ``$-$'' the other (according to some arbitrary
rule). If $M$ is a matching of $G$ such that the vertices of $f$ are
matched only among themselves, we call ``rotation around $f$'' the
transformation which consists in leaving $M$ unchanged outside of $f$,
and in flipping from ``$-$'' to ``$+$'' (or vice-versa) the matching of
the edges of $f$. If some vertices of $f$ are matched to vertices not
belonging to $f$, then the rotation is not possible.

Free paths yield a way to find a face where an elementary rotation is
possible. Given $M\in \Omega$, we pick an arbitrary face $f_1$ and we
construct a growing sequence $\{\gamma_n\}_{n\ge 1}=\{(f_1,\dots,f_n)\}_{n\ge 1}$, of
positive free paths, with $\gamma_1\equiv(f_1)$ (an analogous
construction gives a growing sequence of negative free paths). Given
$\gamma_n$, consider all faces $f$ which are neighbors of $f_n$ and
such that going from $f_n$ to $f$ one crosses a free edge with white
vertex on the right. Choose $f_{n+1}$ (according to some arbitrary
rule) among such faces. If there are no such faces available, we say
that the procedure stops at step $n$. In this case, it means that
every second edge around $f_n$ is occupied by a dimer, and this is
exactly the condition so that a rotation at $f_n$ is possible.
Altogether, we have proven:
\begin{proposition}
\label{prop:rotationpossible}
  Fix $(s,t)\in\intN$.
Let $M\in\Omega$ be such that the corresponding height function stays
between two planes of slope $(s,t)$ and mutual distance $H$. Within
distance $C_{s,t}\,H$ from any face $f$ there exists a face $f^+$
(resp. $f^-$) where a
rotation is possible; such a rotation increases (resp. decreases) the
height at $f^+$ (resp. $f^-$) by $1$ and $h(f^+) \leq h(f) \leq h(f^-)$.
\end{proposition}

\subsubsection{Almost planar height functions}
Here we prove that almost-planar height functions satisfying
\eqref{eq:almost} with $C=1$ do exist (under the assumption that $(s,t)\in\intN$ is a non-integer slope).  Indeed from Theorem
\ref{th:clt} and Borel-Cantelli we get that, for every fixed
$\delta>0$, almost all configurations from $\mu_{s,t}$ satisfy
\begin{eqnarray}
  \label{eq:BC}
|h(f)-\mu_{s,t}(h(f))|\le B+ \delta |\phi(f)-\phi(f_0)|
\end{eqnarray}
for some random $B$, where $f_0$ is the face where the heights are fixed to zero. Take one of these configurations. Let $A_n$ the
set of faces at graph-distance at most $ n$ from $f_0$ and suppose that
$h(f)-\mu_{ s,t}(h(f))<-1$ (the argument is similar if the difference
is $>1$) for some $f\in A_n$.  The same argument that led to
Proposition \ref{prop:strictdecr} shows that, if $\delta$ is chosen
small enough (say much smaller than the constant $\gep_{s,t}$ in
\eqref{eq:claim}), any positive free path  $\gamma$ starting from $f$
is of length $O(\delta n/\gep_{s,t})\le n$ for $n$ large enough.  Therefore, the
last face $f'$ of $\gamma$ is in $A_{2n}$ and (by the properties of
positive free paths) one has $h(f')-\mu_{s,t}(h(f'))< -1$. By
Proposition \ref{prop:rotationpossible}, a rotation is possible at
$f'$ and it increases $h(f')-\mu_{s,t}(h(f'))$ by $1$. The
configuration thus obtained clearly still verifies \eqref{eq:BC} with
the same $B$ and the quantity \[\Delta=\sum_{f\in A_{2n}}|h(f)-\mu_{s,t}(h(f))|{\bf 1}_{|h(f)-\mu_{s,t}(h(f))|>1}
\]
decreased by $1$. Since $\Delta$ is finite, the
procedure can be repeated a finite number of times until there is no
point left in $A_n$ with $|h(f)-\mu_{s,t}(h(f))|>1$.  One concludes
easily using the fact that $n$ can be taken arbitrarily large.  \qed

\section{Dynamics and mixing time}
\label{sec:dynamics}

The dynamics we consider lives on the set $\Omega_{m,G'}$ of matchings on
a finite subset $G'\subset G$ (as in Definition \ref{def:G'}) with boundary condition $m\in \Omega$. Every
face $f$ of $G'$ has a mean-one, independent Poisson
clock. When the clock at $f$ rings, if the rotation around $f$ is
allowed, flip a fair coin: if ``head'' then choose the ``$+$'' matching
of the edges of $f$, if ``tail'' then choose the ``$-$'' matching. In other words, perform the rotation around $f$  with probability $1/2$.

Call $\mu_t^M$ the law of the dynamics at time $t$, started from $M$.
\begin{proposition}
\label{prop:connect}
For $t\to\infty$,  $\mu_t^M$ converges to the uniform measure
$\pi_{m,G'}$.
\end{proposition}
\begin{proof}
  It is obvious that $\pi_{m,G'}$ is invariant and reversible, so one
  should only check that the dynamics connects all the configurations
  in $\Omega_{m,G'}$. This is done by using the free paths of Section
  \ref{sec:gen}.

  Let $M\in \gO_{m,G'}$ and let $M_{max}\in \gO_{m,G'}$ be the
  matching corresponding to the maximal height function $h_{max}$
  introduced in Section \ref{sec:minmax}. The height function $h$ of
  $M$ with reference matching $M_{max}$ is clearly non-positive and
  vanishes outside $G'$.  Pick a face $f$ such as $h(f) \le -1$ and
  consider a positive free path $\gamma$ growing from $f$ (as in the
  proof of Proposition \ref{prop:rotationpossible}). Along $\gamma$
  the height function $h$ cannot grow and, $G'$ being finite, $\gamma$
  has to stop after a finite number of steps. The last face $f'$ of
  $\gamma$ clearly is inside $G'$ (since the height is zero outside)
  and we have already discussed that a rotation is possible at $f'$
  and it increases $h(f')$ by $1$.  By recursion, $M$ can be
  transformed into $M_{max}$ by a finite sequence of elementary
  rotations inside $G'$. Arbitrariness of $M$ allows to conclude.
\end{proof}

As usual \cite{LPW}, an informative way to quantify the speed of approach to
equilibrium is via the mixing time, defined as
\begin{eqnarray}
  \label{eq:tmix}
  \tmix=\tmix(m,G')=\inf\{t>0:\max_{M\in \Omega_{m,G'}}\|\mu_t^M-\pi_{m,G'}\|<1/(2e)\}
\end{eqnarray}
where $\|\mu-\nu\|$ is the total variation distance of measures
$\mu,\nu$ and the choice of the value $1/(2e)$ is conventional (any
other value smaller than $1/2$ would do). With this choice, one has
\cite{LPW}
\begin{eqnarray}
  \label{eq:sottomol}
  \max_{M\in \Omega_{m,G'}}\|\mu_t^M-\pi_{m,G'}\|\le e^{-\lfloor t/\tmix\rfloor}.
\end{eqnarray}

\begin{wrapfigure}{L}{6.5cm}
   \includegraphics[width=5cm]{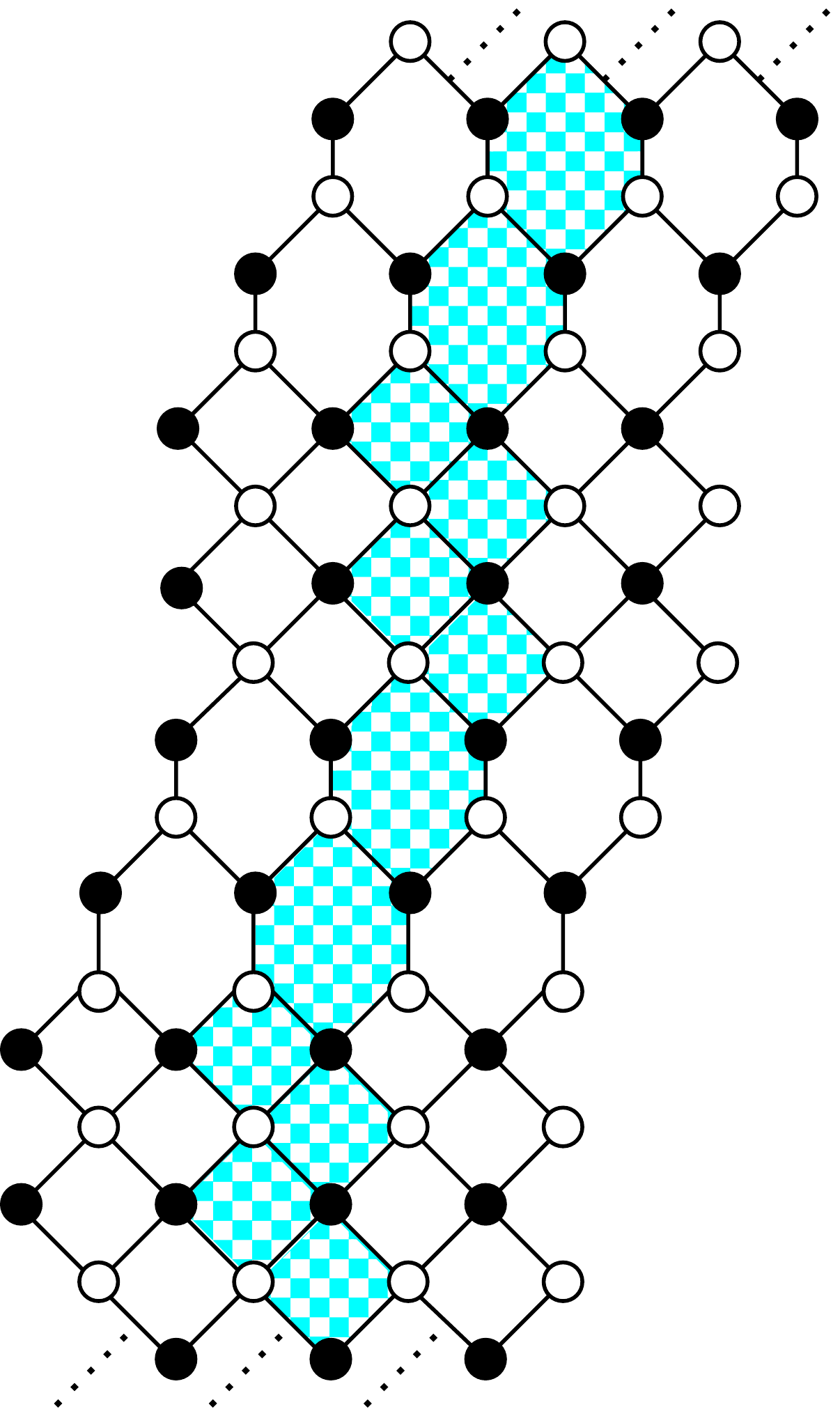}
    \captionsetup{width=6cm}
   \caption{An example of the class of graphs where our results could be extended, see Remark \ref{rem:layers}. The shaded region is a \emph{thread} (cf. Section 
\ref{sec:billes}). Layers of squares and hexagons can be of arbitrary vertical thickness and the periodicity in the ``vertical'' direction can be arbitrarily large. } 
   \label{fig:graph_general}
\vspace{0cm}
\end{wrapfigure}

We will study the mixing time when the boundary conditions are
almost planar.  The following is the main result of this
work: 
\begin{theorem}
\label{th:tmixU}
Fix $(s,t)\in \intN$, $m$ an almost-planar boundary condition of slope
$(s,t)$ and let $G'$ be as in Definition \ref{def:G'}.  If $G$ is
either the square, hexagon or square-hexagon lattice
(cf. Fig. \ref{graphes_et_polygones_newton}) then there exists some
$c>0$ such that
\begin{eqnarray}
  \label{eq:tmixU}
\quad\quad \quad\quad \quad\quad \quad\quad c L^2\le \tmix\le  L^{2+o(1)}.
\end{eqnarray}
\end{theorem}

We refer to Section \ref{sec:previous} above for a discussion of previously
known results.



\begin{remark} The proof of the lower bound in \eqref{eq:tmixU}
  actually shows the following: if the dynamics is started from the
  maximal configuration, which has an excess volume  $c\, L^3$
  with respect to the typical (almost flat) equilibrium configuration,
  it takes a time $c_1 L^2$ before the excess volume becomes smaller
  than say $(c/2)\,L^3$ (which is still very large w.r.t. typical volume fluctuations). In this sense, the equilibration time lower bound is optimal.
\end{remark}

\begin{remark}
\label{rem:layers}
Our result could be extended to a class of graphs obtained by
alternating periodically layers of squares and hexagons (see Fig.
\ref{fig:graph_general}). On the other hand, we will explain in
Section \ref{sec:entro} why our method does not (and should not!) work
for general periodic bipartite graphs $G$, in particular not for graphs like
the square-octagon lattice which possesses a ``gaseous phase''.
\end{remark}

\subsection{Monotonicity}\label{sec:monotonie}


It is natural to introduce the following partial order on $\gO$ : $M
\geq M'$ if and only if  $h_M(f) \geq h_{M'}(f)$ for every $f\in G^*.$ As usual, the reference face $f_0$ is assumed to be fixed
once and for all.  Note that the partial order does not depend on the
reference matching used to define the height. We say that an event $ A
\subset \gO$ is increasing if $M \geq M'$ and $M' \in A$ implies $M\in
A$. We define in the usual way stochastic domination: $\mu \succeq \mu'$
if $\mu(A) \geq \mu'(A)$ for every increasing event $A$.

\begin{proposition}\label{prop:monotonie_dynamique}
   The dynamics defined in Section \ref{sec:dynamics} is monotone, that is
$\mu_t^M \succeq \mu_t^{M'}$ for every  $t$  and every $M\ge M'$.
\end{proposition}
\begin{proof}
  Couple the dynamics started from $M$ and $M'$ by using the same
  clocks and the same coin tosses. Partial order is preserved along time. Indeed, it suffices to observe that
if $h_M(f)=h_{M'}(f)$ and a rotation at $f$ that increases the height by 
$1$ is possible for $M'$, then necessarily the configuration of the edges 
of $f$ in $M$ is the same as in $M'$, otherwise at some face $f'$ neighboring
$f$ one would have $h_M(f')<h_{M'}(f')$.
\end{proof}

\begin{remark}
\label{rem:GMC}
As in \cite[Sec. 2.2]{CMT}, one can realize all the evolutions $M_t^{M_0}$ for all possible initial conditions $M_0$ on the same probability space, with the property that if $M_0\le M_0'$ then almost surely $M_t^{M_0}\le M_t^{M_0'}$ for every $t\ge0$. This construction is called
\emph{global monotone coupling}.
\end{remark}

\begin{proposition}\label{prop:monotonie_mesure}
  If $A$ is an increasing event, then
$
   \pi_{m,G'}( \cdot | A) \succeq \pi_{m,G'}.
$
\end{proposition}

\begin{proof}
  Remark that in the proof of Proposition \ref{prop:connect} we
  showed that the maximal configuration can be reached from any other
  by a chain of rotations that increase the height, so $h_{\max} \in
  A$ and $A$ is connected. Consider the original dynamics started from $h_{\max}$ and the
  reflected dynamics (again started from $h_{\max}$) where each update
  that would leave $A$ is canceled. It is clear that they converge to
  $\pi_{m,G'}$ and $\pi_{m,G'}(\cdot| A)$ respectively and that, when
  coupled by using the same clocks and coin tosses, the second always
  dominates the first.
\end{proof}

Monotonicity allows to apply ``censoring inequalities'' of Peres and
Winkler \cite{PW} which, roughly speaking, say the following: if the
dynamics is started from the maximal or minimal configuration, deleting
some updates along the evolution in a pre-assigned way
(i.e. independently of the actual realization of the dynamics)
increases the variation distance from equilibrium. The precise
statement we need (cf. Corollary \ref{th:PW} below) is a bit more general
than what is proven in \cite{PW} but the proof is almost identical, so
we will just point out where some modification is needed.

Consider a probability measure $\pi$ on $\gO$ and $P^{(v)}, v \in \cV$
a set of transition kernels that satisfy reversibility
($\pi(\sigma)P^{(v)}(\sigma \rightarrow
\eta)=\pi(\eta)P^{(v)}(\eta\rightarrow \sigma)$) and monotonicity. We
define a dynamics on $\gO$ by assigning a Poisson clock of rate $c_v$
to each $v\in \mathcal V$ and applying $P^{(v)}$ when $v$ rings. The dynamics of
Section \ref{sec:dynamics} corresponds to $\cV= (G')^*$, $c_v=1$ and $P^{(f)}$ the kernel that corresponds to a
rotation around $f$ with probability $1/2$ (if allowed).

\begin{theorem}
\label{th:prePW}
  Let $\nu_0$ be a probability measure on $\Omega$ such that $\frac{ \td
    \nu_0}{\td \pi}$ is increasing. Consider $\nu_t$ the law at time
  $t$ of the dynamics started from $\nu_0$. Then, for every $t\ge0$, $\frac{ \td
    \nu_t}{\td \pi}$ is increasing and, if
 $\{\mu_t\}_{t\ge0}$ is a family of probability measures such
  that $\nu_t \preceq \mu_t$ for all $t$, one has
\[
    \norm{\nu_t-\pi} \leq \norm{\mu_t - \pi}.
\]
\end{theorem}

\begin{corollary}
\label{th:PW}

Let $\nu_0$ be as in Theorem \ref{th:prePW}. Suppose that for all $v\in
\cV$, for all $\nu$ such that $\frac{\td \nu}{\td \pi}$ is increasing, we have $P^{(v)} \nu \preceq \nu$. Let $\mu_t$ be the law at time $t$ of the
dynamics
started from $\nu_0$, where the rates $c_v$ of the Poisson clocks
are replaced by deterministic time-dependent rates $\tilde c_v(s)$, such that
$0\le \tilde c_v(s)\le c_v$ for every $0\le s\le t$. Then,
\[
  \text{for every } t\ge0,\quad \nu_t \preceq \mu_t \text{ and } \norm{\nu_t-\pi} \leq \norm{\mu_t - \pi}.
\]
\end{corollary}

\begin{remark}
  The hypothesis of ${\td \nu_0}/{\td \pi}$ increasing is immediate if the
  dynamics is started from the maximal configuration $h_{\max}$, since in that case $\nu_0$ is
concentrated on $h_{max}$.
\end{remark}

As in \cite{PW}, the proof of Theorem \ref{th:prePW} follows directly
from the following two lemmas.
\begin{lemma}\label{lemme:derivee_croissante}
  With the above definitions, for any probability measure $\mu$, if
  $\frac{\td \mu}{\td \pi}$ is increasing then $\frac{\td P^{(v)}
    \mu}{\td \pi}$ is increasing.
\end{lemma}
This replaces Lemma 2.1 of \cite{PW}, which uses explicitly the fact that
the dynamics is of ``heat-bath'' type.
\begin{proof}
\begin{align}
\frac{\td P^{(v)} \mu}{\td \pi}(\sigma) & = \frac{1}{\pi(\sigma)} \sum_s \mu(s) P^{(v)}(s \rightarrow \sigma)
          =  \frac{1}{\pi(\sigma)} \sum_s \frac{\mu(s)}{\pi(s)} \pi(s) P^{(v)}(s \rightarrow \sigma) \\
         & =  \frac{1}{\pi(\sigma)} \sum_s \frac{\mu(s)}{\pi(s)} \pi(\sigma) P^{(v)}( \sigma \rightarrow s) \label{eq:usage_reversibilite}
          =  \bbE^{(v)}_{\sigma} \bigl[ \frac{\td \mu}{\td \pi} (X) \bigr]
\end{align}
where $X$ is the state after one action of $P^{(v)}$, starting  from
$\sigma$. The third equality uses the
reversibility and the monotonicity of $P^{(v)}$ shows that the last
expression is increasing in $\sigma$.
\end{proof}

\begin{lemma}\cite[Lemma 2.4]{PW}\label{lemme:peres_winkler}
  If $\mu$, $\nu$ are two probability measures on $\Omega$ such that
  $\frac{\td \nu}{\td \pi}$ is increasing and $\nu \preceq \mu$, then
  $\norm{\nu-\pi} \leq \norm{\mu - \pi}$.
\end{lemma}

\begin{proof}[Proof of Corollary \ref{th:PW}]
Decompose the  Poisson point process (PPP) of density $c_v$ on $\mathbb R^+$
as the union of two independent PPPs, $X$ and $Y$, of non-constant densities
$\tilde c_v(s)$ and $c_v-\tilde c_v(s)\ge0$. The dynamics $\mu_t$ is obtained
by erasing the updates from the process $Y$.
  From Lemma \ref{lemme:derivee_croissante} we get that $\frac{ \td \nu_t}{\td \pi}$ is increasing. Censoring an update $P^{(v)}$ at time $t$ conserves the stochastic domination because, by induction, $ \nu_t=P^{(v)} \nu_{t^-} \preceq \nu_{t^-} \preceq \mu_{t^-} = \mu_t$.
\end{proof}

\section{Mapping to a ``bead model''}

\subsection{From dimers to ``beads''}
\label{sec:billes}

From this point onward, we will assume that the graph $G$ is either
 the square, hexagon  or square-hexagon graph (see Fig.
\ref{graphes_et_polygones_newton}) since we will use some of their geometric
properties.

The set $\Omega$ of matchings of $G$ can be mapped into the configurations
of what we call a \emph{bead model}. Such a correspondence is valid for
more general graphs than the square, hexagon and square-hexagon,
provided that the graph in question possesses a certain ``fibration''
with fibers (or \emph{threads}) satisfying the properties described below.  See Remark \ref{rem:layers} for a  more general class of graphs where this
construction would work.

\begin{figure}[h]
   \centering
   \includegraphics[width=5cm]{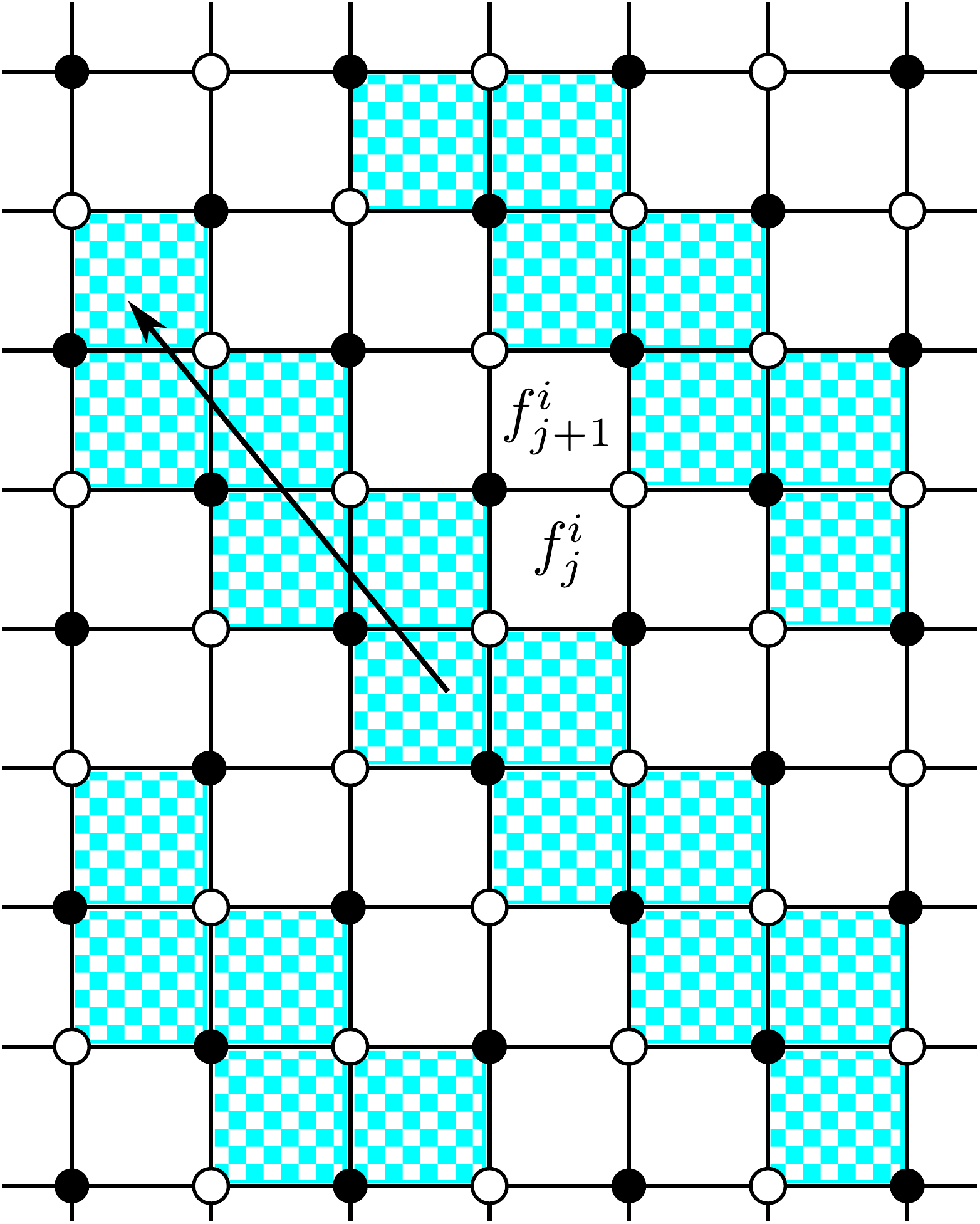}
\includegraphics[width=5cm]{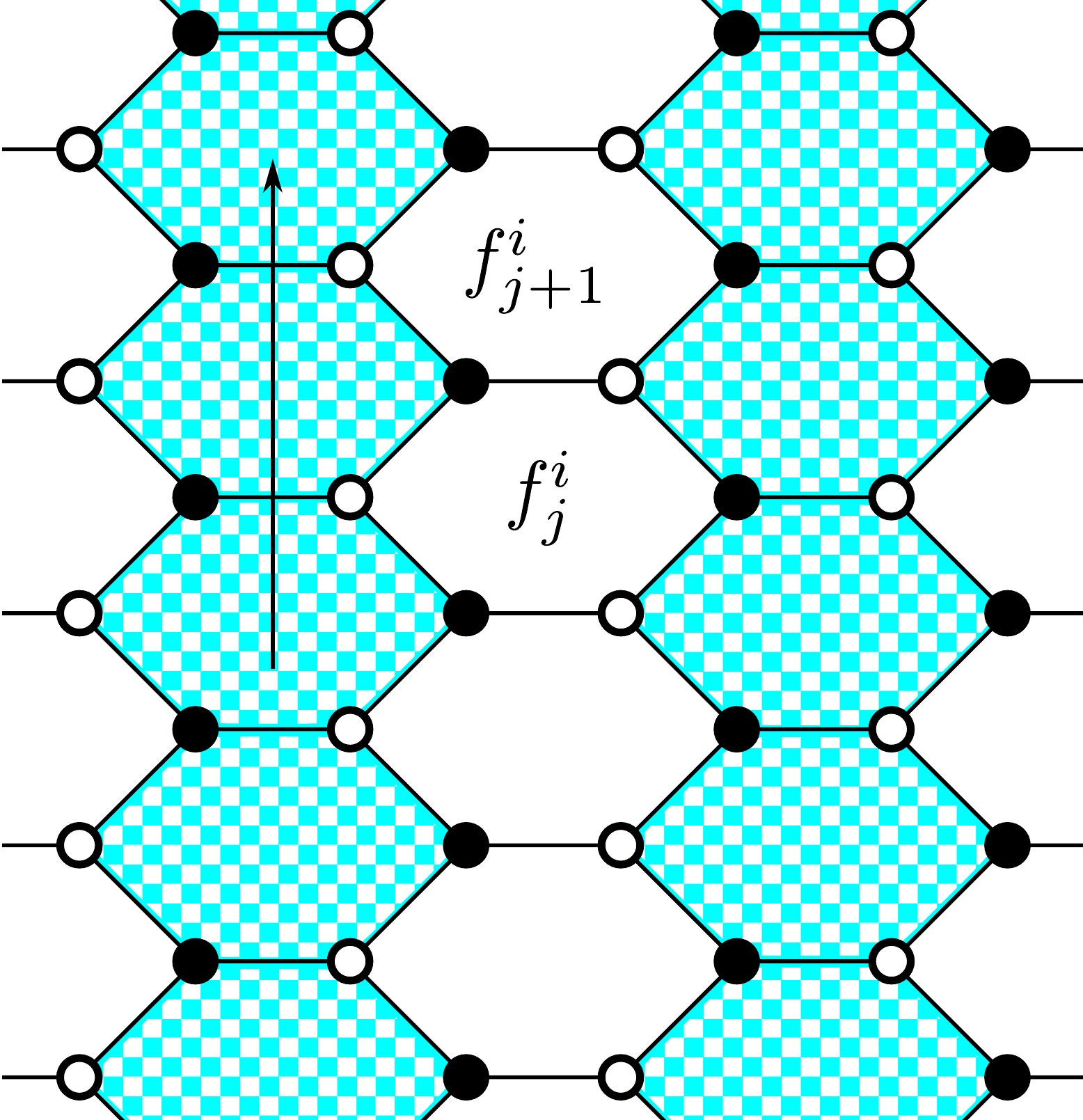}
\includegraphics[width=4cm]{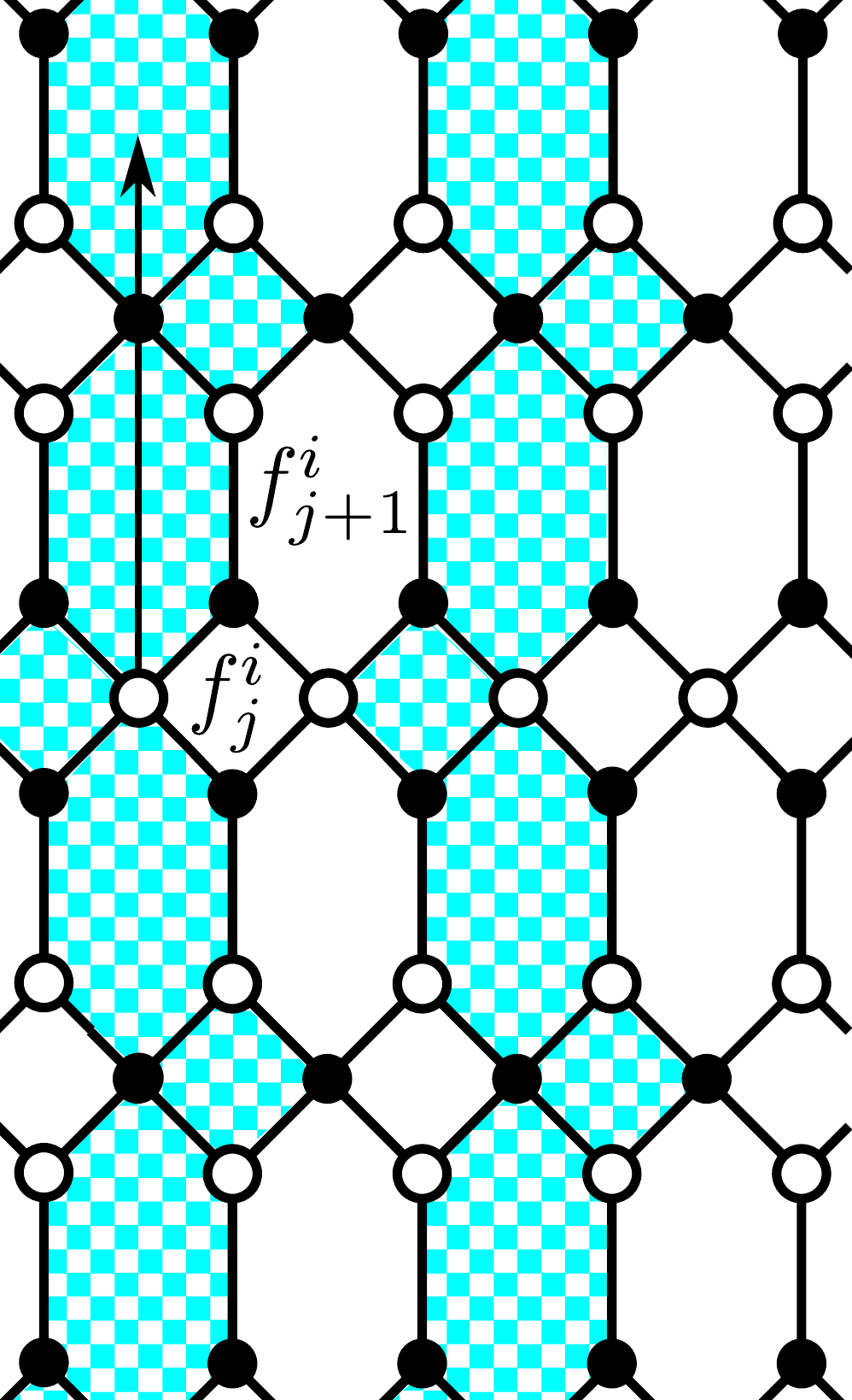}
   \caption{The paths called threads and used in the definition of the bead model for the square, hexagon and square hexagon graph. The arrow shows the orientation of the threads.
     }
   \label{fig:beads}
\end{figure}

As is apparent from Fig. \ref{fig:beads}, for the three types of graph we are
considering, there exists a family of directed periodic paths $\{\gamma_i\}_{i\in\bbZ}$ (called \emph{threads}) on $G^*$ such that
\begin{enumerate}[(i)]
\item  labeling the faces along thread $\gamma_i$ as $\{f^i_j\}_{j\in \bbZ}$,  face $f^i_j$ neighbors only $f^i_{j\pm1}$
and some faces of $\gamma_{i\pm1}$;

\item going from the face $f^i_j$ to $f^i_{j+1}$, one crosses an edge
  of $G$ (call it $e^i_j$) which is positively orientated. Such edges
  are called \emph{transverse edges} and dimers on transverse edges
  are called \emph{beads};

\item threads $\gamma_i$ are obtained one from the other by a suitable $\bbZ^2$ translation
and $\cup_i\gamma_i=G^*$.
\end{enumerate}

\begin{proposition}
  Consider a finite sub-graph $G'\subset G$ as in Definition \ref{def:G'}
(recall that $G\setminus G'$
  is connected) and a boundary condition $m \in \gO$. A matching in
  $\gO_{m,G'}$ is uniquely determined by the position of its
  beads. Furthermore the number of beads in $G'$ on each thread $\gamma_i$ is the same
  for every $M\in\Omega_{m,G'}$.
\end{proposition}
\begin{proof}
  From the definition of height function and using property (ii) above
  of threads, which says that transverse edges are all crossed with
  the same orientation, the height at some $f \in (G')^*$ belonging to
  thread $\gamma_i$ is determined by the number of beads on $\gamma_i$
  between the boundary of $G'$ and $f$. Hence the position of
  beads uniquely determines the height function and thus the
  matching. The total number of beads on $\gamma_i\cap (G')^*$ is
  determined by the height difference between two faces $f_1,f_2$ of
  $\gamma_i\setminus (G')^*$, such that the portion of $\gamma_i$
  between $f_1$ and $f_2$ includes $\gamma_i\cap (G')^*$. This height
  difference is clearly independent of the particular chosen matching
  in $\Omega_{m,G'}$ (because $f_1,f_2$ are outside $G'$ and $G\setminus G'$ is connected).
\end{proof}

In order to have a complete picture, we have to determine the
condition for a set of beads' positions to correspond to a matching of
$G$, that is the kind of constraints beads impose on each other. 
We first need some notations (see Fig. \ref{fig:exemple_billes}).


\begin{definition}
\label{def:ordering}
Let as above $e^i_j$ denote the $j^\text{th}$ transverse edge of the
thread $\gamma_i$ and let $\rb_j^i$ (resp. $\rw_j^i$) denote its black
(resp. white) vertex. Let $\Gamma^i$ be the set of vertices of $G$
which belong both to $\gamma_i$ and to $\gamma_{i+1}$, and order
vertices in $\Gamma^i$ following the same direction as for the faces
along the threads. Note that $\Gamma^i$ contains the white vertices of
transverse edges of $\gamma_i$ and the black vertices of transverse
edges of $\gamma_{i+1}$. Given transverse edges
$e^i_{j},e^{i+1}_{j'}$ on $\gamma_i,\gamma_{i+1}$, we write
$e^i_j\prec e_{j'}^{i+1}$ (resp $e^i_j\succ e_{j'}^{i+1}$) if
$\rw_j^i$ is below (resp. above) $\rb_{j'}^{i+1}$.
\end{definition}

\begin{proposition}\label{prop:modele_billes}
  A set of bead positions corresponds to a matching in $\Omega$ if and
  only if, for any two consecutive beads on the same thread
  (i.e. beads on transverse edges $e^i_a,e^i_b, a<b$ with no bead
  between them along the same thread $\gamma_i$), there is a unique
  bead in thread $\gamma_{i-1}$ and a unique bead in  thread $\gamma_{i+1}$ such
  that their positions $e^{i-1}_c ,e^{i+1}_d$ satisfy $e^i_a\prec
  e^{i-1}_c \prec e^i_b$ and $e^i_a\prec e^{i+1}_d \prec e^i_b$.
\end{proposition}

\noindent \emph{Proof.} We advise the reader to keep an eye on Fig. \ref{fig:exemple_billes} while
reading this proof. 

Proof of the ``only if'' part. Without loss of generality we can
consider threads $\gamma_i$ and $\gamma_{i+1}$. Note that following
$\Gamma^i$ between $\rw^i_a$ and $\rw^i_b$ there is necessarily
exactly one more black vertex than white vertex (because $G$ is
bipartite).  Since by assumption there are no beads on $\gamma_i$
between $e^i_a$ and $e^i_b$, all the white vertices have to be matched
within $\Gamma^i$. This leaves exactly a single black vertex which has
to be matched along a transverse edge in $\gamma_{i+1}$. The
corresponding dimer is the unique bead such that its position
$e^{i+1}_d$ satisfies $e^i_a\prec e^{i+1}_d\prec e^i_b$.


  Proof of the ``if'' part. Suppose that the bead positions are given
  and that they satisfy the properties above. This automatically fixes
  which  transverse edges are occupied and which are free. To
  see that the rest of the matching is also (uniquely) determined,
  proceed as follows.  With the same notations as above, consider the
  vertices of $\Gamma^i$ between $\rw^i_b$ and $\rb^{i+1}_d$.  These
  vertices are all matched with each other (because by construction
  $e^{i+1}_d$ is the ``highest'' bead in $\gamma_{i+1}$ such that
  $e^{i+1}_d\prec e^i_b$) and they form a path with an equal number of alternating
  white and black vertices, so there is a unique way of matching them. The same goes for vertices between $\rw^i_a$ and $\rb^{i+1}_d$. \qed
\begin{remark}
  \label{rem:rotazioni}

Fix $G'\subset G$ and a boundary condition $m\in\Omega$.  Under the
measure $\pi_{m,G'}$, conditionally on the positions of the beads in
$\gamma_{i\pm1}$, the beads of $\gamma_i$ are independent and each has
a \emph{uniform distribution} in a certain finite set of adjacent
transverse edges (two transverse edges being adjacent if they are of
the form $e^i_j,e^i_{j+1}$; remark that on the square lattice they can
actually share a vertex, while on the hexagonal lattice they
cannot). Indeed, given a bead on the transverse edge $e^i_a$, it is
possible to move it up to $e^i_{a+1}$ (resp. down to $e^i_{a-1}$) via
a rotation of the face $f^i_{a+1}$ (resp. $f^i_{a}$), provided that
$f^i_{a+1}$ (resp. $f^i_{a}$) is within $G'$ and that the new position
does not violate the ordering properties of Proposition
\ref{prop:modele_billes}.  Uniformity of the distribution is trivial
from uniformity of the unconditional measure $\pi_{m,G'}$.
Note also that moving a bead up (resp. down) implies changing by $-1$ (resp. $+1$) the height of the face just above (below) it along the thread.
\end{remark}

\begin{wrapfigure}{r}{5cm}
\centering
   \includegraphics[width=5cm]{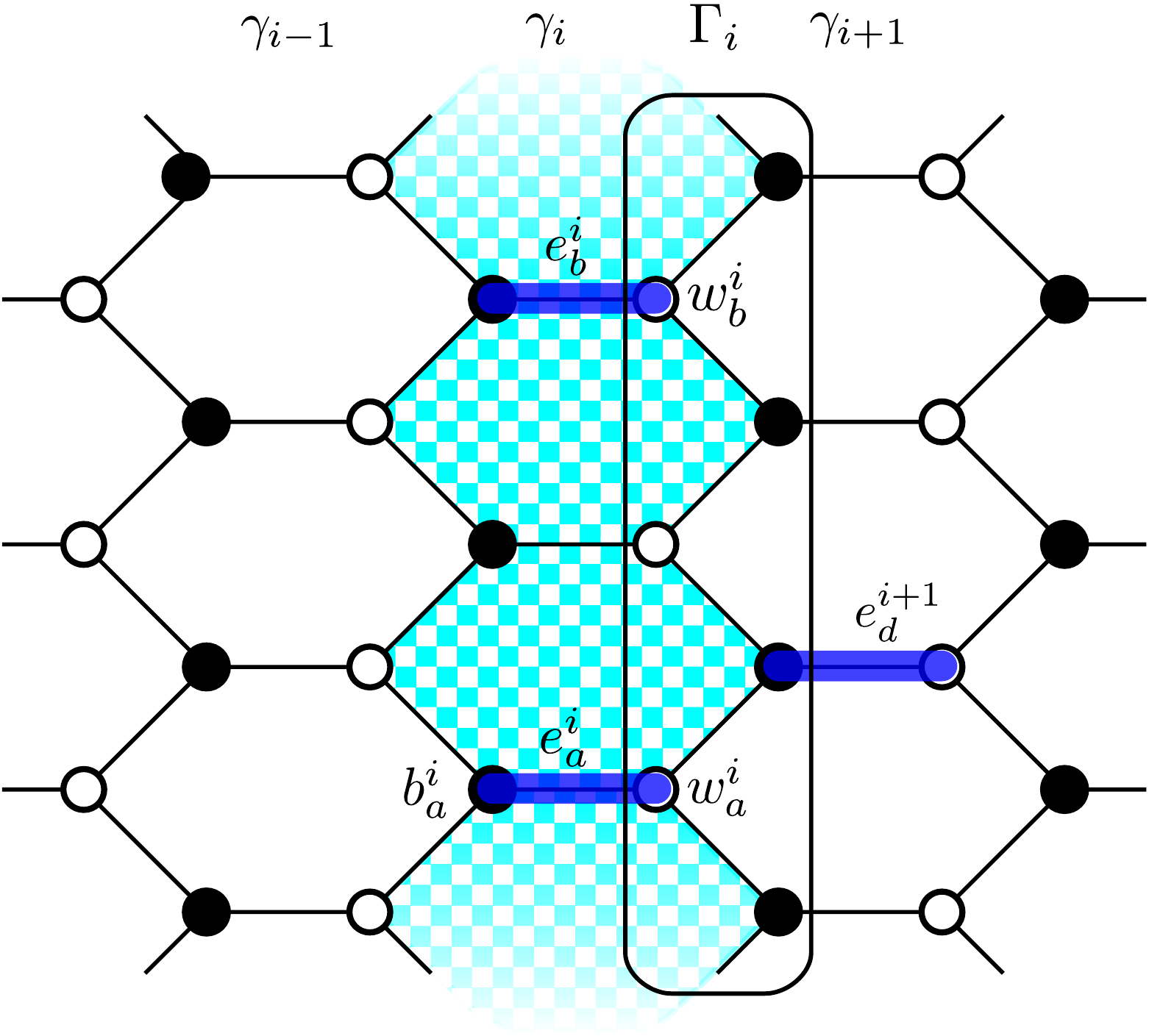}
   \captionsetup{width=5cm}
   \caption{An example of the situation in the proof of proposition
     \ref{prop:modele_billes}. Beads on the edges $e_a^i$ and $e_b^i$
     are indicated by thick lines.  }
   \label{fig:exemple_billes}
\end{wrapfigure}


\subsection{Dynamics in terms of  beads}\label{sec:dyn_billes}

The Glauber dynamics defined in Section \ref{sec:dynamics} has a
simple interpretation in terms of the bead model. As observed in Remark
\ref{rem:rotazioni}, a rotation is equivalent to moving a bead to an
adjacent transverse edge in the same thread (in particular, rotations
are possible only at faces adjacent to a bead, and the configuration of
beads outside $G'$ is frozen).  We can then redefine
the dynamics as follows. Each bead in $G'$ has a mean-one independent Poisson
clock; when it rings, with probabilities $1/2,1/2$ move the bead
either up or down to the adjacent transverse edge if allowed by the
boundary conditions and by the ordering
properties. 

\subsubsection{Fast dynamics}
\label{sec:fd}
As mentioned in the introduction, we will not work directly with the
original Glauber dynamics but rather with an auxiliary one.  We will
actually need \emph{two} auxiliary dynamics: one, that we call
\emph{synchronous fast dynamics}, will be useful to upper bound the mixing
time, while the \emph{asynchronous fast dynamics} will provide a lower bound.

\begin{definition}\
\label{def:sincro}
\begin{enumerate}[(i)]
\item We define the \emph{synchronous fast dynamics} as follows. We have two
independent mean-1 Poisson clocks. When the first (resp. second) one
rings we resample all the bead positions on even-labelled
(resp. odd-labelled) threads following the equilibrium measure
$\pi_{m,G'}$  conditioned on the state of beads on odd
(resp. even) threads.  
\item The \emph{asynchronous fast dynamics} is
defined instead by giving each bead in $G'$ an independent mean-1 Poisson
clock. When a clock rings, the position of the corresponding bead is
resampled from $\pi_{m,G'}$ conditioned on the position of all other beads.
Note that, by construction, beads outside $G'$ are frozen.
\end{enumerate}
\end{definition}

Thanks to monotonicity of the original Glauber dynamics, one sees
easily that both synchronous and asynchronous fast dynamics are
monotone.

\begin{remark}\label{rq:description_updates}
  Recall Remark \ref{rem:rotazioni}: the positions accessible to a
  single bead, given the beads of neighboring threads and the boundary
  conditions, form a segment
  of the transverse edges of its thread and these segments are mutually
  non-intersecting. Thus, under both synchronous and asynchronous
  dynamics each bead is resampled with the uniform law on a finite segment.
\end{remark}

\subsubsection{Comparisons of mixing times, and constrained dynamics}\label{sec:comp_dyn}


As announced, the asynchronous dynamics will provide a mixing time
lower bound for the original one. 
\begin{proposition}
\label{prop:ta}
   Fix $G'\subset G$ as in Definition \ref{def:G'} and a boundary condition $m$, not necessarily almost planar.
   Let $\tmix$ be the mixing time for the original dynamics and let $T^a$ be the first time such that the law of the asynchronous dynamics, started
from the maximal configuration, is within variation distance $1/(2e)$ from equilibrium. Then, $\tmix\ge T^a$.

\end{proposition}
\begin{proof}
  Recall the definition \eqref{eq:tmix} of mixing time: to get a lower
  bound, we can just look at the evolution from the maximal
  configuration $h_{max}$, which we can therefore assume to be the
  initial configuration of both original and asynchronous fast
  dynamics.  From the description of the asynchronous fast dynamics in
  Section \ref{sec:fd}, we see that we can couple the two dynamics
  using the same Poisson clocks for each bead. To prove that the asynchronous
  dynamics approaches equilibrium faster than the original one it is
  enough to show that if $\frac{\td \nu}{\td \pi}$ is increasing then,
  writing $P^{(b)}$ and $\tilde P^{(b)}$ for the kernels corresponding
  to an update of the bead $b$ according to the original and
  asynchronous fast dynamics respectively, one has 
  \begin{eqnarray}
    \label{bs}
    \tilde P^{(b)} \nu \preceq
    P^{(b)} \nu. 
  \end{eqnarray} Indeed, together with monotonicity 
  this guarantees that the height function is stochastically lower
  under the asynchronous dynamics than under the original one and
  then Theorem
  \ref{th:prePW}
  can be applied (with $\mu_t$ the law of the original dynamics and $\nu_t$ that of the asynchronous one).

  By conditioning on all the beads except $b$, we can assume that
  $\nu$ is a measure on some interval $\{0, \ldots, k\}$ and that
  $\pi$ is the uniform measure on the same interval (cf. Remark
  \ref{rem:rotazioni}) in which case it is trivial to check
  \eqref{bs}.  Indeed, $\tilde P^{(b)} \nu=\pi$ for every $\nu$, and
  $\td P^{(b)} \nu/\td\pi$ is increasing (cf. Lemma
  \ref{lemme:derivee_croissante} and recall the assumption $\td
  \nu/\td\pi$ increasing), which implies  $\pi\preceq P^{(b)} \nu$.
\end{proof}


\subsubsection{Constrained dynamics}

Next, we bound $\tmix$ from above using the synchronous dynamics: this
works well only if the dynamics is constrained between two
configurations whose height functions are not too different. Given two
matchings $M_-\le M_+$ in $\Omega_{m,G'}$ ($M_-$ will be
called ``the floor'' and $M_+$ ``the ceiling'') the constrained
dynamics is defined in the subset $\Omega^{M_\pm}_{m,G'}\subset\Omega_{m,G'}$
such that $M_-\le M\le M_+$: it is obtained from the original
dynamics, erasing all updates which would exit $\Omega^{M_\pm}_{m,G'}$.
It is elementary to check that monotonicity still holds, and the equilibrium measure is of course
$\pi_{m,G'}^{M_\pm}:=\pi_{m,G'}(\cdot|M_-\le \cdot \le M_+)$.  The
distance between floor and ceiling is defined as $H=\max_{f\in
  G^*}(h_{M_+}(f)-h_{M_-}(f))$. To avoid a proliferation of notations,
we still call $\tmix$ the mixing time  of the dynamics constrained
between $M_-$ and $M_+$, and
$\mu_t^M$ its law at time $t$, with initial condition $M$.

To estimate $\tmix$ within logarithmic multiplicative errors, we can
restrict ourselves to the evolution started from the extreme
configurations (see for instance \cite[Eq. (6.5)]{CMT}):
\begin{lemma}\label{lemme:mixage_depuis_maximale}
 Consider the dynamics constrained
between $M_-$ and $M_+$.   For any $t > 0$ and any $M\in \Omega_{m,G'}^{M_\pm}$,
\[
   \norm{\mu_t^{M} - \pi_{m,G'}^{M_\pm}} \leq 2 H \abs{(G')^*} \max\left[
\norm{\mu_t^{M_+} -\pi_{m,G'}^{M_\pm} },\norm{\mu_t^{M_-} -\pi_{m,G'}^{M_\pm} }
\right]
\]
with $ \abs{(G')^*}$ the number of faces of $G'$.
\end{lemma}

In the next result, $\tmix^s$ denotes the mixing time of the
  synchronous dynamics constrained between $M_-$ and $M_+$ (just
  take the Definition \ref{def:sincro} of  the synchronous dynamics and replace $\pi_{m,G'}$ by $\pi^{M_\pm}_{m,G'}$ there):
\begin{proposition}
\label{prop:ts}
  Fix $G'\subset G$ as in Definition \ref{def:G'} and boundary condition $m$. Suppose $m$
  is almost planar with slope $(s,t)\in\intN$ and consider
  floor/ceiling $M_\pm\in\Omega_{m,G'}$, at distance
  $H$ from each other. 
  We have
\[
   \tmix \leq \bigl( C_{s,t} H^2  \log^3(H \abs{(G')^*}) \bigr) \tmix^s.
\]

\end{proposition}
\begin{proof}
  A complete proof for the hexagonal graph is given in \cite[Section
  6.2]{CMST} and since for the square or square-hexagon graph not much
  changes, we will be somewhat sketchy.

  For simplicity we will write $\pi$ for $ \pi_{m,G'}^{M_\pm}$.  One
  first proves that a single update of the synchronous fast dynamics
  can be realized by letting the original dynamics evolve for a time
  $O(H^2(\log H)( \log \abs{(G')^*}))$ while censoring some updates.
  Indeed, consider the dynamics obtained by starting from $M$ and setting to
  $0$ the rate of updates of beads on, say, odd threads in the
  original dynamics. It is clear that this auxiliary evolution converges to the
  uniform measure on configurations of beads on even thread
  conditioned by the beads on odd threads, which is exactly the
  measure after one ``even update'' of the synchronous fast
  dynamics, and Remark \ref{rq:description_updates} allows us to easily
  compute its mixing time. Beads on odd
  threads are frozen and those on even threads are completely
  independent so we have to
  compute the mixing time for a set of independent one-dimensional
  simple random walks on domains of the type $\{0, \ldots, k_i\}$. By
  Proposition \ref {prop:strictdecr} the $k_i$ are bounded by $C_{s,t}
  H$ (because, if a bead can be moved $n$ steps up, necessarily there
  is a length-$n$ free path along the thread) so the mixing time for
  each walk is $O(H^2\log H)$ and for $O(\abs{(G')^*})$ such walks it
  becomes $O(H^2 (\log H )(\log \abs{(G')^*}))$.

  It is then clear that the law of the synchronous fast dynamics at
  time $t$, call it $\nu_t^M$, coincides (except for a negligible
  total variation error term), with that of the original dynamics at
  time $ n(t)\Delta$ after censoring suitable updates, where
  \[\Delta=C_{s,t} H^2 (\log H)( \log \abs{(G')^*})\le C_{s,t} H^2 \log^2(H
  |(G')^*|) \] and $n(t)$ is a Poisson random variable of average $2t$
  (this is the number of updates within time $t$ for the synchronous
  dynamics). Since we will take $t$ large, we can replace
  $n(t)$ with its average (we skip details). If Corollary \ref{th:PW}
  is applicable (see below), we obtain then that $ \norm{\mu^{M_\pm}_{2t \Delta}
    -\pi}\le \norm{\nu^{M_\pm}_t-\pi}.  $ Then, using Lemma
  \ref{lemme:mixage_depuis_maximale} and \eqref{eq:sottomol},
\begin{gather}
\sup_M\norm{\mu^M_{2A \Delta \tmix^s}-\pi}\le
2H|(G')^*|\max\left[\norm{\mu_{2A \Delta \tmix^s}^{M_+} -\pi
  },\norm{\mu_{2A \Delta \tmix^s}^{M_-} -\pi } \right] \\
\le 2H|(G')^*|\max\left[
\norm{\nu^{M_+}_{A \tmix^s}-\pi},\norm{\nu^{M_-}_{A \tmix^s}-\pi}
\right]\le 2H|(G')^*|e^{-\lfloor A\rfloor}
\end{gather}
which is smaller than $1/(2e)$ for some $A$ of order $\log (H
|(G')^*|)$.  

To see that Corollary \ref{th:PW} is applicable, we have
to check that ${\dd \mu}/{\dd \pi}$ increasing implies
$P^{(b)}\mu\preceq \mu$ with $P^{(b)}$ the kernel of the update of
bead $b$ under the original dynamics, where beads move by $\pm1$ along
their respective thread. Conditioning on all other beads, we can assume that
$\mu$ is a probability on an interval $\{0,\dots,k\}$ and that $\pi$
is the uniform measure on the same interval. Then, summation by parts
shows that
\[
P^{(b)}\mu (f)-\mu(f)=-\frac12\sum_{x=1}^k[\mu(x)-\mu(x-1)][f(x)-f(x-1)]
\]
which is negative if $f$ is increasing ($\mu$ is also increasing).
\end{proof}

\subsubsection{Volume drift}\label{sec:drift_volume}

In this section we study the time evolution of the volume between two
configurations under the (a)synchronous fast dynamics.  This will be
the key to evaluate their mixing time and thus, thanks to Propositions
\ref{prop:ta} and \ref{prop:ts}, the mixing time of the original
dynamics. Note that in Proposition \ref{prop:volume_drift} we \emph{do not}
require the boundary condition to be almost-planar.

\begin{proposition}\label{prop:volume_drift}
  Let $M^1, M^2 \in \gO_{m,G'}$ be such that $M^1 \leq M^2$ and let
  $M^i_t, i=1,2$ denote the evolution starting from $M^i$ and
  following the fast dynamics (synchronous or asynchronous: we use the
  same notation). Letting $\cF_t$ denote the filtration induced by
  $\{M^i_s,i=1,2,s\le t\}$ and $V_t = \sum_{f \in G^*} [ h_{M^2_t}(f)
  - h_{M^1_t}(f)]$, then $V_t$ is a supermartingale:
\[
  \E \bigl[ V_t | \cF_{t'} \bigr] \leq V_{t'} \text{ for every } t\ge t'.
\]
The same holds if the fast dynamics is constrained between a floor $M_-$ and
a ceiling $M_+$.
\end{proposition}

Note that, since the volume is expressed as a sum of height differences, it does not matter
whether $M_t^1$ evolves independently of $M^2_t$ or not.
\begin{proof}

  First of all note that  the expected drift 
  \begin{eqnarray}
    \label{eq:frift}
    \lim_{\gep\searrow0}\frac{1}{\gep}[\bbE(V_{\gep})-V_0]
  \end{eqnarray}
  is the same for the synchronous and for the asynchronous fast
  dynamics (as a function of the initial conditions $M^1,M^2$): this is thanks to the
  fact that the volume is a sum of height differences over faces, that
  expectation is linear and that beads in the same thread are updated
  independently in a step of the synchronous dynamics. (This does not
  imply that the process $V_t$ itself or even its average is the same
  for the two dynamics). In the following, we will therefore assume
  that we deal with the synchronous dynamics and prove that \eqref{eq:frift} is
not positive. 

From the proof of Proposition \ref{prop:connect}, we see that there
exists a sequence of configurations $M_{(0)},\dots,M_{(k)}$ such that
$M_{(0)}=M^2,M_{(k)}=M^1$ and $M_{(i)}$ is obtained by $M_{(i-1)}$ via
a rotation that decreases the height at some face $f$. Writing the
volume difference between $M^2$ and $M^1$ as a telescopic sum of
volume differences between $M_{(i)}$ and $M_{(i-1)}$ and using
linearity of the expectation, we see that to prove \eqref{eq:frift} we
can restrict to the case where $M^1$ and $M^2$ differ only by a
rotation on a single face $f$.  We will actually prove that the
expected change of $V$ from a single update is $0$ except when $f$ is
suitably close to the boundary of $G'$, in which case it can be
negative (see Remark \ref{rq:borne_inf_drift} for a more precise
discussion).

The proof will be given only for the
square-hexagon graph since it contains all the difficulties, but the
same method works equally well for the hexagon or square graph.


\begin{figure}[h]
   \centering
    (a)\includegraphics[width=6cm]{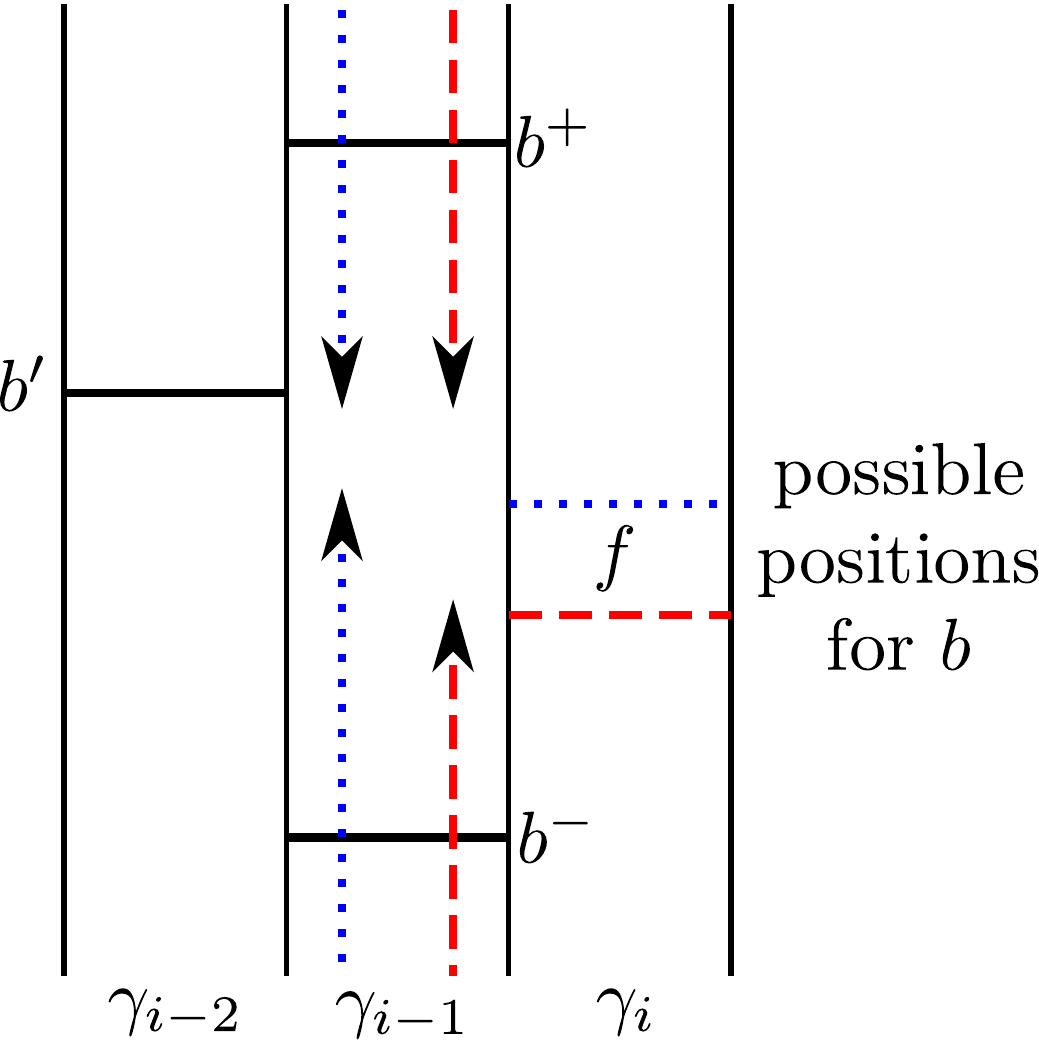}
(b)\includegraphics[width=5cm]{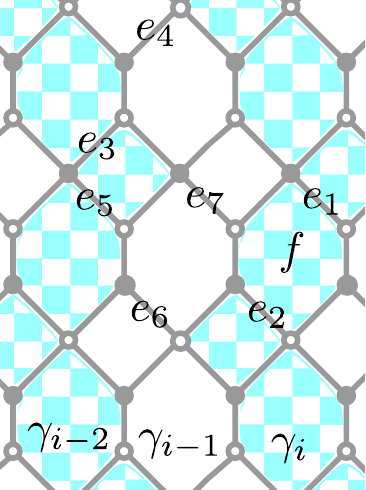}
   \caption{Volume drift when $f$ is a hexagon. (a): a schematic representation of threads $\gamma_i,\gamma_{i-1},\gamma_{ i-2}$
     with the beads $b,b^\pm,b'$. (b): a more detailed view of the
     three threads in question, with the hexagonal face $f$ and the
     transverse edges $e_1,\dots,e_7$ mentioned in the proof.
     \label{calcul_drift_hexagone} }
\end{figure}

\begin{figure}[h]
   \centering
\includegraphics[width=6cm]{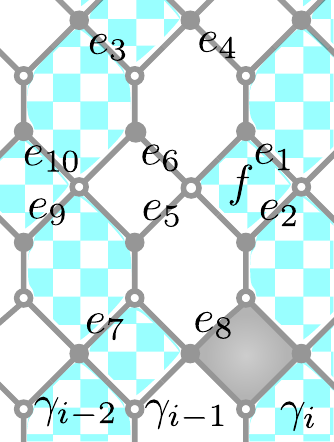}
   \caption{Volume drift when $f$ is a square. The threads
     $\gamma_i,\gamma_{i-1},\gamma_{ i-2}$
with the face $f$ and the transverse edges $e_i, i=1,\dots,10$. A square face like $f$, which is adjacent to hexagons of the thread to its left, is called
a ``Type I'' square, while a square like the shaded one (adjacent to hexagons of the thread to its right) is called ``Type II''.
Note
that if $f$ were a Type II square say again on $\gamma_{i}$, then by symmetry $\gamma_{i-1}$ would
not contribute to the volume change but the thread $\gamma_{i+1}$ would. 
   \label{calcul_drift_carre}}
\end{figure}

{\bf Case 1: $f$ is a hexagon.}  Suppose that the face $f$, whose
rotation brings $M^1$ to $M^2$, is a hexagon on thread $\gamma_i$
(cf. Fig. \ref{calcul_drift_hexagone}). Then, there is a certain
bead $b$ which in the two configurations is on two different adjacent
transverse edges of $\gamma_i$ on the boundary of $f$ (such edges are
called $e_1$ and $e_2$ in the picture).  Now consider a step of the
synchronous dynamics. If threads with the same parity as $\gamma_i$
are updated then clearly the evolved configurations can be coupled in
order to coincide, and the volume decreases by $-1$. We need to show
that when threads of the opposite parity are updated the average
volume increase is at most $+1$. Clearly, only threads $\gamma_{i\pm1}
$ contribute to the average change of volume (because there is no
discrepancy between $M^1$ and $M^2$ on  threads $\gamma_{i\pm 2}$, which therefore
``screen away'' $\gamma_{i\pm n}, n\ge 3$ from the discrepancy at $f$)
and by symmetry we will just show that $\gamma_{i-1} $ contributes at
most $1/2$.

On $\gamma_{i-1}$ there are a certain number of beads: call $b^+$ the
lowest bead which is above $b$ (with the ordering convention of
Definition \ref{def:ordering}) and $b^-$ the highest bead below
$b$. Note that one or both of them could be absent in $M^1,M^2$ (for
instance $\gamma_{i-1}$ could contain no beads at all): however,
suitably changing the dimer configuration outside of $G'$ (which has
no effect on the dynamics) we can always assume that such beads do
exist (possibly outside $G'$).  Thanks to Proposition
\ref{prop:modele_billes}, there exists then a unique bead $b'$ in
$\gamma_{i-2}$, which is lower than $b^+$ and higher than $b^-$, see
Fig. \ref{calcul_drift_hexagone}(a).

A look at Fig. \ref{calcul_drift_hexagone}(b) and Proposition \ref{prop:modele_billes} suffices
to convince that only the following two mutually exclusive cases can occur (Fig. \ref{calcul_drift_hexagone}(a) corresponds to the first one):
\begin{itemize}
\item if $b'$ is at or above transverse edge $e_3$, then $b^+$ is at
  or above edge $e_4$. Then, the distribution of $b^+$ given all the
  other beads does not depend on whether $b$ is at $e_1$ or $e_2$, and
  therefore $b^+$ (and \emph{a fortiori} beads on $\gamma_{i-1}$ above
  $b^+$) gives no contribution to the average volume change. As for
  $b^-$, instead, its set of possible positions (which forms an
  interval of adjacent transverse edges of $\gamma_{i-1}$, recall
  Remark \ref{rem:rotazioni}), includes exactly one edge more (called
  $e_7$ in the picture) when $b$ is at $e_1$ w.r.t when $b$ is at
  $e_2$. Since the position of $b^-$ is uniform among the available
  positions and since the average of a uniform random variable on
  $\{0,\dots,k\}$ is $k/2$, the average volume change arising from
  bead $b^-$ is $(k+1)/2-k/2=1/2$ (recall that moving a bead up by one transverse
  edge results in decreasing the height of a face by $-1$). It is also
  clear that beads in $\gamma_{i-1}$ and below $b^-$ do not contribute
  to the volume change, since their set of available positions is
  disjoint from that of $b^-$ (cf. Remark
  \ref{rq:description_updates}) and does not change when $b$ is moved
  from $e_1$ to $e_2$.

  In this discussion, we ignored so far the fact that rotations for
  the dynamics are allowed only in the sub-graph $G'$. For the bead
  $b^-$ to reach (after the update) its new available position $e_7$
  starting from the present one, all the necessary elementary
  rotations along thread $\gamma_i$ have to involve faces contained in
  $(G')^*$.  Otherwise, the edge $e_7$ is actually \emph{not} a
  possible position for $b^-$: the set of effectively available
  positions for $b^-$ is then the same when $b$ is in $e_1$ or $e_2$,
  so the volume change associated to $b^-$ is not $1/2 $ but $0$. In
  conclusion, taking into account the fact that $G'$ is not the whole
  graph, only decreases the average volume change. A similar reasoning
  shows that the floor/ceiling constraint can only decrease the
  average volume change. This observation will be picked up again in
  Remark \ref{rq:borne_inf_drift}.

\item if instead  $b'$ is at or below transverse edge $e_5$, then $b^-$ is at
  or below edge $e_6$. The argument is then similar to the first case
  (this time $b^-$ does not feel the effect of the discrepancy and $b^+$ has one
  more edge available, again $e_7$, in $M^2$ than in $M^1$). As in the first case, constraints from boundary conditions or from floor/ceiling can only decrease the volume change.
\end{itemize}

Altogether, in both cases the average volume change from  $\gamma_{i-1}$
is either $1/2 $ or $0$ (depending on the position of $f$ with respect
to the boundary of $G'$ and on the presence of floor/ceiling) and, since the same holds for  $\gamma_{i+1}$,
the overall average volume change is at most zero.

{\bf Case 2: $f$ is a square} (cf. Fig. \ref{calcul_drift_carre}).
This time assume that the discrepancy between $M^1$ and $M^2$ is at a
square face $f$ in $\gamma_i$, i.e. that a bead $b$ is either on edge
$e_1$ or $e_2$. Again, when threads with the parity of $\gamma_i$ are
updated the average volume decreases by $-1$. In this case, it is also
clear that the distribution of beads on $\gamma_{i+1}$ after an update
is the same starting from $M^1$ or $M^2$ (because the edges $e_1$ and $e_2$ meet on the same vertex of $\gamma_{i+1}$), so a non-zero contribution
to the average volume change this time can come only from
$\gamma_{i-1}$ and we will show that this contribution is at most
$+1$. With the same conventions as in Case 1 for the beads $b^\pm$ and
$b'$, we distinguish this time three cases (to avoid repetitions, we
ignore effects due to floor/ceiling and to the fact that $G'\ne G$:
exactly as in Case 1, this can only decrease the average volume
change):
\begin{itemize}
\item if $b'$ is at or above $e_3$ then $b^+$ is at $e_4$ or higher
  and does not feel the effect of the discrepancy. At the same time,
  according to whether $b$ is at $e_2$ or $e_1$, the edges $e_5,e_6$
  are available or not for $b^-$. The average volume change induced by
  $b^-$ is then $(k+2)/2-k/2=1$.

\item similarly, if $b'$ is at or below $e_7$ then $b^-$ does not feel
  the discrepancy (it is at or  below $e_8$) and $b^+$ has two extra available positions (again
  $e_5,e_6$) when $b$ is at $e_2$. Again this gives volume change
  $+1$.

\item finally, when $b'$ is either at $e_9$ or $e_{10}$, then both $b^+$
  and $b^-$ feel the discrepancy: indeed, when $b$ is at $e_2$ then
  position $e_5$ is not available for $b^-$ and $e_6$ is available for
  $b^+$, while when $b$ is at $e_1$ then
  position $e_5$ is available for $b^-$ and $e_6$ is not available for
  $b^+$. Again, the average volume change is $+1$.
\end{itemize}
\end{proof}

\begin{remark}\label{rq:borne_inf_drift}
  It is useful to emphasize that the proof of Proposition
  \ref{prop:volume_drift} showed the following. Assume for simplicity
  that there is no floor/ceiling constraint on the dynamics. If $M^1$
  and $M^2$ differ by a single rotation at $f\in (G')^*$, then the
  average volume change after an update is between $-1$ and $0$. To
  decide whether it is zero or non-zero, proceed as follows. The
  rotation around $f$ changes the set of available positions for a
  certain number (at most two, actually) of beads in threads
  neighboring the thread of $f$: some positions which were not allowed
  before the rotation of $f$ become allowed and vice-versa. If all the
  elementary rotations leading such beads along their threads to the
  new available positions are allowed (i.e. if all the faces
  corresponding to such elementary rotations are in the finite domain
  $(G')^*$) then the average volume change is zero. Otherwise, it is
  different from zero.
\end{remark}

\begin{remark}
\label{rem:calore}  From the proof of Proposition \ref{prop:volume_drift} one can also
  understand why the study of the dynamics  on the square or
  square-hexagon lattice is qualitatively more challenging than on the
  hexagonal lattice. The basic observation due to D. Wilson in
  \cite{Wilson} (although it was not formulated in this terms) is the
  following. Consider the ``fast dynamics'' for the hexagonal lattice
  with initial condition given by some matching $m$, and let
  $H^m_t(i)$ be average at time $t$ of the sum of the heights of the
  faces (in $(G')^*$) belonging to thread $i$.  Then, for any pair of
  initial configurations $m^1,m^2$, the function
  $\psi_t(\cdot):=H^{m^1}_t(\cdot)-H^{m^2}_t(\cdot)$ satisfies the
  discrete heat equation
\begin{eqnarray}
  \label{eq:heateq}
  \partial_t \psi_t(i)=\frac12\Delta \psi_t(i)+R_t=\frac12[\psi_t(i+1)-2\psi_t(i)+\psi_t(i-1)]+R_t
\end{eqnarray}
where the ``error term'' $R_t$ is due to the boundary of $G'$ and to
floor/ceiling constraints (if present) and can be ignored for the sake
of this discussion. Call $V_t =\sum_i \psi_t(i)$ the average volume
difference between the two evolving configurations. Since there are
$O(L)$ threads that intersect $G'$ and the lowest eigenvalue of the
discrete Laplacian on $\{1,\dots,L\}$ is of order $L^{-2}$, one
deduces immediately that after a time of order $L^2\times \log V_0$,
the average volume $V_t$ is very small so that the two configurations
have coupled with high probability.

That the same kind of argument \emph{does not} work for other lattices
can be seen as follows. Consider for instance the square-hexagon
lattice, and assume that (with the terminology of the caption of Fig.
\ref{calcul_drift_carre}) $m^1$ and $m^2$ differ only by the rotation
at a square face $f$ of Type I on thread $j$. Then, clearly at time
zero $\psi_0(j)=1$ and $\psi_0(j')=0$, $j'\ne j$. While it is still
true that \eqref{eq:heateq} holds for $t=0$ and $i=j$ (the initial
drift of $\psi(j)$ is $-1=\frac12\Delta\psi_0(j)$), the equation does
not hold (even at $t=0$) for $i=j\pm1$: indeed, the proof of
Proposition \ref{prop:volume_drift} shows that the initial drift of
$\psi(j-1)$ is $1=\grad\psi_0(j)=\psi_0(j)-\psi_0(j-1)$ (instead of
$1/2=\frac12\Delta\psi_0(j-1)$) and that of $\psi(j+1)$ is $0$
(instead of $1/2$). If on the contrary the face $f$ were a square of
Type II, one would find that the initial drift of $\psi(j-1)$ is $0$
and that of $\psi(j+1)$ is $1=-\grad\psi_0(j+1)$.

We believe that, for initial conditions $m^1,m^2$ such that their
height differences are ``smooth'' on the macroscopic scale, Equation
\eqref{eq:heateq} should still (approximately) hold, for $L$
large. Indeed, there are as many Type I as Type II squares in each
thread: if each of the two types contributes approximately equally to
the height differences $\psi_0(i)$, from the above reasoning one finds
that the initial drift of $\psi(j-1)$ is approximately
$1/2\grad\psi_0(j)-1/2\grad\psi_0(j-1)=1/2\Delta\psi_0(j-1)$. However,
trying to pursue this route seems quite hard (one should show that
``smoothness'' is conserved for positive times) and we had to devise
an alternative approach instead, based on Proposition
\ref{prop:volume_drift} and on Theorem \ref{th:wafer} of next section.

\end{remark}

\section{Proof of Theorem
\ref{th:tmixU}}

\subsection{Mixing time upper bound}\label{sec:upper_bound}

Here we prove the mixing time upper bound of Theorem
\ref{th:tmixU}. The crucial step (Theorem \ref{th:wafer}) is to give an almost-optimal
estimate when the dynamics is constrained between floor and ceiling of
small mutual distance $H$ (in the application, we will take $H=L^\epsilon$ with $\epsilon$ small).
Then, an argument developed in \cite{CMT} allows to deduce the mixing time estimate
for the unconstrained dynamics (see Section \ref{sec:meancur} for a sketch).

\subsubsection{A martingale argument}
\label{sec:mart}
The basic step is to prove the following:

\begin{theorem}
\label{th:wafer}
Consider the same setting as in Theorem \ref{th:tmixU}, but assume
that the height function is constrained between ceiling  and floor
 that are almost-planar configurations of slope $(s,t)\in\intN$, of
mutual distance $H$. Then, $\tmix=O(L^2 H^9 \log^4(HL))$.
\end{theorem}

The main idea will be to apply to the volume between two configurations
 the following classical bound on the hitting time for a supermartingale:
\begin{lemma}\label{lemme:temps_martingale}
  Let $X_t$ be a continuous-time supermartingale such that almost surely
  $0 \leq X_t \leq M$ for every $t\in\bbR^+$ and $ \liminf_{\delta \rightarrow 0}\frac{1}{\delta} \E[ (X_{t+\delta}-X_t)^2|\cF_t]  \geq \nu >0$ whenever $X_t>0$.
 Suppose $X_0 = i > 0$ almost surely, fix $0 <
  m < i$ and let $T^{(m)}= \inf \{ t: X_t \leq m \}$ be the hitting time of
  $[0,m]$. Then we have
\begin{equation*}
   \E [T^{(m)}] \leq \frac{ 2Mi}{\nu}.
\end{equation*}
\end{lemma}
(Just note that if $Z(t)=X_t^2-2MX_t-\nu\min(t,T^{(0)})$, then
$Z_t$ is a negative sub-martingale and compute the average of $Z(t)$ for $t=T^{(m)}$).

\medskip
Let $V_t$ denote the volume between the maximal and minimal evolutions
$M^+_t$, $M^-_t$ \emph{under the synchronous fast dynamics}.  Proposition \ref{prop:volume_drift} shows that
$V_t$ is a super-martingale. Because of the floor and ceiling at
distance $H$, we clearly have $0\leq V_t \leq \abs{(G')^*}H$
deterministically. To apply Lemma \ref{lemme:temps_martingale} we only
need a lower bound on $\E[ (V_{t+\delta}-V_t)^2|\cF_t] $. It is
important to remark that such a quantity \emph{does depend} on how $M^+_t$, $M^-_t$
are coupled, while by linearity it is not necessary to specify the
coupling to compute the drift $\E[ V_{t+\delta}-V_t|\cF_t] $.

\begin{lemma}\label{lemme:variance_volume}
   There exists a global monotone coupling under which 
   \begin{eqnarray}
     \label{eq:crochet}
 \liminf_{\delta \rightarrow 0}\frac{1}{\delta} \E[ (V_{t+\delta}-V_t)^2|\cF_t] \geq c \frac{V_t}{H^6}    
   \end{eqnarray}
  where $c$ is a constant depending only on the slope $(s,t)$.
\end{lemma}

\begin{proof}[Proof of Theorem \ref{th:wafer}]
  Applying Proposition \ref{prop:ts}, it is enough to give the  upper
  bound \[\tmix^s \leq \cst L^2 H^7\log(LH)\] for the
  mixing time of the synchronous fast dynamics.

  Let $m_i= 2^{-i} H \abs{(G')^*}$ and $T_i = \inf \{ t: V_t \leq m_i
  \}$. Remark that, up to time $T_i$, $V_t$ satisfies the hypothesis
  of Lemma \ref{lemme:temps_martingale} with $M=H \abs{(G')^*}$ and
  $\nu=c m_i/H^6$. Thus we have \[\E [T_{i} - T_{i-1}] \leq H^6\frac{
    2Mm_{i-1}}{c m_i} \leq c' \abs{(G')^*} H^7.\]
  Finally, since $V_t$ takes integer values, the hitting time of $0$ is
  equal to the hitting time of $[0,1/2]$ which is $T_{i_0}$ for $i_0 =\lceil
  \log_2 (2H\abs{(G')^*})\rceil$. We have proved
\begin{equation}
 \E T_{i_0} \leq c' \abs{(G')^*} H^7 \log (H\abs{(G')^*}) =
c'' L^2 H^7 \log (LH).
\end{equation}
Therefore, $\bbP(T_{i_0}>2e c''L^2 H^7 \log (LH))<1/(2e)$, which implies
$\tmix^s\le 2e c''L^2 H^7 \log (LH)$: indeed, under a global monotone coupling,
once maximal and minimal evolutions have coalesced, all the evolutions with arbitrary initial conditions have coalesced too.
\end{proof}

\begin{proof}[Proof of  Lemma \ref{lemme:variance_volume}]
  Let $h^\pm_t$ be the height functions corresponding to the extremal
  evolutions $M^\pm_t$.  Each face contributes at most $H$ to the
  volume difference, so there are at least $V_t/H$ faces where the
  height difference $h^+_t(f)-h^-_t(f)$ is at least $1$. For each of
  those, by Proposition \ref{prop:rotationpossible} there exists a
  face $f^-$ at distance at most $C H$ where again
  $h^+_t(f^-)-h^-_t(f^-)\ge 1$ and a rotation in $M^{+}_t$ that
  decreases the height is possible. We can thus find at least
  $CV_t/H^3$ distinct such faces and each of them is the
  face directly above a non-frozen bead for $M^+_t$ (i.e. a bead that
  can be moved upward in $M^+_t$ via an elementary rotation). Call $B_t$ the set of such beads, $|B_t|\ge C V_t/H^3$.

  The global monotone coupling mentioned in the claim is defined as
  follows. 
We take two mean-one independent Poisson clocks: when the first one
rings we update the beads in even threads, when the second one rings
we update beads in odd threads. The beads are updated as
follows. Suppose for instance that the first clock rings. Then, sample
independently for each transverse edge $e$ and each bead in each even thread a
uniform $[0,1]$ variable $U_{b}(e)$ (any continuous law would work the
same). A bead $b$ in an even thread then chooses the accessible
transverse edge $e$ (given the positions of beads in odd threads) with
the lowest value of $U_{b}(e)$. It is easy to check that this defines a
monotone coupling between evolutions with any possible initial
condition (we emphasize that each evolution uses \emph{the same
  realization} of the $U_b(e)$ variables to determine the outcome of an
update).

We now turn to the estimate of $\nu_t := \liminf_{\delta \rightarrow
  0}\frac{1}{\delta} \E[ (V_{t+\delta}-V_t)^2|\cF_t]$. For any bead
$b$, let $V_t^{(b)}$ denote its contribution to the volume, i.e. the
difference of the labels of the transverse edges occupied by $b$ in
$M^+_t$ and $M^-_t$. Finally let $A^{(+)}$ (resp. $A^{(-)}$) denote
the event that there is an update of even parity and no update of odd
parity (resp. an update of odd parity and no update of even parity)
between time $t$ and $t+\delta$ (each has probability
$\delta+o(\delta)$) and for each bead $b$ let $A^{(b)}$ be $A^{(\pm)}$
according to the parity of $b$. We have (since the occurrence of two
updates has probability of order $\delta^2$)
\begin{gather}\label{eq:variance_infinitesimale}
  \frac{1}{\delta} \E [(V_{t+\delta} - V_t)^2|\cF_t ] =
\E [(V_{t+\delta} - V_t)^2|\cF_t,A^{(+)} ]+\E [(V_{t+\delta} - V_t)^2|\cF_t,A^{(-)} ]+o(1)\\
\label{eq:aseconda}
\ge \var [V_{t+\delta}|\cF_t,A^{(+)} ]+\var [V_{t+\delta} |\cF_t,A^{(-)} ]+o(1)
=
\sum_b \var( V^{(b)}_{t+\delta} | \cF_t, A^{(b)}) + o(1)
\end{gather}
(in the last step, we used the fact that conditionally on $A^{(+)}$, the 
variables $(V^{(b)}_{t+\delta}-V_t^{(b)})$ are independent for different $b$ and 
are zero for $b$ of odd parity).
 For each bead $b$ four cases can occur:
\begin{enumerate}[(i)]
\item The set of transverse edges accessible to $b$ in a single
  update (given the beads of the other parity) is different in
  $M^{+}_t$ and $M^{-}_t$ and, at least for one of them, it consists
  of \emph{strictly more than one}  transverse edge. We let $B_t^{\neq}$ be the set of such beads.
  An elementary computation\footnote{One can check that the worst case is when the intervals of
 transverse edges accessible to $b$ in $M_t^+$ and $M_t^-$ are
 of the form $\{a,\dots,a+k-1\}$ and $\{a,\dots,a+k\}$. In
 this case, after an update $V^{(b)}=0$ with probability $k/(k+1)$ and
 its average is $1/2$ so the variance in question is at
 least $(k/4)/(k+1)\ge 1/8$. }  shows that for such bead
  $\var(V_{t+\delta}^{(b)}|\cF_t,A^{(b)}) \geq \frac{1}{8}$.  

\item The accessible domain for $b$ is the same in $M^{+}_t$ and
  $M^{-}_t$ but its positions in the two configurations are
  different. Let $B_t^=$ be
  the set of such beads. Remark that if the event $A^{(b)}$ occurs, then
  $V_{t+\delta}^{(b)}=0$ almost surely while
  $V_{t+\delta}^{(b)}=V_{t}^{(b)}\ge 1$ if it does not. 

\item The accessible domain and the initial position of $b$ are the
  same in $M^{+}_t$ and $M^{-}_t$. In this case
  $V_{t+\delta}^{(b)}=V_{t}^{(b)}=0$ conditionally on $A^{(b)}$, so these beads give no contribution
to the volume variation.

\item The accessible domain for $b$ has only a single edge in both $M^+_t$ and $M^-_t$. In this case there is no movement possible for $b$ until threads of the opposite parity are updated, so $b$  makes again no contribution conditionally on $A^{(b)}$.

\end{enumerate}
Remark that the set $B_t$ introduced above is included in 
$B_t^{\neq}\cup B_t^=$, so we have $\abs{B_t^=} +
\abs{B_t^{\neq}} \geq C V_t/H^3$. Indeed, if $b\in B_t$ than it can be moved in $M_t^+$, which excludes case (iv), and $V^{b}_t\ne 0$, which excludes case (iii).

Suppose that $\abs{B_t^{\ne}} \geq \abs{B_t^=}/(\alpha_{s,t}H)$, with $\alpha_{s,t}$ a slope-dependent constant to be determined later; for $b \in
B_t^{\neq}$ we have $\var(V_{t+\delta}^{(b)}|\cF_t,A^{(b)}) \geq
\frac{1}{8}$ so the right-hand size of
\eqref{eq:aseconda} gives, taking the limit $\delta\to0$, \[\nu_t \geq
\frac{\abs{B_t^{\neq}}}8 \geq  a_{s,t}\frac{V_t}{H^4}\] for some positive
$a_{s,t}$.  

Suppose on the contrary that
$\abs{B_t^{\neq}} \leq \abs{B_t^=}/(\alpha_{s,t}H)$ and (by symmetry) that at least half the
beads $b \in B_t^=$ are on even threads. After an even update they each
contribute $V_{t+\delta}^{(b)} - V^{(b)}_t \leq -1 $ so $V_{t+\delta}
- V_t \leq -\abs{B_t^=}/2 + c_{s,t}H\abs{B_t^{\neq}} \leq -  |B_t^=|/4$ 
(we used the fact that, for every $b$ and in particular for $b\in B_t^{\neq}   $,
$  V^{(b)}_{t+\delta} - V^{(b)}_t  \leq c_{s,t}\,H$ due to the floor/ceiling at mutual 
distance $H$ (cf. Proposition \ref{prop:strictdecr}) and we chose $\alpha_{s,t}=4c_{s,t}$).
As a consequence, in the limit $\delta\to0$ \eqref{eq:variance_infinitesimale} gives
$\nu_t \geq b_{s,t} V_t^2/H^6$ in this case, for some other positive constant $b_{s,t}$. The conclusion follows from $\min(V_t/H^4,V_t^2/H^6)\ge V_t/H^6$.
\end{proof}

\begin{remark}
  The power $H^6$ in the lemma is clearly far from optimal. A finer
  analysis of the contribution of $B_t^{\neq}$ and $B_t^=$ would
  probably improve the power to $H^3$. We do not follow this route
  because ultimately the precision of the upper bound is limited by
  the equilibrium estimate of Theorem \ref{th:clt} and also because
  even the bound $\abs{B_t^=} + \abs{B_t^{\neq}} \geq C V/H^3$ is certainly
  far from optimal for a typical configuration.
\end{remark}


\subsubsection{A mean curvature motion approach}
\label{sec:meancur}
Given Theorem \ref{th:wafer} and Theorem \ref{th:stimeflutt} on the
equilibrium height fluctuations, the proof of the bound $\tmix\le
c(\epsilon)L^{2+\epsilon}$ is essentially identical to the proof that
$\tmix=O( L^2(\log L)^{12})$ for the dynamics on the hexagonal lattice
with almost-planar boundary conditions, see \cite[Th. 2]{CMT}. Indeed,
Theorem \ref{th:wafer} plays the role of \cite[Prop. 3]{CMT} while
Theorem \ref{th:stimeflutt} replaces \cite[Th. 1]{CMT}.  Therefore,
below we will only recall the main ideas, and we skip all details.

\begin{remark}
  The reason why in Theorem \ref{th:tmixU} we get the $L^\epsilon$
  correction to the mixing time instead of a factor $(\log L)^{12}$ as
  in \cite{CMT} is that the fluctuation estimates of Theorem
  \ref{th:stimeflutt} are a bit weaker than those of
  \cite[Th. 1]{CMT}: since anyway the exponent $12$ is certainly
  non-optimal (we conjecture the correct value to be $1$, cf. the
  Introduction), we have not tried to refine Theorem
  \ref{th:stimeflutt} (for instance, one might try to control the
  exponential moments of the height fluctuations).
\end{remark}
The first step is the following (cf. \cite[Prop. 2]{CMT}):

\noindent{\bf Step 1}
\emph{If the the height function of the  initial condition $\xi\in
\Omega_{G',m}$ is within distance $L^{\epsilon/10}$ from the almost-planar
configuration $m$ (i.e. if $|\xi(f)-h_m(f)|\le L^{\epsilon/10}$ for
every $f\in (G')^*$), then (for any given $C$) for all times smaller
than $L^{C}$ the height function stays within distance
$2L^{\epsilon/10}$ from $m$, except with probability $O(L^{-C})$.}

\medskip
This means that, until time $L^C$, the dynamics is essentially identical
to a dynamics with floor/ceiling at mutual distance $O(L^{\epsilon/10})$.
 Together with Theorem \ref{th:wafer} this
implies:

\noindent{\bf Step 2}
\emph{Again if the initial condition $\xi$ is within distance
$L^{\epsilon/10}$ from $m$, after time $L^{2+\epsilon}$ the law of the
configuration has small variation distance from equilibrium.  
}

\medskip
Therefore, to prove $\tmix\le c(\epsilon)L^{2+\epsilon}$ it is
sufficient to prove:
\begin{claim}
At time $c'(\epsilon)L^{2+\epsilon}$ the
evolutions started from maximal/minimal configurations are with high
probability within distance $L^{\epsilon/10}$ from $m$ (this is the
analog of \cite[Prop. 1]{CMT}).  
\label{claim96}
\end{claim}

Consider for instance the evolution $h^{max}_t$ started from the
maximal initial configuration $h_{max}\in \Omega_{m,G'}$ and assume for
simplicity of exposition that the slope of the quasi-planar boundary condition $m$
is $(s,t)=(0,0)$. Let $C_u$ be a spherical cap whose height is $u$ and
whose base is a disk $D$ of radius $\rho_L=L\log L$ with the finite graph $G'$
approximately at its center (recall that the diameter of $G'$ is of
order $L\ll \rho_L$). Call $R_u$ the radius of curvature of $C_u$,
which satisfies $(2R_u-u)u=\rho_L^2$ and let $\psi_u(f)$ denote the
height of $C_u$ above a face $f$ which is inside $D$. Then, the key to Claim
\ref{claim96} is:
\begin{claim}
With overwhelming probability, the height
function $(G')^*\ni f\mapsto h^{max}_t(f)$ is below the deterministic function
$(G')^*\ni f\mapsto \psi_{u_n}(f)$ for all times $t\in[t_n,L^3]$,
where $u_n=\rho_L-n$ and the deterministic time sequence $t_n$ is
defined by
\[
t_0=0, \quad t_n=t_{n-1}+R_{u_n}L^{\epsilon/2}, \quad n\le M:=\rho_L-L^{\epsilon/10}.
\]
\label{claim97}
\end{claim}
Indeed, it is easy to verify that $R_{u_n}\sim \rho_L^2/(2u_n)$ and (similarly to \cite[Eq. (6.14)]{CMT})
that $t_M=O(L^{2+\epsilon})$.  Therefore, if we show the above claim
for $n$ up to $M$, we deduce that at some time $O(L^{2+\epsilon})$ the
configuration is within distance $u_M=L^{\epsilon/10}$ from the flat
configuration $m$, and Claim \ref{claim96} follows.

For $n=0$, the statement of Claim \ref{claim97} is true (deterministically, not just with high
probability, since the maximal height at a face $f\in G'$ is of order
$L\ll \rho_L$). Suppose we want to deduce claim $n+1$ from claim $n$,
and look for definiteness only at a face $f$ at the center of the disk
$D$.  Consider a disk $D_{n+1}$ centered at $f$, of radius
$R_{u_n}^{1/2}L^{\lambda\epsilon}$ with $\lambda$ to be chosen later: by
monotonicity, given the claim at step $n$, we can replace the
evolution $h^{max}_t$ restricted to $D_{n+1}$, in the time interval
$[t_n,L^3]$, by an evolution where:
\begin{itemize}
\item the configuration outside $D_{n+1}$ is frozen and equals some
  height function  which, on the boundary of $D_{n+1}$,
  is within distance $O(1)$ from the function $\psi_{u_n}(\cdot)$;
note that for $f'$ at the boundary of  $D_{n+1}$ one has  $\psi_{u_n}(f')\approx
u_n-(1/2)L^{2\lambda\epsilon}$;

\item the ``initial'' height function at time $t_{n}$ in $D_{n+1}$ approximates
  within $O(1)$ the function $\psi_{u_n}(\cdot)$.
\end{itemize}
By Step 2, the time such dynamics takes to reach equilibrium is
$O(R_{u_n}L^{20 \lambda \epsilon})\ll t_{n+1}-t_n$ if $\lambda<1/40$, so
that at time $t_{n+1}$ the configuration is essentially at
equilibrium (with the above specified boundary conditions around $D_{n+1}$). Next, elementary geometry and Theorem \ref{th:stimeflutt}
shows that, at equilibrium, the height function at $f$ is with
overwhelming probability lower than $u_n-1=u_{n+1}$: this is because, as we
remarked above,
the boundary height around the boundary of $D_{n+1}$ is approximately
$u_n-(1/2)L^{2\lambda\epsilon}$. We deduce therefore that, with high probability,
$h^{max}_t(f)\le u_{n+1}= \psi_{u_{n+1}}(f)$ for $t\in[t_{n+1},L^3]$
and a similar argument works for any other face $f\in (G')^*$. The
claim at step $n+1$ is then proven.

\subsubsection{Gaseous phases and entropic repulsion}

\label{sec:entro}
Theorem \ref{th:tmixU} has been formulated for three specific - though
quite natural - graphs. While, as explained in Remark
\ref{rem:layers}, our method can be extended to a wider class of
graphs, there is no hope it works for a general infinite bipartite
graph. We would like to convince the reader there is a good reason for
this.  One of the main steps of our argument (cf. Section
\ref{sec:drift_volume}) is to prove that, given two height functions
$h^1\le h^2$, after a step of the (fast) dynamics the mutual volume
has not increased in average. The proof of Proposition
\ref{prop:volume_drift} shows that this is true (for the hexagonal,
square and square-hexagon graphs), \emph{independently of the boundary
conditions} (in particular, for almost-planar conditions, independently
of the slope $(s,t)$) \emph{and independently also of the presence of
floor/ceiling constraints}. While obtaining the mixing time upper bound of
order $L^{2+\epsilon}$ requires considerable extra work, the volume decrease result
 implies rather directly \cite{LRS}
that the mixing time is at most polynomial in $L$, since (i) the
maximal volume difference between two configurations is of order $L^3$
and (ii) the ratio of mixing times of the fast and original dynamics
is at most polynomial in $L$.

Now take for instance the square-octagon graph, with almost-planar
boundary conditions of slope $(0,0)$. As we remarked in Section
\ref{sec:phaseclass}, in this situation the infinite-volume Gibbs
measure $\mu_{0,0}$ corresponds to a ``gaseous'' (or ``rigid'') phase:
the height $h(f)$ at a face $f$ has bounded variance and the random
variables $h(f),h(f')$ are essentially independent for $f,f'$ far
away. This is very reminiscent of the situation in the classical
$(2+1)$-dimensional Solid-on-Solid (SOS) interface model \cite{SOS} at low
temperature $1/\beta$ \cite{BW}. Let us just recall that the SOS
model describes an interface with integer heights $\phi(x)$, labelled
by $ x\in \bbZ^2$, with measure proportional to
$\exp(-\beta\sum_{|x-y|=1}|\phi(x)-\phi(y)|)$ and zero boundary
conditions $\phi\equiv0$ around a $L\times L$ box.  Recently, it was
proved in \cite{SOSdyn} that, when a floor at height zero is present
(i.e. when heights are constrained to be non-negative) the mixing time
of the Glauber dynamics for the SOS model is exponentially large in
$L$.  This effect is due to the rigidity of the interface \emph{and}
to the presence of the floor, which pushes the interface to a height
of order $\log L$ (entropic repulsion \cite{BEF}). If a ``gaseous'' phase does behave like a SOS interface, which is very
likely, then that the dimer dynamics for the square-octagon graph with
almost-planar b.c. of $(0,0)$ slope and floor at height zero has also
a mixing time growing exponentially with $L$. As a consequence, for
the reasons exposed above, the volume decrease cannot hold as stated in proposition \ref{prop:volume_drift} (i.e. for general floor/ceiling constraint and boundary condition) for the square-octagon graph (it cannot be true that the mutual volume
between two arbitrary height functions decreases on average, under
some reasonable ``fast dynamics'', for the square-octagon graph with
general boundary conditions and floor/ceiling constraints). If on the
other hand the volume decrease holds only for particular boundary conditions and
floor/ceiling, then the mathematical mechanism for that must be considerably more subtle than in Proposition \ref{prop:volume_drift}.

\subsection{Mixing time lower bound}\label{sec:lower_bound}

In this section we establish the mixing time lower bound of Theorem
\ref{th:tmixU}.  Thanks to Proposition \ref{prop:ta}, it is enough to show
that at time $T_0=\epsilon L^2$ the asynchronous dynamics started from the
maximal configuration is still at variation distance at least
$1/(2e)$ from equilibrium.

The strategy is the following.  First, we define (for each of the
three types of graph in question) a special (very non-planar) boundary
condition $p\in\Omega$ and finite domain $W_L\subset G$ of diameter of
order $L$ for which it is easy to prove that, starting from the
maximal configuration, the drift of the volume is lower bounded by $-c
L$. Therefore, after time $T_0=\epsilon L^2$ the eroded volume is at
most $c \epsilon L^3$ and the configuration (call it $\tilde M_{T_0}$)
is still away from its (non-flat) equilibrium shape.  Next, a
monotonicity/coupling argument allows to deduce that, again at time
$T_0$, the configuration $M_{T_0}$, evolving this time in our original domain $G'$
with the almost-planar boundary condition $m$ we are interested in, is
above $\tilde M_{T_0}$ and that it is also far from its typical (flat,
this time) equilibrium shape.

\subsubsection{Pyramids}\label{sec_existence_pyramide}


\begin{figure}[htp]
   \centering
   \includegraphics[width=5cm]{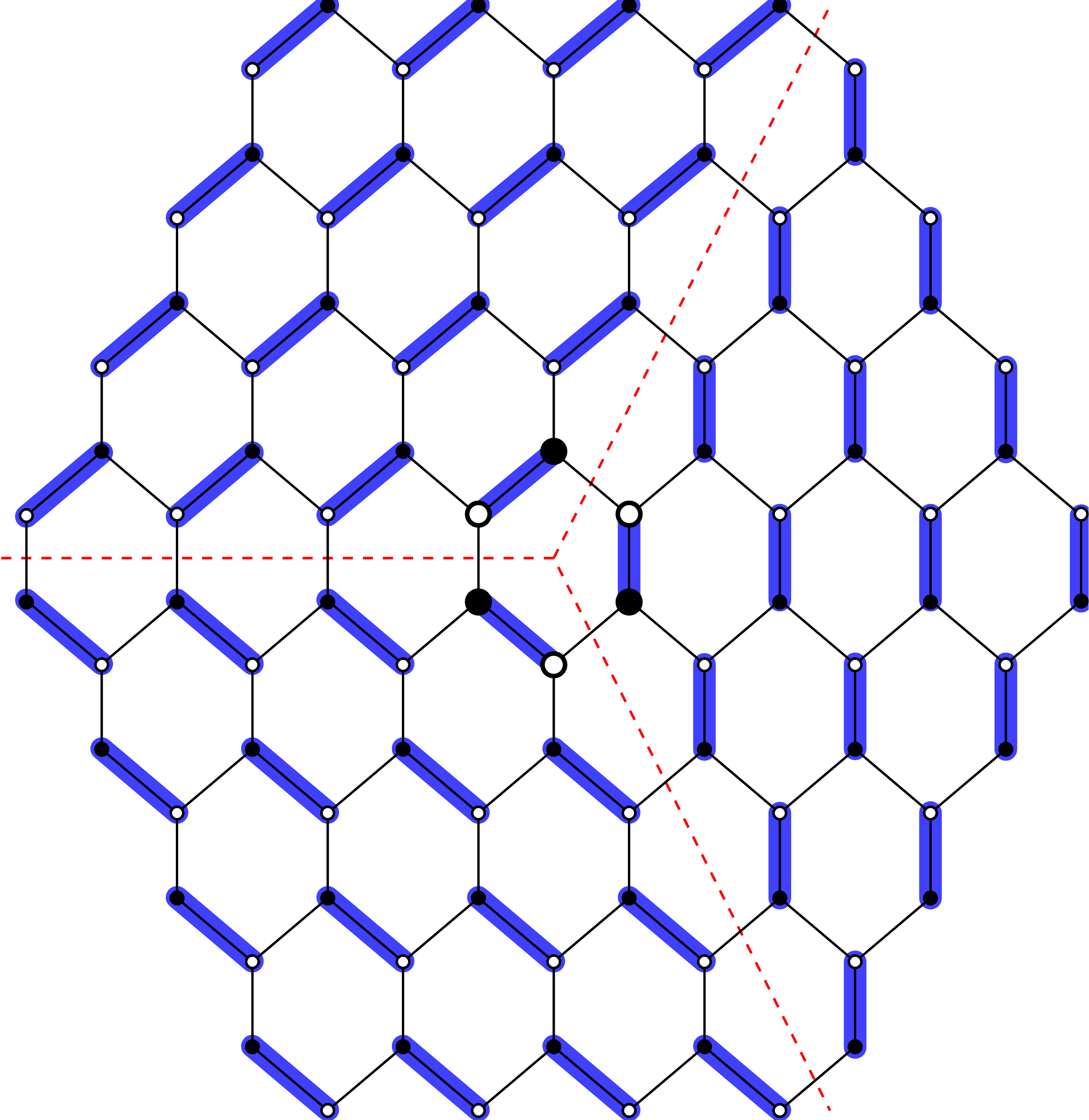}
  \includegraphics[width=5cm]{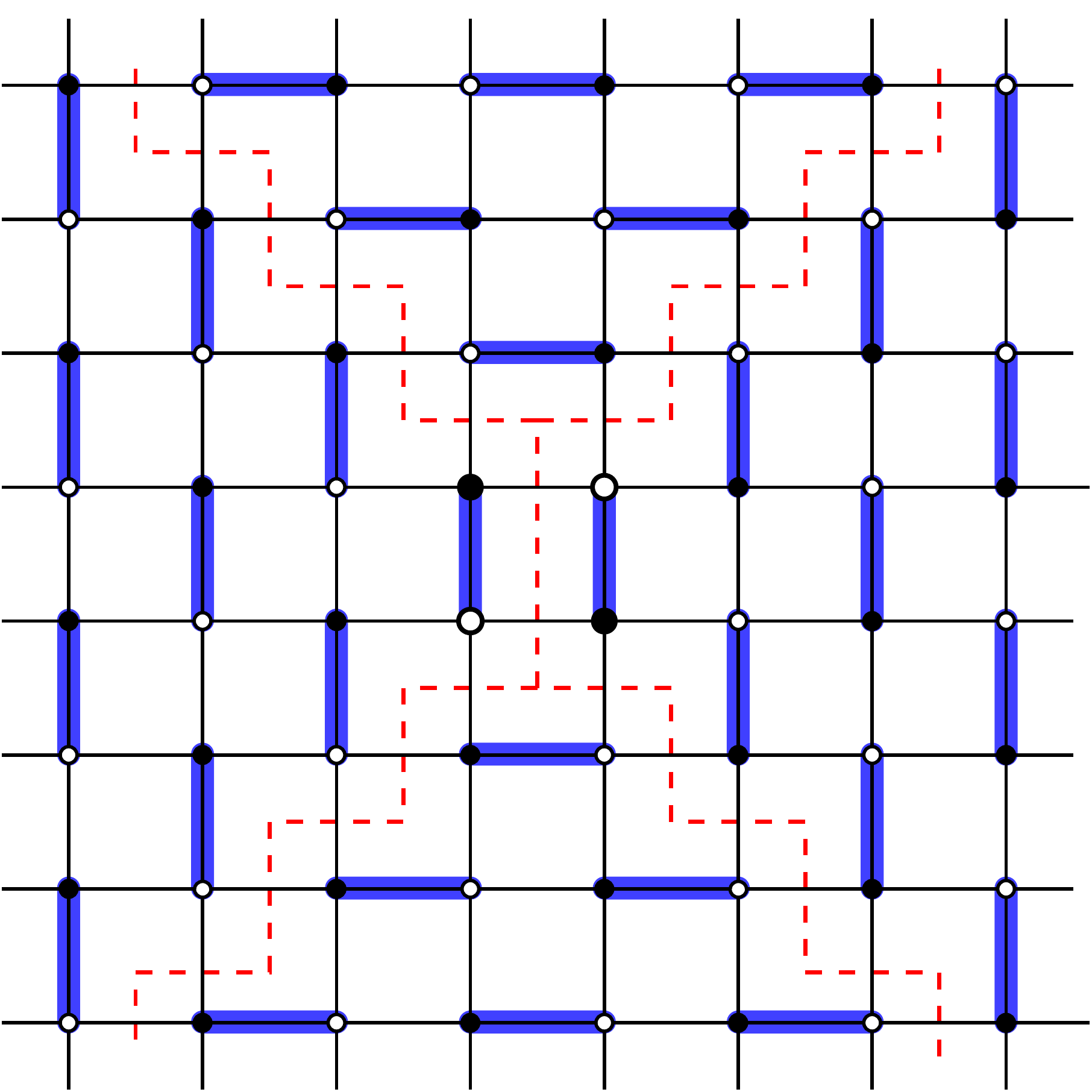}
\includegraphics[width=5cm]{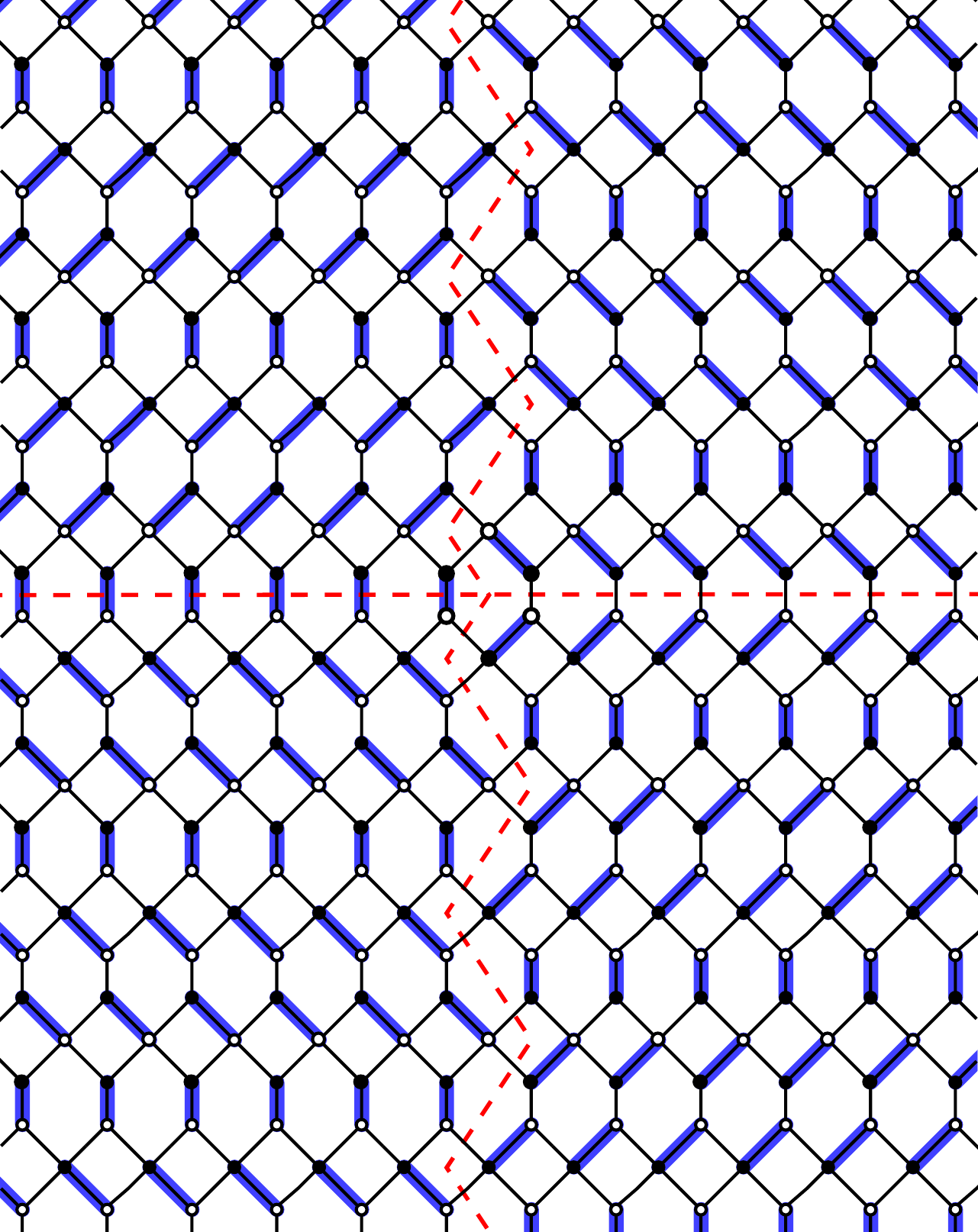}
   \caption{``Pyramids'' for the square, hexagon and square-hexagon
     graphs. The matchings are intended to extend over the whole
     infinite graph. The dotted
     lines separate the domains tiled with different periodic matchings of extremal slope.
     \label{fig:pyramide_matching}}
\end{figure}

For each of the three graphs (square, hexagon, square-hexagon)
consider the special matching $p\in\Omega$ (``$p$'' for ``pyramid'' for
reasons to become clear soon) of the whole $G$, defined through Fig.
\ref{fig:pyramide_matching}. Note that in all three cases $G$ is
divided into a finite number of infinite domains, separated by dotted
lines (three domains for the hexagon graph and four for the square and
square-hexagon graph): each domain corresponds to a vertex in the
Newton polygon $N$ of $G$ and it is tiled with the periodic matching
corresponding to that vertex
(cf. Fig. \ref{graphes_et_polygones_newton}). The central face $f_0$
is assumed to contain the origin of $\bbR^2$ and we fix the height
there to $0$. The large-scale height function $\bbR^2\ni x\mapsto
H(x)$, obtained by rescaling the lattice spacing and the heights by
$1/L\to0$ while keeping $f_0$ at the origin of the plane can be
described as follows. For each vertex $v$ of $N$, take the plane of
the corresponding slope which contains the origin of $ \bbR^3$, and
let $s_v$ be the half-space below it. The intersection of all the
$s_v$ with $v$ ranging over the vertices of $N$ is clearly a pyramid
$\Pi$ with vertex at the origin of $ \bbR^3$. The boundary of $\Pi$
gives the height function $H(\cdot)$.  It is possible to prove (but we
will not need this directly) that the discrete height function
associated to matching $p$ is given by $h(f)=-D(f,f_0)$, cf. Section
\ref{sec:characterisation_lineaire}.  This observation could be used
to build ``pyramids'' in a systematic way, for other graphs.

\begin{remark}
  For the hexagonal lattice, the pyramid $p$ just corresponds (in
  terms of stepped surfaces, cf. Fig. \ref{exemple_chemin_pente}) to
  the surface of the corner of an infinite cube with vertex at the
  origin of $\bbR^3$.
\end{remark}

\begin{figure}[htp]
   \centering
   \includegraphics[width=5cm]{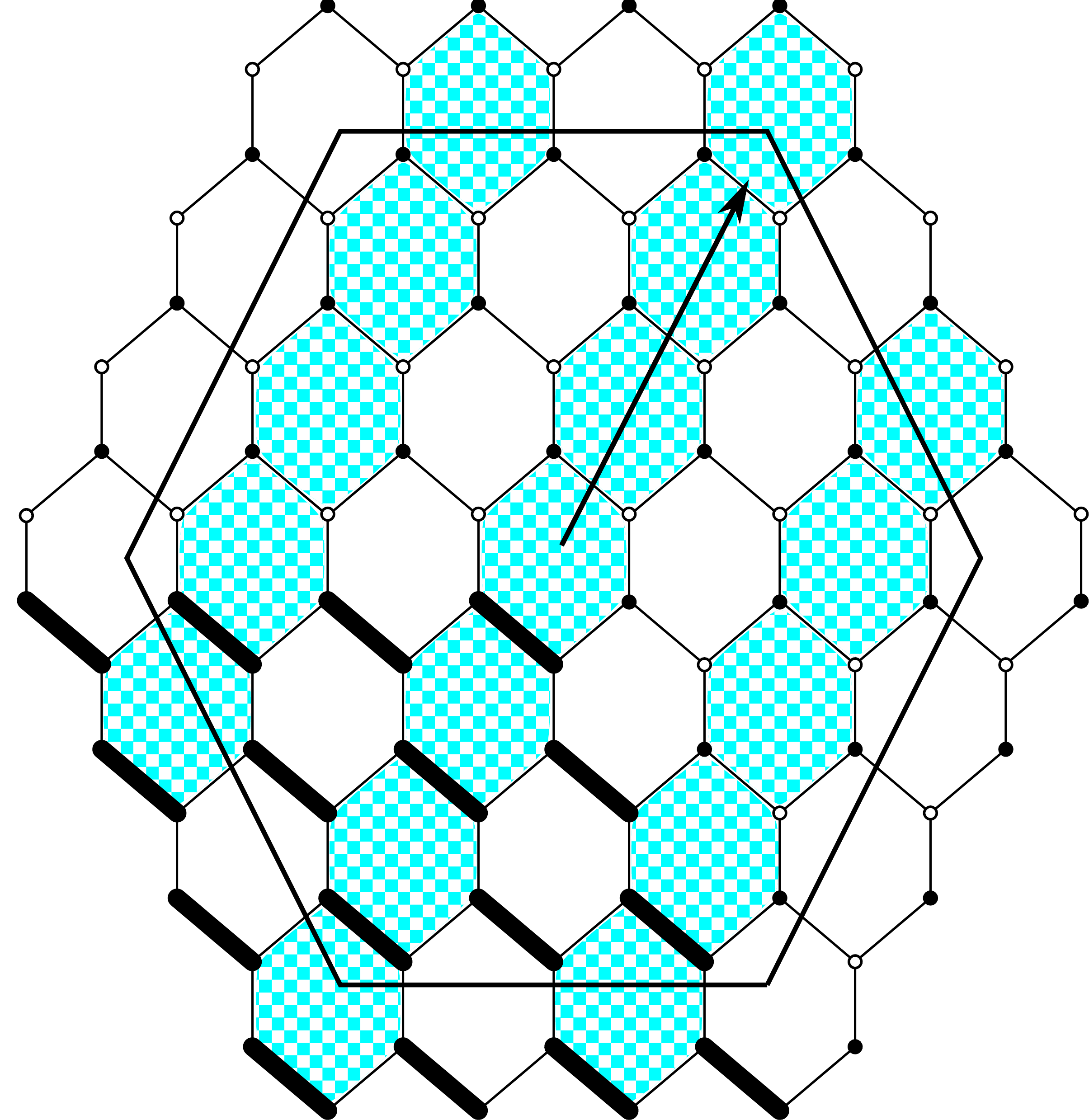}
 \includegraphics[width=5cm]{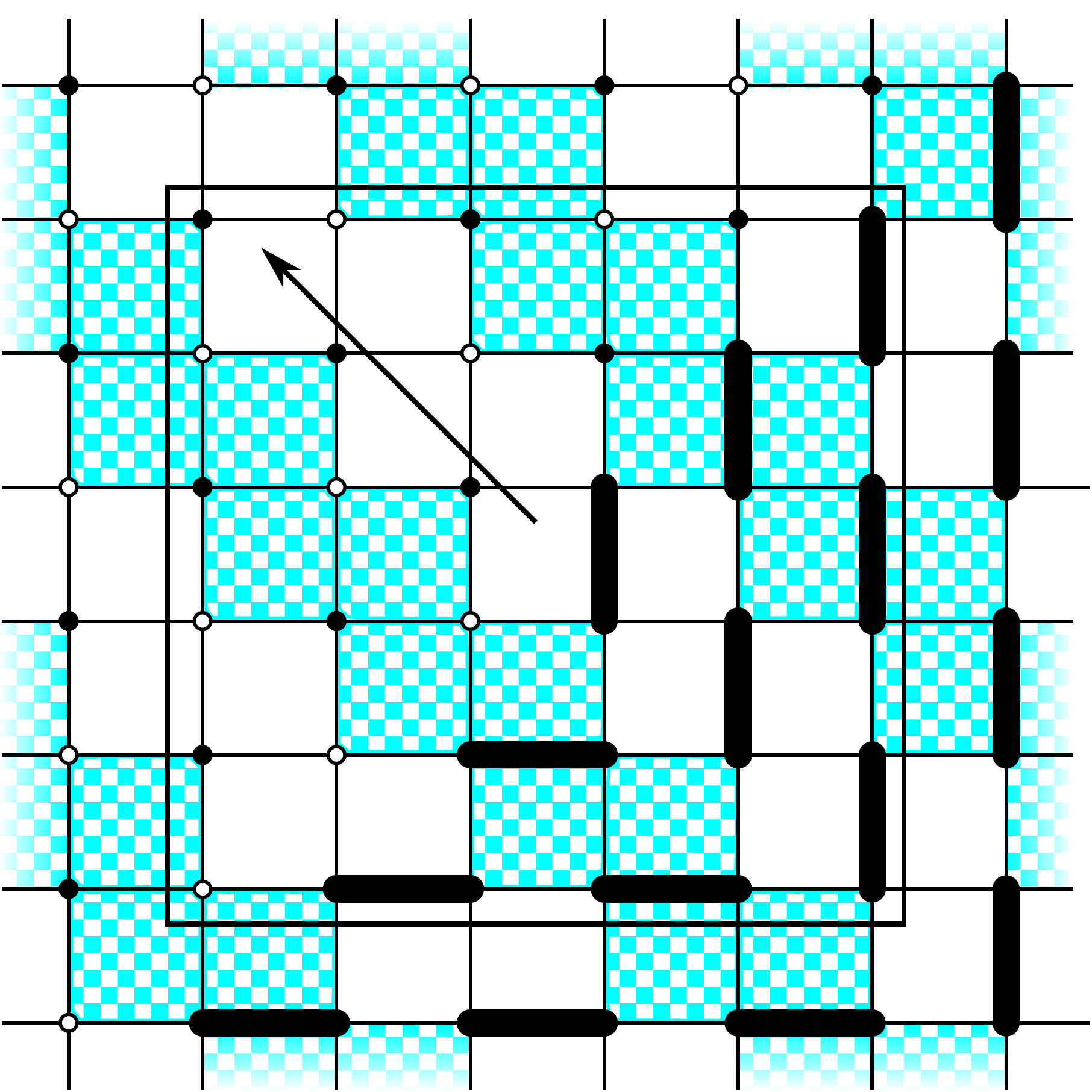}
    \includegraphics[width=5cm]{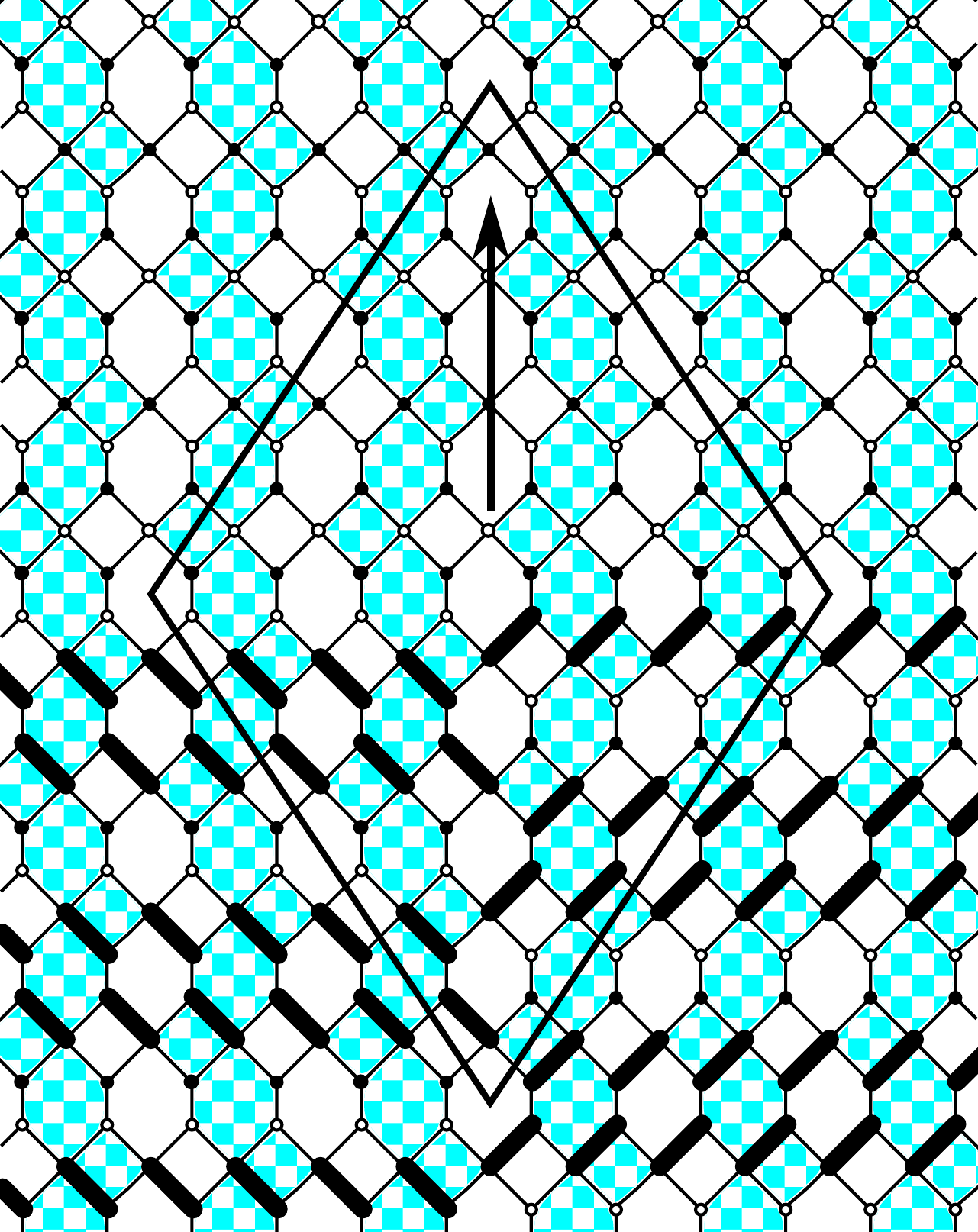}
   \caption{The sub-graph $W_L$ and the associated beads of configuration $p$.
    Threads are in light blue and white and beads are
     in solid black. Arrows mark the orientation of threads.
}
    \label{fig:pyramide_billes}
\end{figure}

Next, we need to introduce a finite sub-graph $W_L\subset G$, with the
face $f_0$ in the center. This is defined through Fig.
\ref{fig:pyramide_billes}: for the hexagonal lattice this is the
portion of $G$ included in a $(2L+1)\times (2L+1)\times (2L+1)$
hexagon, for the square lattice it is a $(2L+1)\times (2L+1)$ square
and for the square-hexagon lattice it is the portion of $G$ contained
in a suitable lozenge (delimited by a full line in the picture) whose
diagonals contain $2L+1$ hexagons. Note that the maximal configuration
in $\Omega_{p,W_L}$ is just $p$: indeed, just observe that none of the
beads in $W_L$ can be moved to a lower transverse edge in the same
thread.
\begin{proposition}
\label{prop:driftpiramide}
Consider the asynchronous fast dynamics on the finite graph $W_L$
defined above, with ``pyramid'' boundary condition $p$. Let
$M^{max}_t$ (resp. $M^{eq}_t$) denote the state at time $t$ of the
dynamics started from the maximal configuration, which is nothing but
$p$ (resp. from the equilibrium uniform measure $\pi_{p,W_L}$) and let
$\cF_t$ be the sigma-algebra generated by $\{M^{max}_s,M^{eq}_s,s\le
t\}$. Remark that $M^{eq}_t$ is stationary. Let $V_t = \sum_{f\in W_L}
[h_{M^{max}_t}(f) - h_{M^{eq}_t}(f)]$. Then, for every $t \geq t'$
\[
   \E \bigl[ V_t - V_{t'} | \cF_{t'} \bigr] \geq -C L(t-t')
\]
for a certain constant $C$.
\end{proposition}
\begin{proof} The two evolutions can be coupled so that the one
  started from $p$ always dominates the stationary one. Therefore, it
  is sufficient to prove that, given two height functions $h^-\le
  h^+$, the initial-time drift of their volume difference, 
see \eqref{eq:frift},  is lower
  bounded by $-C L$. Let $\partial^- W_L$ denote the set of faces in
  $W_L$ at graph-distance at most $r$ from the exterior of $W_L$ and assume
  the following
\begin{claim}
\label{claimvendredi} There exists some finite $r$ such that
the initial-time volume drift is zero whenever $h^+$ and $h^-$
differ at a single face $f\in W_L\setminus\partial^- W_L$.
\end{claim}
As in the proof of Proposition \ref{prop:volume_drift}, take a
sequence of configurations $h_{(0)},\dots,h_{(k)}$ such that
$h_{(0)}=h^+,h_{(k)}=h^-$ and $h_{(i)}$ is obtained by $h_{(i-1)}$ via
a rotation that decreases the height at some face $f$. Along the
sequence $\{h_{(i)}\}_i$, each face in $\partial^-W_L$ can be rotated
at most a number $c(r)$ of times (since its height difference w.r.t. a
face outside $W_L$ at distance $r$ can take at most $c(r)=2r+1$ 
values). Then write the volume difference between
$h^+$ and $h^-$ as a telescopic sum of volume differences between
$h_{(i)}$ and $h_{(i-1)}$: the terms corresponding to a rotation in
$W_L\setminus \partial^-W_L$ give zero drift (cf. Claim
\ref{claimvendredi}) and those with rotation in $\partial^-W_L$ give a
contribution lower bounded by $-1$ (cf. Remark
\ref{rq:borne_inf_drift}).  Then, one gets that the volume drift is
lower bounded by $-c(r)$ times the number of faces in $\partial^-
W_L$, which is of order $L$, and Proposition \ref{prop:driftpiramide}
follows.
\end{proof}

\begin{proof}[Proof of Claim \ref{claimvendredi}]
  
  In order to follow, the reader should have in mind the proof of
  Proposition \ref{prop:volume_drift}: there we proved that the initial-time volume
  drift is zero except for ``boundary effects'', and here we show that
  indeed the boundary effects are not there sufficiently far away from
  the boundary.  As in the proof of Proposition
  \ref{prop:volume_drift}, we consider for definiteness only the
  square-hexagon lattice, the other two cases being very similar. Note
  that there are $2L+1$ threads intersecting $W_L$ (we label them
  $\gamma_i,i=-L,\dots,L$ from left to right) and that thread
  $\gamma_i$ contains $L+1-|i|$ beads in $W_L$ (we label them $b^i_j,
  j=1,\dots,L+1-|i|$ from the lowest to the highest with respect to
  the orientation of the thread, see Fig. \ref{fig:drift0}(a) for a
  schematic drawing). This geometric structure is essential for the
  proof of the Claim and, as the reader can check from
  Fig. \ref{fig:pyramide_billes}, it is common to the ``pyramids'' of
  the three types of graphs.  The effect of the boundary of $W_L$ is
  simply that the beads of thread $\gamma_i$ are constrained to stay
  strictly below some transverse edge $e^+_i\in\gamma_i$ and strictly
  above some transverse edge $e^-_i\in\gamma_i$ (such edges are
  surrounded by circles in Fig. \ref{fig:drift0}(a)).

By symmetry, we assume that the face $f$ belongs to $\gamma_i$ for
some $-L+r\le i\le 0$ (the case $-L\le i<-L+r$ is excluded otherwise
$f\in\partial^-W_L$) and assume for definiteness that $f$ is a
hexagonal face (the argument is essentially identical when $f$ is a
square). Then, $h^+$ is obtained from $h^-$ by moving a bead $b^i_j$ one
``transverse edge'' lower along thread $\gamma_i$ (with the
notations of the proof of Proposition \ref{prop:volume_drift} (Case
1), $b=b^i_j$ moves from $e_1$ to $e_2$).  We need to show that, if
$f\in W_L\setminus \partial^-W_L$, when
threads of the opposite parity than $i$ are updated while the threads
with the same parity as $i$ are frozen, thread $\gamma_{i-1}$
contributes exactly $1/2$ to the change of volume (the same holds for
$\gamma_{i+1}$).

\begin{figure}[htp]
   \centering
   (a)\includegraphics[width=4cm]{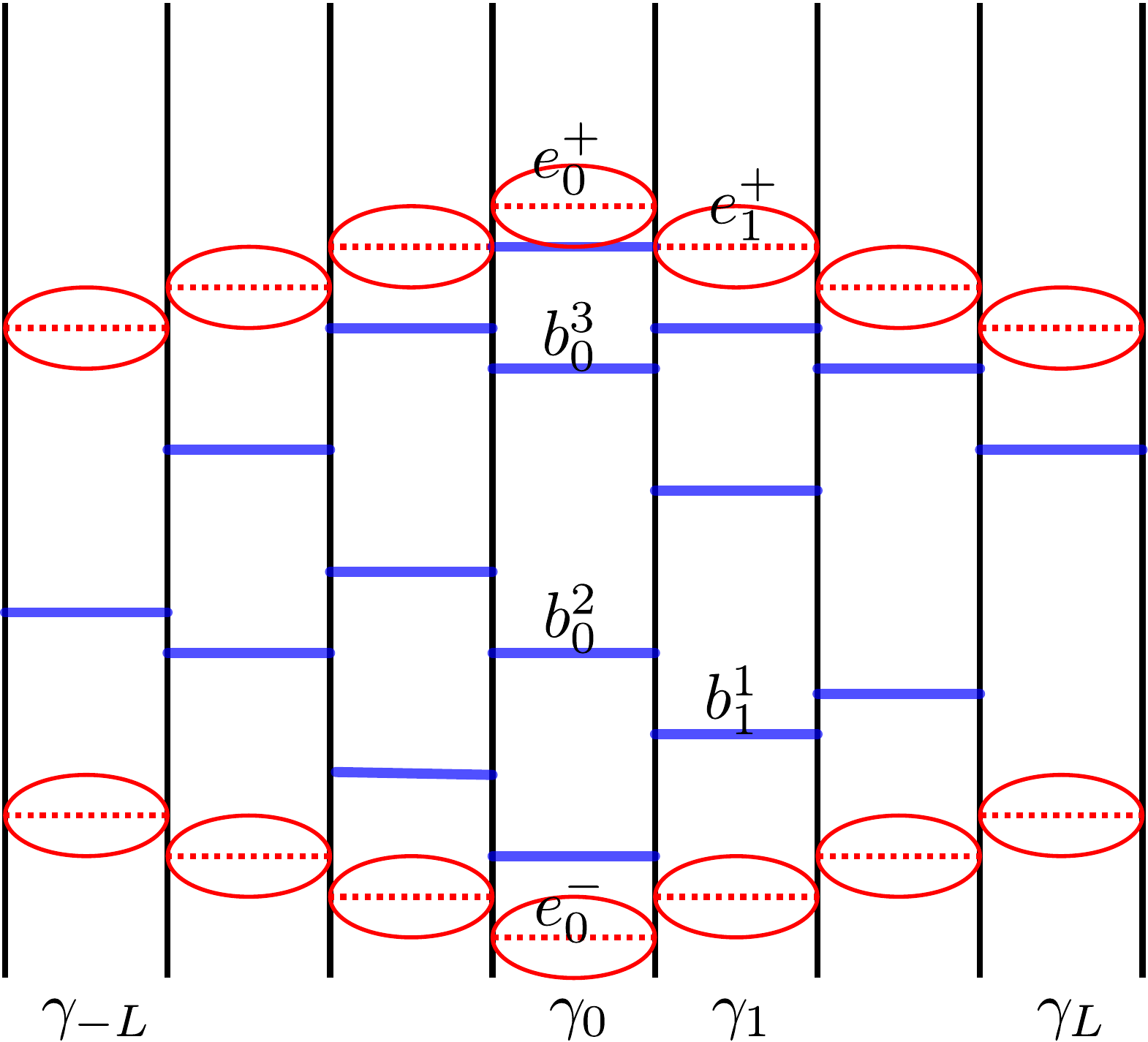}
\medskip
(b)\includegraphics[width=4cm]{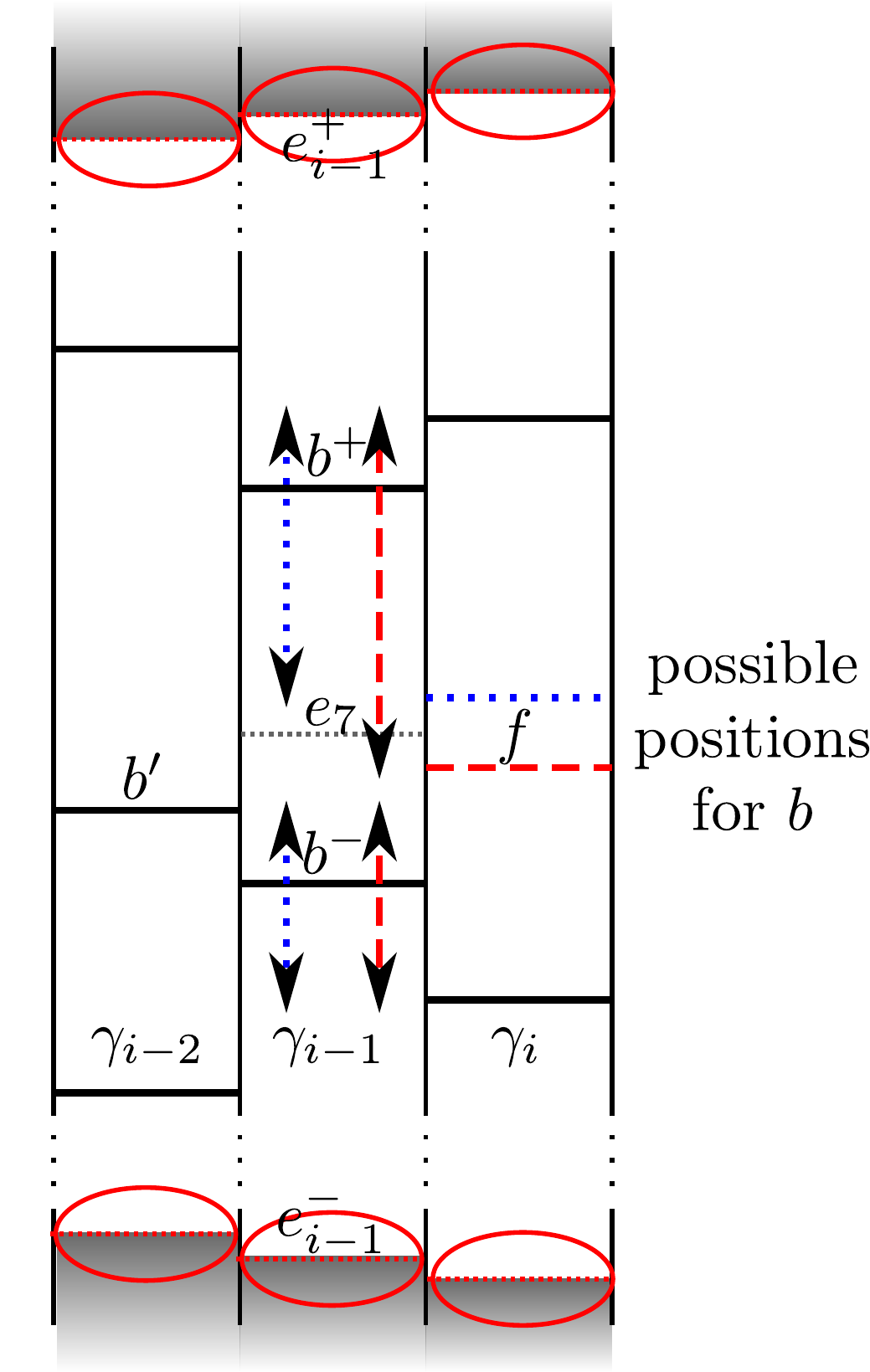}
\medskip
(c)\includegraphics[width=4cm]{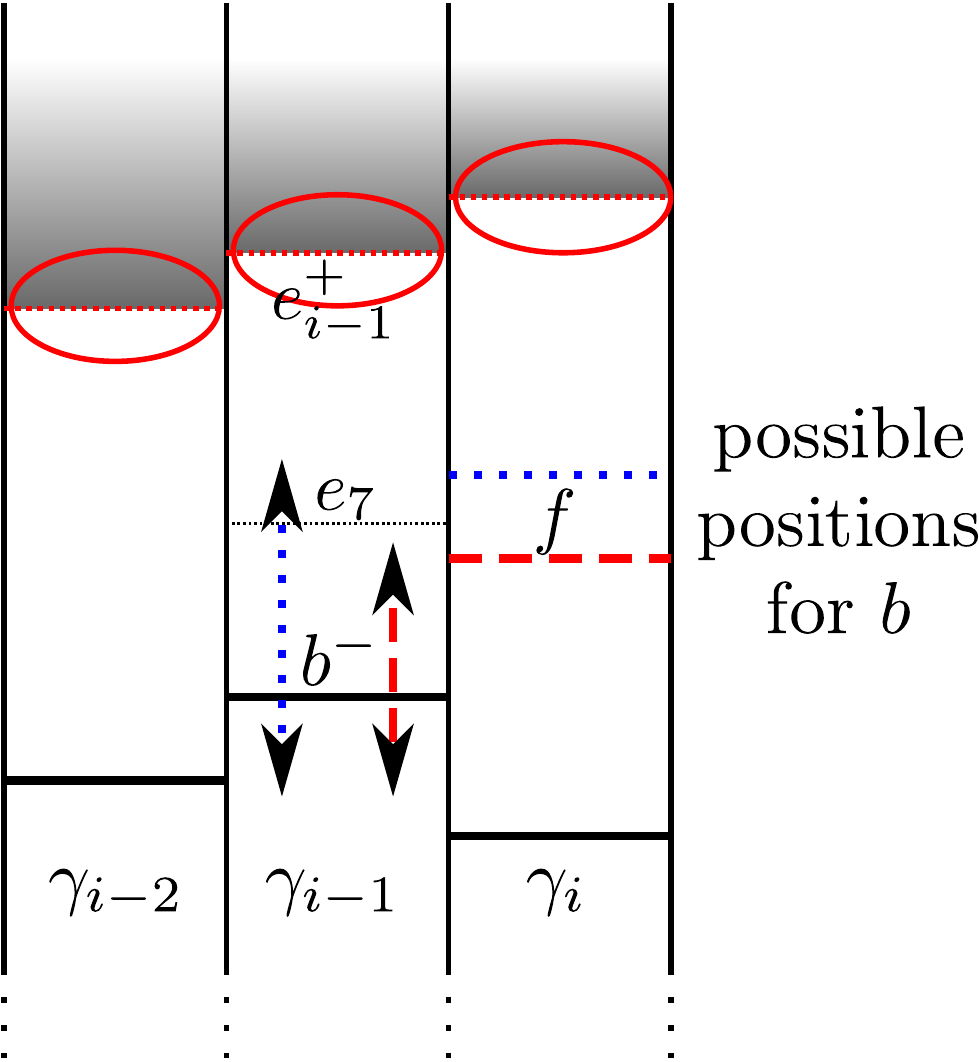}
  \caption{ (a): A schematic view of
    the threads and beads in $W_L$ (here, $L=3$). (b):  The case
    $1<j<L+1-|i|$. Here, $i=0$ and $j=3$. Boundary conditions only prevent
    beads from entering the gray region outside $W_L$, so they cannot
    prevent $b^+$ from moving down to $e_7$.
(c): The case $j=L+1-|i|$. If $f\in \partial^-W_L$ then $e_7$ can be
at $e^+_{i-1}$, in which case it is not an allowed bead position;
otherwise, $e_7$ is in $W_L$ (since it is at distance of order $1$
from $f$) and $b^-$ can reach such position.
\label{fig:drift0}
 }
\end{figure}

We distinguish two cases, represented respectively in Figures \ref{fig:drift0}(b) and \ref{fig:drift0}(c):
\begin{itemize}
\item $1<j<L+1-|i|$.  Recall that from Section \ref{sec:drift_volume} that, according to the position
of $b^{i-2}_{j-1}$ (which was called $b'$ there), either $b^{i-1}_j$
(which was called $b^+$) has one more available transverse edge
(called $e_7$) when the configuration before the update is
$h^+$ rather than $h^-$, or otherwise $b^{i-1}_{j-1}$ (which was
called $b^-$) has one more  available transverse edge (again $e_7$)
starting from $h^-$ rather than from $h^+$. Say that the former is the
case. As discussed in Remark \ref{rq:borne_inf_drift}, the volume change 
due to thread $\gamma_{i-1}$
is $1/2$ unless
the boundary conditions prevent $b^+$ from moving down to $e_7$, i.e. if
$e_7$ is not higher than $e^-_{i-1}$: this
however cannot be the case, since $e_7$ is clearly higher than
$b^-=b^{i-1}_{j-1}$  which is itself higher than $e^-_{i-1}$
(here we use that $j-1\ge1$,
i.e. the bead $b^{i-1}_{j-1}$ is in $W_L$). 

\item $j=L+1-|i|$ (or $j=1$, by symmetry). The argument is slightly
  different since this time neither $b^+=b^{i-1}_j$ nor
  $b'=b^{i-2}_{j-1}$ exist (since $j>L+1-|i-1|$), or equivalently we
  can imagine that $b^+,b'$ are higher than $e^+_{i-1},e^+_{i-2}$
  respectively (i.e. they are outside $W_L$). Since the edge $e_3$ of
  Fig. \ref{calcul_drift_hexagone} is at distance of order $1$ from
  $f$, we deduce that if $f$ is at distance larger than some finite
  $r$ from the boundary of $W_L$ then $b'$ is above $e_3$.  In this
  case, from Section \ref{sec:drift_volume}, we get that $b^-$ has one
  more available position (transverse edge $e_7$) when the update of
  $\gamma_{i-1}$ is performed starting from configuration $h^-$ rather
  than $h^+$. It is clear from Fig. \ref{fig:pyramide_billes} and
  \ref{calcul_drift_hexagone} that, if $f$ is at distance at least $r$, with $r$ sufficiently large,
  from the boundary of $W_L$, the edge $e_7$ is lower than
  $e^+_{i-1}$, so there is no obstruction for $b^-$ to actually reach
  it.
\end{itemize}

\end{proof}





\subsubsection{Lower bound on mixing time}

\begin{figure}[h]
   \centering
  \includegraphics[height=5cm]{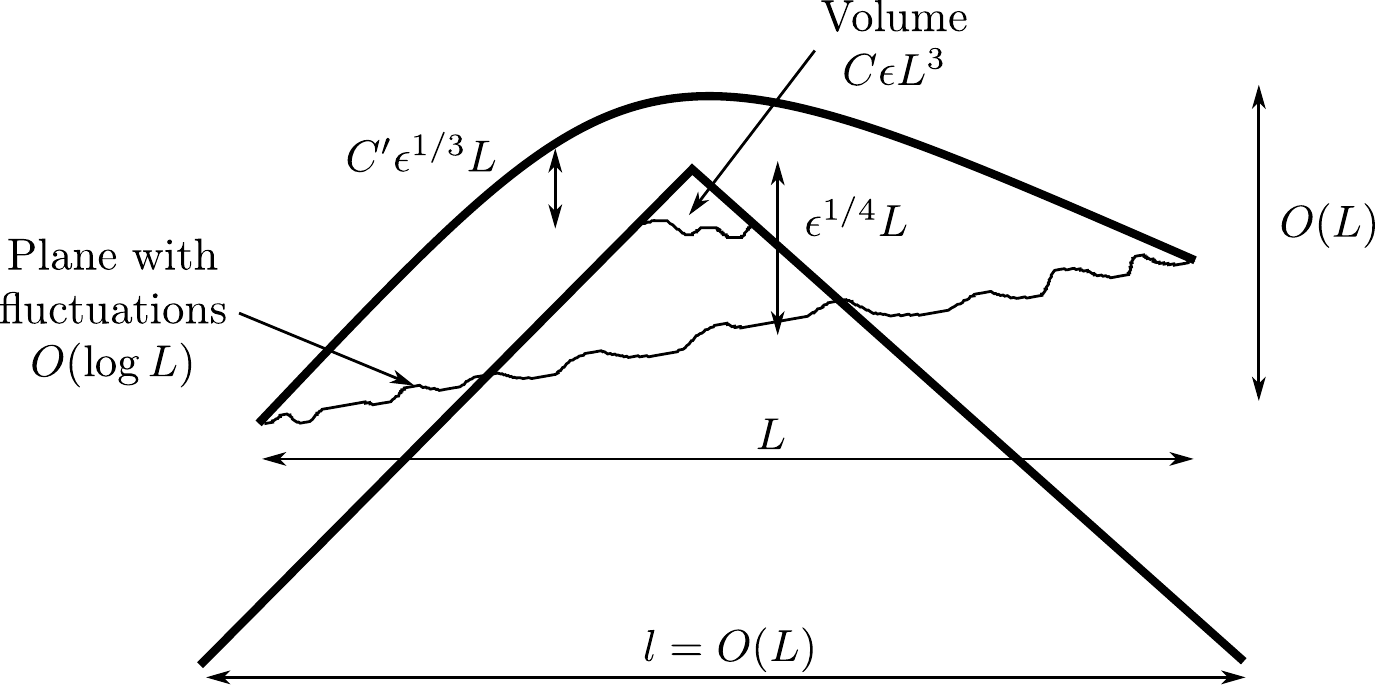}
   \caption{An illustration of the proof of the lower bound. The
     pyramid and, above it, the maximal configuration with planar
     boundary condition $m$. The top of the pyramid is $O(L)$ above
     the typical ``almost-planar'' equilibrium configuration (wiggled
     line). After a time $\epsilon L^2$, the pyramid has only lost a
     volume $\epsilon L^3$, so its top is still well above the
     ``almost-planar'' surface.  By stochastic domination, the same
     holds for the maximal evolution with boundary condition $m$.
   \label{fig:borne_inf}}
\end{figure}

Recall Definition \ref{def:G'}: the finite graph $G'$ we are
interested in is the portion of $G$ enclosed in $L U$, where $U$ is a
smooth bounded domain of $\bbR^2$, which without loss of generality we
can assume to include the origin in its interior (so that $f_0$ is at
distance of order $L$ from the boundary of $G'$). Now consider the
finite graph $W_\ell\subset G$ defined in Section \ref{sec_existence_pyramide},
and choose $\ell$ to be the minimal integer such that $W_\ell \supset G'$
(it is clear that $\ell=O(L)$).

Consider first the evolution in $W_\ell$ with boundary condition $p$
(and started from the maximal configuration $p$):
by Proposition \ref{prop:driftpiramide} and Markov's inequality
we have that at time $T_0=\epsilon L^2$,
\begin{eqnarray}
  \label{eq:vt}
  \mathbb P(h_{T_0}(f_0)\le -\epsilon^{1/4}L )=
O(\epsilon^{1/4})
\end{eqnarray}
 (indeed observe that, if $h(f_0)< -\epsilon^{1/4}L$
then the eroded volume is at least $const \times (\epsilon^{1/4} L)^3\simeq L T_0\epsilon^{-1/4} $ while the average eroded volume is of order $L T_0$).

Now recall that the boundary condition $m$ we are interested in is
almost-planar of slope $(s,t)$: up to an irrelevant global change of
the heights by an additive constant, we can assume that its height
function $h_m(\cdot)$ is within $O(1)$ from a plane of slope $(s,t)$ and
containing the point $(0,0,-2\epsilon^{1/4}L)$. If $h_p(\cdot)$ is the
height function associated to the ``pyramid'' matching $p$ and 
if $\epsilon$ is sufficiently small, $h_m(f)\ge h_p(f)$ for every $f$
along the boundary of $G'$ (the height difference is actually of order $L$). To see this, just recall that the slope of the
height function associated to $p$ is negative and maximal (in absolute
value) along any straight line starting from $f_0$, while the slope
$(s,t)$ is in the interior of the Newton polygon, so it is not an
extremal slope (see Figure \ref{fig:borne_inf}). Note that $\epsilon$ has to be taken very small if the slope $(s,t)$ is very close to $\partial N$.

We have then that, by monotonicity, the evolution in $G'$, with
boundary condition $m$ and started from the maximal configuration, is
stochastically higher than the restriction to $G'$ of the evolution in
$W_\ell$ with boundary condition $p$. As a consequence, for the former
dynamics we have that \eqref{eq:vt} still holds.
Thanks to Theorem \ref{th:stimeflutt}, at equilibrium (with almost-planar
boundary condition $m$) the typical equilibrium height at $f_0$ is around
$-2 L \epsilon^{1/4}$ and the probability that it is higher than
$-L\epsilon^{1/4}$ tends to zero with $L$. This suffices to conclude that
at time $T_0$ the variation distance from equilibrium is still close to $1$.
\qed

\appendix

\section{Gaussian behavior of fluctuations}

\label{sec:moments}

In this section we prove Theorem \ref{th:clt}. The proof is quite
independent of the rest of the paper. The main tools are the results
from \cite{KOS} that were stated in Section \ref{sec:phase_pure}. The
proof is very similar to that of \cite[Section 7]{K_non_planaire} but the setting
here is more general and we give a more explicit control of the
``error terms''.

  From the definition of height function, we know that $h(f)-h(g)$
  is determined by the dimers crossed by a path from $f$ to $g$. In
  particular, for the square, hexagon and square-hexagon graphs,
  assume that $f$ and $g$ are on the same thread (cf. Section
  \ref{sec:billes}): then, $h(f)-h(g)$ is determined just by the
  number $N_{f,g}$ of such dimers. Label $e_1,\dots,e_k$ the
  transverse edges between $f$ and $g$: it is easy to realize,
  using Theorem \ref{thm:proba_evn_fini}, that the set of occupied
  edges among the $e_i$ forms a determinantal point process: in other
  words, the probability $\mu_{s,t}( e_{i_1}\in M, \dots, e_{i_m}\in
  M)$ can be written as $\det(A(i_a,i_b)_{1\le a,b\le m})$ for a
  certain $k\times k$ matrix $A$ directly related to $K^{-1}_{s,t}$. Now, it turns
  out \cite{Kenyon_lectures} that for the hexagonal lattice such
  matrix is Hermitian. In this case, a well-known theorem by Costin
  and Lebowitz (cf. for instance \cite{Sosh}) implies that $N_{f,g}$
  is distributed like the sum of independent Bernoulli random
  variables, whose parameters are the eigenvalues of $A$. In
  particular, if the variance of $N_{f,g}$ diverges as $k\to\infty$,
  the variable $[N_{f,g}-\mu_{s,t}(N_{f,g})]/\sqrt{{\var}(N_{f,g})}$ tends to
  $\cN(0,1)$. Unfortunately, for the square and square-hexagon graph
  the matrix $A$ is not Hermitian and has complex eigenvalues, the
  Costin-Lebowitz theorem does not apply and the asymptotic moments
  have to be computed otherwise.

\subsection{Choice of paths}\label{section_presentation_chemin}

Fix an integer $k$. We are interested in the behavior of
\begin{equation}
  \mu_{s,t}\left[ M(f,f')^k\right]=\mu_{s,t}\left(\Bigl[ \bigl( h(f)-\mu_{s,t} (h(f)) \bigr) - \bigl( h(f') - \mu_{s,t} (h(f')) \bigr)\Bigr]^k\right)
\end{equation}
when the distance between $f$ and $f'$ goes to infinity. Remark that
for any path $C$ on $G^*$ from $f$ to $f'$ that only crosses edges in
the positive direction\footnote{exercise: prove that such a path exists for
every $f,f'$}, $h(f') - h(f)$ is exactly the difference between the
number of dimers crossed by the random matching minus those crossed by
the reference matching, so $M(f,f')=\sum_e[ 1_{e \in M} - \mu_{s,t} (e
\in M)]$ with the sum over all edges crossed by $C$. We will compute
moments of higher order by taking $k$ such paths $C_1, \ldots, C_k$ :
\begin{equation}\label{definition_M_kL}
\mu_{s,t} \left[  M(f,f')^k\right] =\mu_{s,t} \Bigl[ \sum_{e_1 \in C_1} \left(\delta_{e_1} - \mu_{s,t}( \delta_{e_1} )\right) \ldots \sum_{e_k \in C_k}\left(\delta_{e_k} - \mu_{s,t}( \delta_{e_k}) \right)  \Bigr]
\end{equation}
where we write $\delta_e$ for  $1_{e \in M} $.

\begin{figure}[h]
   \centering
   \includegraphics[width=8cm]{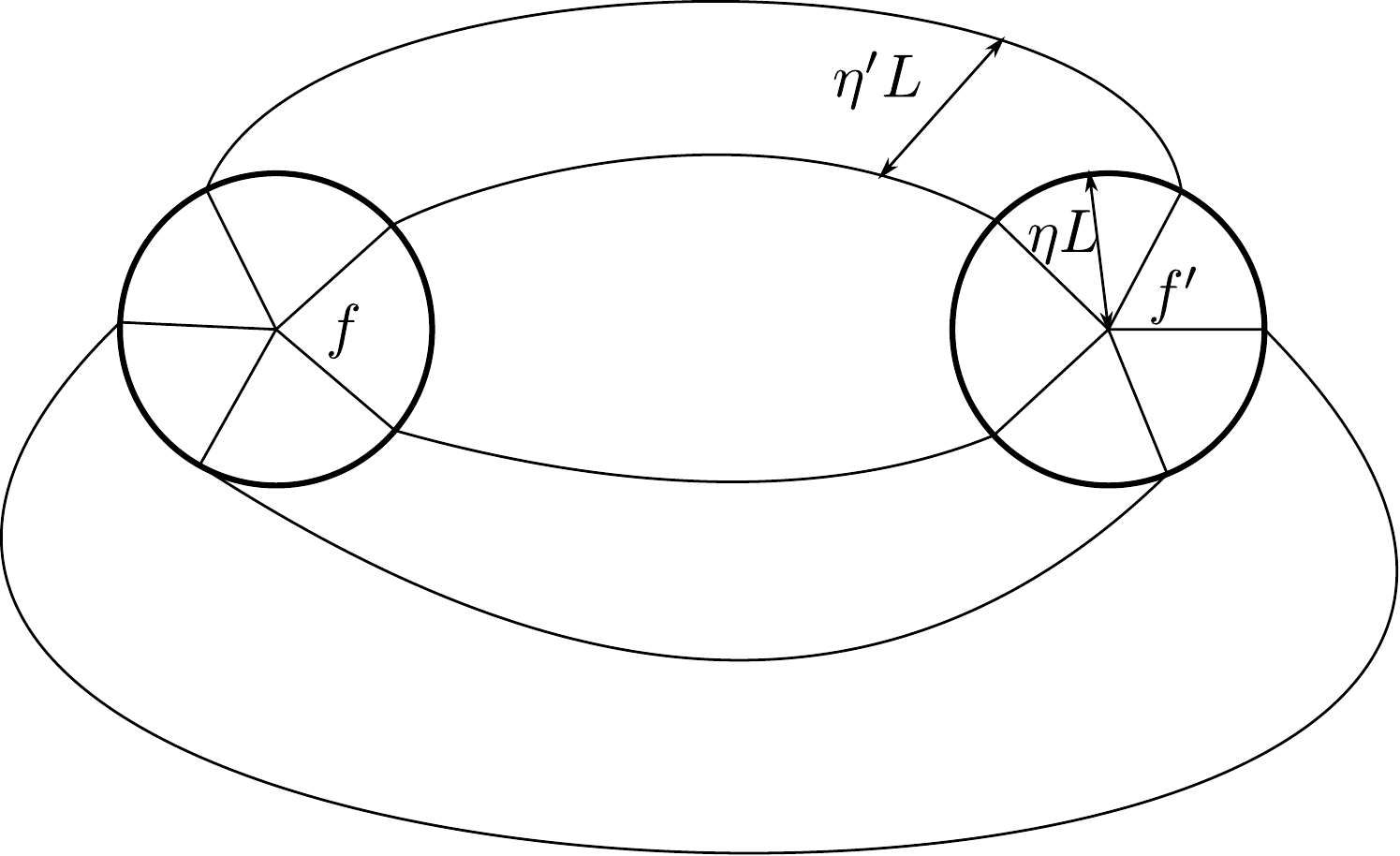}
    \caption{The paths $C_1,\dots,C_k$ (here $k=5$) along which is computed the height,
 displayed at large scale so that they look like continuous curves. They are macroscopically away from each
      other except near $f$ and $f'$ where they are periodic (so that
      on large scale they look linear). \label{fig:chemins_fluctuations}}
\end{figure}
In principle we can choose the paths $C_1 \ldots, C_k$ freely as long
as they only cross edges in the positive direction. In practice, for
reasons that will be clear later we
will adopt the following construction, illustrated
in Fig. \ref{fig:chemins_fluctuations}. 
Fix $\eta,\eta'>0$. Inside balls of radius $\eta L$ around $f$ and
$f'$, the $C_i$ are portions of length $\eta L$ of infinite periodic
paths (cf. Definition \ref{def:periodico}) and have mutually different
asymptotic directions. Outside of these balls the paths stay at
distance at least $\eta' L$ of each other and their length is of order
$L$. Furthermore $\eta, \eta'$ and the infinite periodic paths depend
on $k$ but not on $L$.



\subsection{Exact simplifications}

Fix $k$  edges $e_1=(\rw_1,\rb_1)\in C_1, \ldots
,e_k=(\rw_k,\rb_k)\in C_k$, and consider the corresponding term
(written $\Pi(e_1, \ldots, e_k)$) of the sum \eqref{definition_M_kL}.

We write $\delta_i$ for $\delta_{e_i}$, $K^{-1}_{i,j} $ for
$K_{s,t}^{-1}(\rb_i, \rw_j)$ and $K_{ij}$ for $K_{s,t}(\rw_i,\rb_j)$.  We also
write $S_k$ the set of permutations of $\{ 1, \ldots, k \}$, $\tilde
S_k$ the set of permutations without fixed points and $\tilde S_k^2$
those of order $2$. Given $\sigma\in S_k$, we denote $\epsilon(\sigma)$ it
signature.

First we express each term $\Pi(e_1, \ldots, e_k)$ of the sum \eqref{definition_M_kL} as a sum over permutations with no fixed points.
\begin{lemma}\label{lemme:somme_sans_pt_fixe}
   Given  edges $e_1, \ldots ,e_k$ as above, we have
\begin{equation}\label{eq_somme_sans_pt_fixe}
   \Pi(e_1, \ldots, e_k) \equiv \mu_{s,t}\bigl[ \prod_{n=1}^k \bigl(\delta_n - \mu_{s,t}(\delta_n)\bigr)\bigr] = \prod_{n=1}^k K_{nn} \sum_{\sigma \in\tilde{S}_k} \epsilon(\sigma) \prod_{n=1}^k  K^{-1}_{n \sigma(n)}.
\end{equation}
\end{lemma}
This is exactly like \cite[Lemma 21]{Kdom}.

Now we replace $K^{-1}_{n \sigma(n)}$ by its asymptotic expression
from Theorem \ref{th:asy}.  To lighten notations we use the following
conventions. Recall that $e_i=(\rw_i,\rb_i)$ and write (in a unique way)
$\rb_i=\hat\rb_i+(x_i,y_i),\rw_i=\hat\rw_i+(x'_i,y'_i)$ with $\hat\rb_i,\hat\rw_i$ in the 
fundamental domain $G_1$ and $(x_i,y_i)\in\bbZ^2,(x'_i,y'_i)\in\bbZ^2$. Note that
$(x_i-x'_i,y_i-y'_i)$ is a vector of order $1$. Then
let $U_i:=U_{s,t}(\hat\rb_i)$, $V_j:=V_{s,t}(\hat\rw_j)$ and 
$\phi_i:=\phi(x_i,y_i)$, $x_{ij} = x'_i-x_j$ and $y_{ij} = y'_i-y_j$. We
then get
\begin{multline}\label{eq_remplacement_K-1}
   \Pi(e_1,\dots,e_k)/\Bigl( \prod_{n=1}^k K_{nn} \Bigr) = \sum_{\sigma \in \tilde{S}_k} \epsilon(\sigma) \prod_{n=1}^k  \left(i\frac{U_n V_{\sigma(n)} w_0^{x_{n\sigma(n)}} z_0^{y_{n\sigma(n)}}}{2\pi(\phi_n-\phi_{\sigma(n)})} + \text{complex conjugate}\right.\\
 \left.+ O\Big(\frac{1}{x_{n\sigma(n)}^2+y_{n\sigma(n)}^2}\Big) \right)\equiv \sum_{\sigma \in \tilde{S}_k} \epsilon(\sigma) \prod_{n=1}^k (A_{n,\sigma(n)}+A^*_{n,\sigma(n)}+R_{n,\sigma(n)}).
\end{multline}
We then expand the product over $n$ and we call ``dominant terms''
those such that no $R$ factors appear and, in addition, such that within each cycle
of $\sigma$ either only $A$ terms or only $A^*$ terms appear.
All other terms of the expansion are called ``error terms''.
We first show that in the dominant terms we can assume that all cycles
of the permutation $\sigma$ are of order two. This comes from the
following result:
\begin{lemma}
\label{lemm:1ciclo}
If $\tilde S^c_\ell$ denotes the set of cyclic permutations of
$\{1,\dots,\ell\}$, then
   \begin{equation}\label{eq:ordre_2}
  \sum_{\sigma \in \tilde S^c_\ell}\epsilon(\sigma)\prod_{n=1}^\ell A_{n\sigma(n)}=     \sum_{\sigma \in \tilde S^c_\ell}\epsilon(\sigma)\prod_{n=1}^\ell \frac
      i{2\pi} \frac{U_n V_{\sigma(n)} w_0^{x_{n\sigma(n)}}
        z_0^{y_{n\sigma(n)}}}{\phi_n-\phi_{\sigma(n)}} = 0\left(= \sum_{\sigma \in \tilde S^c_\ell}\epsilon(\sigma)\prod_{n=1}^\ell A^*_{n\sigma(n)}\right)
   \end{equation}
if $\ell>2$.
\end{lemma}
\begin{proof}
Recall
that $\epsilon(\sigma)$ is constant for cyclic permutations, 
thus it is enough to show that
\begin{equation}
  \label{eq:1ciclo}
   \sum_{\sigma \in \tilde S^c_\ell} \prod_n \frac{U_nV_{\sigma(n)} w_0^{x_{n\sigma(n)}} z_0^{y_{n\sigma(n)}}}{\phi_n-\phi_{\sigma(n)}}=\prod_n \left(U_n V_n\right)
w_0^{\sum_i (x'_i-x_i)}z_0^{\sum_i (y'_i-y_i)}
\sum_{\sigma \in \tilde{S}_\ell^c} \prod_n
\frac1{\phi_n-\phi_{\sigma(n)}}=0
\end{equation}
as soon as $\ell>2$.
The second equality is purely algebraic and is given in \cite[Lemma 7.3]{K_non_planaire}. 
\end{proof}
Decompose a permutation $\sigma$ into cycles and remark that in equation \eqref{eq_somme_sans_pt_fixe} and \eqref{eq_remplacement_K-1}, the sum over permutations with a given cycle structure can be factorized as a product over cycles.
Then, Lemma \ref{lemm:1ciclo} implies that the dominant terms in
\[
\sum_{\sigma \in \tilde{S}_k\setminus \tilde S^2_k} \epsilon(\sigma) \prod_{n=1}^k (A_{n,\sigma(n)}+A^*_{n,\sigma(n)}+R_{n,\sigma(n)})
\]
exactly cancel each other so only ``error terms'' are left (recall that $\tilde S^2_k$ is the set
of permutations without fixed points, and with only cycles of order $2$).
Altogether, we have proven (cf. \eqref{eq_remplacement_K-1})
\begin{equation}
\label{eq:primaexp}
  \mu_{s,t}(    M(f,f')^k) = \sum_{e_1 \in C_1} \cdots \sum_{e_k \in C_k} \sum_{\sigma \in\tilde{S}_k^2} \epsilon(\sigma) \left(\prod_{n=1}^k K^{-1}_{n \sigma(n)}\right)
\left( \prod_{n=1}^k K_{nn} \right) + \text{error terms}.
\end{equation}
In particular, if $k$ is odd then there are only ``error terms''
because $\tilde S^2_k$ is empty, while, using equation
\eqref{eq_somme_sans_pt_fixe} separately for all pairs,
\begin{align}\label{gaussien_plus_termes_suplementaires}
   \mu_{s,t}(M(f,f')^{2k}) & = \sum_{\sigma \in\tilde{S}_{2k}^2} \mu_{s,t}[M(f,f')^2]^k + \text{error terms} \\
	      & = g_{2k}\left[\var_{\mu_{s,t}}(h(f)-h(f'))\right]^k + \text{error terms}
\end{align}
where $g_{2k} =  (2k)!/(2^k k!)$ is the Gaussian moment of
order $2k$. We will prove in the following section that the error terms are negligible, i.e. give a contribution to $\mu_{s,t}(M(f,f')^k)$ that
is much smaller than $[\var_{\mu_{s,t}}(h(f)-h(f'))]^{k/2}$.

\subsection{Controlling the ``error terms''}
\label{sec:et}
Recall that there are two kinds of ``error terms'' in the expression \eqref{eq:primaexp}
for the $k$-th moment of the height fluctuation: those which contain
both $A$ and $A^*$ factors within the same cycle of $\sigma$ but but no $R$ factor (recall \eqref{eq_remplacement_K-1}), and those which contain
$R$ factors. The former will be shown to be ``small'' because they
oscillate and give a negligible contribution when summed over $e_1,
\dots,e_k$, while the latter are small because their denominator
contains at least one factor $|\phi_j-\phi_{\sigma(j)}|$ more than in
the dominant terms. For simplicity purpose we assume in the following that $\sigma$ contains only a single cycle. The general case is essentially the same since an error term can be factorized as a product over the cycles of $\sigma$.

We will need the following elementary estimates:
\begin{lemma}
\label{lem:F}
  Define
  \begin{eqnarray}
    \label{eq:Fk}
    F_k(L)=\sum_{d_1,\dots, d_k\le
      L}\frac1{(d_1+d_2)\dots(d_k+d_1)},\quad
  \tilde F_k(L)=\sum_{d_1,\dots,d_k\le L}\frac1{(d_1+d_2)\dots(d_{k-1}+d_k)}.
  \end{eqnarray}
One has
$\tilde F_k(L)\le L F_{k-1}(L)   $
and
$F_k(L)=O((\log L)^{\lfloor k/2\rfloor}).  $
\end{lemma}
\begin{proof}
  To get $\tilde F_k(L)\le L F_{k-1}(L)   $, writing $F_{k-1}(L)$ either as a sum over $d_1,\dots,
d_{k-1}$ or over $d_2,\dots,d_k$ it is not difficult to see that
 \[
    2L F_{k-1}(L)-2 \tilde F_k(L)=\sum_{d_1,\dots,d_k}\frac1{(d_1+d_2)\dots(d_{k-2}+d_{k-1})}\,\frac{(d_1-d_k)^2}{(d_1+d_{k-1})(d_2+d_k)(d_{k-1}+d_k)}\ge0
  \]
  (use the fact that the denominator is symmetric under the exchange
  of $d_1$ with $d_k$).  As for $F_k(L)=O((\log L)^{\lfloor
    k/2\rfloor})$, this is an easy computation if $k=2$.  By induction
  on $k$, the proof is concluded if we show that $F_k(2L)-F_k(L)=O(k
  F_{k-2}(L))$ for $k\ge 4$ and $F_k(2L)-F_k(L)=O(1)$ for $k=3$.  To see this, one notes that the dominant contribution
  to $F_k(2L)-F_k(L)$ comes from the terms where one of the variables
  is in $[L,2L]$ while all the others are smaller than $L$:
  altogether, this gives
\begin{align*}
   \sum_{i=1}^k \sum_{d_1,\ldots d_k \leq L} \frac{1}{(d_1+d_2)\ldots(d_{i-1}+(L+d_i))((L+d_i)+d_{i+1})\ldots(d_k+d_1)}
      &\leq k L \tfrac{1}{L^2}\tilde F_{k-1}(L).
\end{align*}
For $k>3$ use $\tilde F_{k-1}(L)\le L F_{k-2}(L)$ and for $k=3$ note that
$\tilde F_{2}(L)=O(L)$.
\end{proof}

The first remark about the computation of $\mu_{s,t}(M(f,f')^k)$ is that we can restrict ourselves to the cases where
either all edges $e_i$ are in the ball $B_f(\eta L)$ of radius $\eta L$ around
either $f$ or in the analogous ball around $f'$. Indeed,  call $d_i$ the
distance of $e_i$ from $f$ and observe that, since the map $\phi$ is
non-degenerate and the paths $C_i$ are almost linear in
$B_f(\eta L)$, one can bound $|\phi_i-\phi_j|$, from above and below, by
a constant times $(d_i+d_j)$. Then the sum of \eqref{eq_remplacement_K-1}
with, for example, $e_1,\dots,
e_{k-1}$ in $B_f(\eta L)$ and  $e_k$ out of $B_f(\eta L)$ is of order
$(1/L)\tilde F_{k-1}(\eta L)\le F_{k-2}(\eta L)=O((\log L)^{\lfloor k/2\rfloor
-1})\ll [\var_{\mu_{s,t}}(h(f)-h(f'))]^{k/2}$. Terms with more than one $e_i$ outside of $B_f(\eta L)$ are even smaller. To fix ideas, we will
assume that all $e_i$ are in $B_f(\eta L)$.

Consider now an ``error term'' containing a factor $R$ (if it contains
more than one, the argument is similar). Summing over $e_i$, this
gives a contribution of order
\begin{gather}\label{forme_simple_terme_erreur}
  \sum_{d_1,\dots,d_k \leq \eta L}
  \frac{1}{(d_1+d_2)\ldots(d_k+d_1)}\frac{1}{d_1+d_2}
   \leq \sum_{d_2,\ldots, d_k} \frac{1}{(d_2+d_3)\ldots(d_{k-1}+d_{k})} \sum_{u=d_2}^\infty \frac{1}{d_k u^2} \\\nonumber
   \leq  const\times \sum_{d_2, \ldots, d_k}
   \frac{1}{(d_2+d_3)\ldots(d_k+d_2)}
   \frac{d_2+d_k}{d_2d_k}=O(F_{k-1}(\eta L))\ll
   [\var_{\mu_{s,t}}(h(f)-h(f'))]^{k/2}.
\end{gather}

We still have to deal with the error terms including no $R$ factors but
both $A$ and $A^*$ within the same cycle of $\sigma$. Recall that for
simplicity of exposition that $\sigma$ is assumed to have a single cycle. Omitting the product of the factors $\pm i/(2\pi)$, these terms 
are of the form
\begin{equation}\label{forme_cpx_terme_oscillant}
  \sum_{e_1,\dots,e_k} e^{ i [r(x_1, \ldots, x_k)+  s(y_1,\ldots,y_k)]} C(e_1, \ldots, e_k) \prod_{i\in J} \frac{1}{\phi_i - \phi_{\sigma(i)}} \prod_{i \in J^c} \frac{1}{{\phi^*}_i - {\phi^*}_{\sigma(i)}}
\end{equation}
where $e_j$ runs over the edges crossed by path $C_j$ and:
\begin{itemize}
\item $J$ (resp. $J^c$) is the set of indices $n\in \{1,\dots,k\}$ for which we take
$A_{n\sigma(n)}$ (resp. $A^*_{n\sigma(n)}$) in the expansion of the
product \eqref{eq_remplacement_K-1}. Note that both $J$ and $J^c$ are
non-empty, proper subsets of $\{1,\dots,k\}$;
\item  $C(e_1,\dots,e_k)$ depends only on
the types of the $k$ edges (two edges being of the same type if they are
related by a translation of $\bbZ^2$):
\[
C(e_1,\dots,e_k)=\prod_{i\in J}\left(U_i V_{\sigma(i)} w_0^{(x'_i-x_i)}z_0^{(y'_i-y_i)}\right)
\prod_{i\in J^c}\left(U^*_i V^*_{\sigma(i)}w_0^{-(x'_i-x_i)}z_0^{-(y'_i-y_i)}\right);
\]

\item 
 $r,s$ are linear functions: if $J'=\{i\in J:\sigma^{-1}(i)\in J^c\}$ and $J''=\{i\in J^c:\sigma^{-1}(i)\in J\}$ (remark that $|J'|=|J''|\ne0$ otherwise $\sigma$ would have more than one cycle)
and writing $w_0=\exp(i\theta_w),z_0=\exp(i\theta_z)$, 
\[
r(x_1,\dots,x_k)+s(y_1,\dots,y_k)=2\theta_w(\sum_{a\in J'}x_a-\sum_{a\in J''}x_a)+2\theta_z(\sum_{a\in J'}y_a-\sum_{a\in J''}y_a).
\]

\end{itemize}
Splitting the sum $ \sum_{e_1,\dots,e_k}$ over the different types of edges, we
can assume without loss of generality that all edges in each path are
of the same type (so that $C(e_1,\dots,e_k)$ becomes a constant; edge types can be different in
different paths) and that edges in the  path $j$ are obtained one from the other via translations by an integer multiple of 
some $v^{(j)}=(v^{(j)}_1,v^{(j)}_2)\in\bbZ^2$. The  edges in the $j$-th path will be labeled by an
integer $d_j$, which runs from $1$ to $M_j=O(\eta L)$ and  $r(x_1,
\ldots, x_k)+  s(y_1,\ldots,y_k)$ becomes is a linear function of $d_1,\dots,d_k$.

Assume without loss of generality that $1\in J'$. From the discussion above we obtain that 
\[
(r+s)(d_1,\dots,d_k)=(r+s)(0,d_2,\dots,d_k)+2(\theta_w v^{(1)}_1+\theta_z v^{(1)}_2)d_1.
\]
Despite the fact that $w_0,z_0$ are unit complex numbers with phase different from $0,\pi$
it could happen that $\Theta:=2(\theta_w v^{(1)}_1+\theta_z v^{(1)}_2)$ is a multiple of $2\pi$, so that $\exp(i(r+s))$ is independent of $d_1$: this can however be avoided if
the asymptotic directions and the period of  path $C_1$ in $B_f(\eta L)$  are chosen suitably (we skip
tedious details on this point). 

We separate the sum over $e_1$ from the others and we make a summation
by parts to get:
\begin{gather}
\label{strazio}
  const\times \sum_{e_2,\ldots, e_k} \frac{ e^{i (r+s)|_{d_1=0}} }{(\tilde{\phi}_2 -
    \tilde{\phi}_{3}) \ldots (\tilde{\phi}_{k-1}-\tilde{\phi}_k) } \\
\nonumber\times\left[  \sum_{d_1} \Bigl( \sum_{d \leq d_1} e^{i \Theta d} \Bigr)
  \frac{\Delta \tilde{\phi}_1 ( \tilde\phi_k + \tilde\phi_2
    -2\tilde\phi_1 (d_1)-\Delta \tilde\phi_1 ) }
  {\bigl(\tilde\phi_1(d_1)-\tilde\phi_2\bigr)
\bigl(\tilde\phi_k-\tilde\phi_1(d_1)\bigr)
\bigl(\tilde\phi_1(d_1+1)-\tilde\phi_2\bigr)
\bigl(\tilde\phi_k-\tilde\phi_1(d_1+1)\bigr)}
  + O\left(\frac{1}{d_2d_k}\right) \right]
\end{gather}
where  the $O(...)$ comes from boundary terms in the
summation by parts (note that $\sum_{d \leq d_1} e^{i \Theta
  d}$ is bounded since $\Theta \neq 0 $(mod $2\pi)$), $\Delta\tilde\phi_1 =
\tilde\phi_1(d_1+1)-\tilde\phi_1(d_1)$ is constant by linearity of $\phi$
and for simplicity of notation $\tilde \phi_i$ can denote either
$\phi_i$ or its complex conjugate. We can thus bound
$\abs{\phi_i - \phi_j}$ by $d_i+d_j$ as before and (up to a constant factor) the absolute
value of \eqref{strazio} by 
\begin{gather}
    \sum_{d_2,\ldots,d_k} \frac{1}{(d_2+d_3)\ldots(d_{k-1}+d_k)}\Bigl( \sum_{d_1} \frac{2d_1+d_2+d_k}{(d_1+d_2)^2(d_1+d_k)^2} +O\Bigr(\frac{1}{d_2d_k}\Bigr) \Bigr) \\
	=\sum_{d_2,\ldots,d_k} \frac{1}{(d_2+d_3)\ldots(d_{k-1}+d_k)(d_2+d_k)}\Bigl(\sum_{d_1} \frac{(2d_1+d_2+d_k)(d_2+d_k)}{(d_1+d_2)^2(d_1+d_k)^2} +O\left(\frac{d_2+d_k}{d_2d_k}\right) \Bigr) \\
      = O\bigl(F_{k-1}(\eta L) \bigr) \ll
   [\var_{\mu_{s,t}}(h(f)-h(f'))]^{k/2}
\end{gather}
because the parenthesis in the second line is clearly bounded.





\section*{Acknowledgments}
F. L. T. acknowledges the support of European Research Council through the
“Advanced Grant” PTRELSS 228032 and of Agence Nationale de la Recherche through grant ANR-2010-BLAN-0108.

\end{document}